\let\ORIlabel\label
\let\ORIrefstepcounter\refstepcounter
\AddToHook{package/hyperref/before}{
	\let\label\ORIlabel 
	\let\refstepcounter\ORIrefstepcounter}
\documentclass[onefignum,onetabnum]{siamart190516}

\usepackage{amsmath,amssymb,amsfonts,latexsym,stmaryrd}
\usepackage{graphicx,float,epsfig} 
\usepackage{mathrsfs} 
\usepackage{color}
\usepackage{setspace} 
\usepackage{lipsum}
\usepackage{graphicx,float}
\usepackage{color}
\usepackage{epstopdf}
\usepackage{multirow}
\usepackage{svg}
\usepackage{url}
\usepackage{diagbox}
\ifpdf
  \DeclareGraphicsExtensions{.eps,.pdf,.png,.jpg}
\else
  \DeclareGraphicsExtensions{.eps}
\fi
\newcommand\err{\texttt{err}}
\newcommand\eff{\texttt{eff}}

\def\O{\Omega}

\def\l{\lambda}

\renewcommand\sp{\mathop{\mathrm{Sp}}\nolimits}

\newcommand{\jump}[1]{\llbracket #1 \rrbracket}

\newcommand{\norm}[1]{\lVert#1\rVert}

\usepackage{booktabs}
\usepackage{float}

\newcommand\bu{\boldsymbol{u}}
\newcommand\bv{\boldsymbol{v}}

\newcommand\bn{\boldsymbol{n}}

\def\hdel{\widehat{\delta}}

\newcommand\bF{\boldsymbol{f}}
\newcommand\bbK{\mathbb{K}}

\newcommand\bT{\boldsymbol{T}}

\def\CT{{\mathcal T}}

\newcommand\btau{\boldsymbol{\tau}}

\newcommand\bphi{\boldsymbol{\varphi}}

\renewcommand\H{\mathrm{H}}

\renewcommand\L{\mathrm{L}}

\renewcommand\O{\Omega}

\newcommand{\vertiii}[1]{{\left\vert\kern-0.25ex\left\vert\kern-0.25ex\left\vert #1 
		\right\vert\kern-0.25ex\right\vert\kern-0.25ex\right\vert}}

\renewcommand\div{\mathop{\mathrm{div}}\nolimits}

\renewcommand\sp{\mathop{\mathrm{sp}}\nolimits}

\newsiamremark{remark}{Remark}
\newsiamremark{hypothesis}{Hypothesis}
\crefname{hypothesis}{Hypothesis}{Hypotheses}
\newsiamthm{problem}{Problem}
\newsiamthm{claim}{Claim}

\headers{Stokes-Brinkman eigenvalue problem}{Felipe Lepe,  Gonzalo Rivera and Jesus Vellojin}

\title{A Stokes-Brinkman-type formulation for the eigenvalue problem in porous media\thanks{Submitted to the editors DATE.
\funding{The second author was partially supported by ANID-Chile through FONDECYT project 1231619.  The third author was partially supported by ANID-Chile through FONDECYT Postdoctorado project 3230302.
}}}

\author{Felipe Lepe\thanks{GIMNAP-Departamento de Matem\'atica, Universidad del B\'io - B\'io, Casilla 5-C, Concepci\'on, Chile. \email{flepe@ubiobio.cl}.}
\and Gonzalo Rivera\thanks{Departamento de Ciencias Exactas,
	Universidad de Los Lagos, Casilla 933, Osorno, Chile. \email{gonzalo.rivera@ulagos.cl}.}
\and Jesus Vellojin\thanks{GIMNAP-Departamento de Matem\'atica, Universidad del B\'io - B\'io, Casilla 5-C, Concepci\'on, Chile. \email{jvellojin@ubiobio.cl}.}}

\usepackage{amsopn}

\ifpdf
\hypersetup{
  pdftitle={The Stokes-Brinkman eigenvalue problem},
  pdfauthor={Felipe Lepe, Gonzalo Rivera and Jesus Vellojin}
}
\fi

\def\CT{{\mathcal T}}
\def\CE{\mathcal{E}}
\newcommand\eu{\texttt{e}_{\boldsymbol{u}}}
\newcommand\ep{\texttt{e}_{p}}
\begin{document}

\maketitle

\begin{abstract}
In this paper we introduce and analyze, for two and three dimensions, a finite element method to approximate the natural frequencies of a flow system 
governed by the Stokes-Brinkman equations. Here, the fluid presents the capability of being within a porous media. Taking advantage of the Stokes regularity results for the solution, and considering inf-sup stable families of finite elements, we prove convergence together with a priori and a posteriori error estimates for the eigenvalues and eigenfunctions with the aid of the compact operators theory. We report a series of numerical tests in order to confirm the developed theory.
\end{abstract}

\begin{keywords}
  porous media, fluid equations, eigenvalue problems,  a priori error estimates, a posteriori error analysis, Stokes-Brinkman equations
\end{keywords}

\begin{AMS}
  35Q35,  65N15, 65N25, 65N30, 65N50
\end{AMS}

\section{Introduction}\label{sec:intro}
Eigenvalue problems are attractive from the mathematical and computational point of view, where the applications of such  problems 
can be found in several fields, as structural engineering,  electromagnetism, flow stability, among others. The analysis of eigenvalue problems is a well stated issue, where a survey of the developed theory on this matter can be found in, for instance, in \cite{MR2652780}. Particularly, the analysis of numerical methods  flow eigenvalue problems has an important amount of contributions
available on the literature, such as \cite{MR4071826,MR4077220,Lepe2021,LRVSISC,MR4728079,MR3095260,MR2473688}

The Brinkman system is relevant in continuum mechanics and fluid dynamics, particularly in the study of porous media flow. This problem addresses the flow behavior of a fluid in a medium where viscous and inertial effects are both relevant, extending Darcy's law (which describes flow in porous media at low velocities) to regimes where these effects play a more significant role. The Brinkman problem applies to a range of physical scenarios, including fluid flow in biological tissues, industrial filtration processes, and environmental engineering. In porous media, high-velocity flows induce viscous and sometimes inertial forces, leading to deviations from simple Darcy flow. The Brinkman equation allows for modeling these effects with an added diffusion term.  We refer to \cite{anaya2016priori,Gatica2018,hannukainen2011computations,juntunen2010analysis,konno2011h,mardal2002robust} and the references therein for a more deep analysis on different techniques and numerical schemes to study this problem.

As stated in \cite{konno2011h}, the Brinkman equations are a parameter-dependent combination of the Darcy and Stokes equations. It acts as a coupling layer between a free-surface flow  and a porous Darcy flow. One of the main difficulties arising from this coupling is that typical numerical schemes fails to predict the behavior when we move from the Stokes limit to Darcy. The standard approach to modeling such problems involves coupling Darcy and Stokes flows and enforcing the Beavers–Joseph–Saffman conditions along the interface \cite{beavers1967boundary,cesmelioglu2024strongly,urquiza2008coupling}. However, with an appropriate scaling of the model, it is possible to study a complete system of equations that predicts fluid behavior in a domain with regions of varying permeability. The Stokes–Brinkman equations, which reduce to either Stokes or Darcy flow depending on the coefficients, were proposed in \cite{popov2007} as a replacement for coupled Darcy and Stokes models. A proper selection of this coefficients allows to represent both free-flow and porous regions, thus avoiding any further assumptions on the interface condition. An example of this approach is provided in \cite{williamson2019posteriori}, where the authors present a residual-based a posteriori error analysis for the Stokes–Brinkman problem discretized with Taylor–Hood finite elements, using this analysis to guide an adaptive mesh refinement process.

The aforementioned study serves as the foundation for the present paper, whose primary goal is to propose an eigenvalue formulation for the Stokes-Brinkman problem. This formulation is based on defining a specific permeability parameter that characterizes both porous and fully permeable regions of some domain. Given the increased resistance to fluid flow within Darcy-type domains, the fundamental modes of the Stokes-Brinkman system can aid in studying energy minimization problems, as the fluid naturally favors regions with lower resistance. Furthermore, by selecting sufficiently low permeability values, certain zones can be treated as obstacles. This particular flexibility in the choice of permeability parameters introduces an additional tool for modeling eigenmodes in domains with multiple porous subdomains, thereby simulating behaviors analogous to domains with internal membranes. Potential applications include models for filtration, subsurface water treatment, and similar environmental and engineering processes.
 
The purpose of this paper is to analyze a primal velocity-pressure formulation for the Stokes-Brinkman eigenvalue model, which we analyze using the compact operators theory. For this type of formulation there exists well know  inf-sup stable  finite element families to approximate the velocity and the pressure. 
We choose mini elements and Taylor-Hood to handle this Stokes-Brinkman eigenvalue problem because they are the most common alternatives to approximate numerically the spectrum of such system. For these numerical schemes, convergence and  a  priori  error estimates  are derived with the aid of the compact operators theory presented in  \cite{MR1115235,MR2652780}. A residual-based a posteriori error analysis is also provided, which under standard techniques, is proved to be reliable and efficient.  Also, we perform a series of numerical experiments to illustrate the eigenvalue problem, where the intention is, as a first step, to conclude that the families of finite elements of our scheme does not allow to introduce spurious eigenvalues on the computational spectrum.  In both two and three dimensions, we examine how varying the permeability parameter influences the spectrum and its regularity, particularly through the formation of regions with high pressure gradients. Furthermore, we provide a physical interpretation of the results, analyzing the shapes of the eigenfunctions for selected permeability values.
 
 The paper is organized in the following manner:  In Section \ref{sec:model_problem} we propose the eigenvalue problem for the Stokes-Brinkman equations. The well-posedness and spectral characterization of the model is derived by means of a general inf-sup condition. Section \ref{sec:fem} deals with the discrete problem, for which existence and uniqueness are established. In addition, the convergence of solution operators and the a priori error estimates for the eigenvalues and eigenfunctions are obtained. Section \ref{sec:apost} provides the a posteriori analysis for the problem, which relies on the inf-sup condition and the Scott-Zhang interpolant. Finally, Section \ref{sec:numerics} is devoted to report several numerical examples to test the convergence and performance of the numerical scheme, and also to study behavior of the spectrum when different permeability scenarios are considered.

\subsection{Notation}\label{sec:notation}
Let $\O$ be a subset of $\mathbb{R}^d$, for  $d\in\{2,3\}$ and  Lipschitz boundary $\partial\O$.   For $t \geq 0$ and $p \in[1, \infty]$, we denote by $\mathrm{L}^p(\O)$ the usual Lebesgue space of maps from $\O$ to $\mathbb{R}$ endowed with the norm $\|\cdot\|_{\mathrm{L}^p(\O)}$, while $\H^{t}(\O)$ denotes a Sobolev space. Vector spaces and vector-valued functions will be written in bold letters. For instance, for $t\geq 0$, we simply write $\boldsymbol{\H}^{t}(\O)$ instead of $\H^{t}(\O)^d$.  As usual, we write $|\cdot|_{t, \O}$ and $\Vert \cdot\Vert_{t,\Omega}$ to denote the seminorm and norm respectively. We define $\boldsymbol{\H}_0^1(\O)$ as the vector space of functions in $\boldsymbol{\H}^1(\O)$ with vanishing trace on $\partial \O$, and $\mathrm{L}_0^2(\O)$ as the space of $\mathrm{L}^2(\O)$ functions with vanishing mean value over $\O$. Finally, $\mathbf{0}$ denotes a generic null vector or tensor.

\section{The model problem}
\label{sec:model_problem}

Let  $\O\subset\mathbb{R}^d$ be an open bounded domain with Lipschitz boundary $\partial\O$. Let $\Gamma_1$ and $\Gamma_2$ be disjoint open subset of $\partial\O$ such that $\partial \O=\overline{\Gamma}_1\cup\overline{\Gamma}_2$. Considering a steady-state of the balance laws for linear fluid equations, the problem to be studied is given as follows (see for example \cite{williamson2019posteriori})
\begin{align}
	\bbK^{-1}\bu - \nu \Delta \bu + \nabla p & = \lambda \bu & \text{in $\Omega$},\label{eq:Brinkman1}\\
	\div\bu &= 0 & \text{in $\Omega$},\label{eq:Brinkman2}\\
	\bu &= \boldsymbol{0}& \text{on $\Gamma_1$}, \label{bc:Sigma} \\
	(\nu\nabla \bu - p\mathbb{I}) \bn   &=\boldsymbol{0} & \text{on $\Gamma_2$} \label{bc:Gamma}.
\end{align}

Here, $\bu$ is the fluid velocity, $p$ is the pressure of the fluid, $\mathbb{I}$ the identity matrix in $\mathbb{R}^{n\times n}$, and the parameter $\nu$ is the viscosity of the fluid. The vector $\bn$ represents the unit normal. We also consider $\mathbb{K}$ as a bounded, symmetric, and positive definite tensor describing the permeability properties of the domain.  If $\mathbb{K}^{-1}\rightarrow \boldsymbol{0}$, the equation is reduced to Stokes, while if $\nu=0$, we have Darcy's law. In this study, we split the domain in two media $\overline{\Omega}:=\overline{\Omega}_S\cup\overline{\Omega}_D$,  where $\Omega_S$ and $\Omega_D$ represent subdomains where there is free flow and porous media, respectively and assume  $\nu>0$. For the free-flow domain, we take $\mathbb{K}^{-1}=\boldsymbol{0}$, while $\mathbb{K}^{-1}\gg\boldsymbol{0}$ is considered in $\Omega_D$.  This allows to study the eigenmodes on domains where we have membrane-like behavior or internal filters.

To derive the weak form of the governing equations, we first define the appropriate functional spaces for the velocity and pressure. The space for the velocity is given by
$$
	\boldsymbol{\H}_{\Gamma_1}(\Omega):= \{ \bv \in \boldsymbol{\H}^1(\Omega): \bv = \boldsymbol{0} \text{ on } \Gamma_1\},
$$
whereas $\mathrm{L}^2(\Omega)$ is the space for the pressure. We note that thanks to the natural boundary condition on $\Gamma_2$,  the term $\bu\cdot\bn$  automatically adjusts itself on the outflow boundary to ensure that mass is conserved (the volume of flow entering is the same as the volume of flow exiting the domain), hence $p$ is uniquely defined (see for instance \cite[Chapter 5]{elman2014finite}).  On the other hand, if $\Gamma_2 = \emptyset$, then the pressure is defined up to a constant. For this case, we take $	\boldsymbol{\H}_{\Gamma_1}(\Omega):=\boldsymbol{\H}_0^1(\Omega)$ for the velocity space, and $\L_0^2(\Omega)$ for the pressure.

Throughout this work, we assume that the permeability tensor is positive definite for all $\bv\in\boldsymbol{\H}_{\Gamma_1}(\Omega)$. More precisely, there exist positive constants $\mathbb{K} _*,\mathbb{K}^*>0$ such that
\begin{equation*}
	\label{eq:kappa-bounds}
	0<\mathbb{K}_*\Vert\bv\Vert_{0,\O}^2 \leq (\mathbb{K}^{-1}\bv,\bv)_{0,\Omega}\leq \mathbb{K}^*\Vert\bv\Vert_{0,\O}^2.
\end{equation*}

By testing equations \eqref{eq:Brinkman1} with $\bv\in\boldsymbol{\H}_{\Gamma_1}(\Omega)$ and \eqref{eq:Brinkman2} with $q\in \mathrm{L}^2(\Omega)$, integrating by parts, and using the boundary conditions \eqref{bc:Sigma}-\eqref{bc:Gamma} we obtain the following variational problem: Find $\lambda\in\mathbb{R}$ and $(\boldsymbol{0},0)\neq(\bu,p)\in\boldsymbol{\H}_{\Gamma_1}(\Omega)\times \mathrm{L}^2(\Omega)$ such that
\begin{align*}
	\int_{\Omega} \mathbb{K}^{-1} \bu \cdot \bv + \nu \int_{\Omega} \nabla \bu :\nabla \bv - \int_{\Omega} p \div \bv & = \lambda \int_{\Omega} \bu \cdot \bv, &\forall\bv\in\boldsymbol{\H}_{\Gamma_1}(\Omega), \\
	- \int_{\Omega} \div \bu\,  q &= 0,  &\forall q\in \mathrm{L}^2(\Omega).
\end{align*}
Let us define the continuous  bilinear forms $a:\boldsymbol{\H}_{\Gamma_1}(\Omega)\times\boldsymbol{\H}_{\Gamma_1}(\Omega)\rightarrow\mathbb{R}$ and $b:\boldsymbol{\H}_{\Gamma_1}(\Omega)\times \mathrm{L}^2(\Omega)\rightarrow \mathbb{R}$ as follows:
$$
\begin{aligned}
	a(\bu,\bv)&:=\int_{\Omega} \mathbb{K}^{-1} \bu \cdot \bv + \nu \int_{\Omega} \nabla \bu :\nabla \bv,\,\,\,\forall\bu,\bv\in\boldsymbol{\H}_{\Gamma_1}(\Omega),\\
	\quad\text{and}\quad
	b(\bv,q)&:= - \int_{\Omega} q \div\bv\quad\forall\bv\in\boldsymbol{\H}_{\Gamma_1}(\Omega),\forall q\in \L^2(\Omega).
\end{aligned}
$$
Thus, the weak formulation of  system \eqref{eq:Brinkman1}--\eqref{bc:Gamma} is stated in the following manner:
\begin{problem}\label{prob:continuous}
	Find $\lambda\in\mathbb{R}$ and $(\boldsymbol{0},0)\neq(\bu,p)\in\boldsymbol{\H}_{\Gamma_1}(\Omega)\times \mathrm{L}^2(\Omega)$ such that
\begin{equation*}
	\label{eq:porous_system_continuous}
	\left\{
	\begin{aligned}
		a(\bu,\bv) + b(\bv,p) &= \lambda(\bu,\bv)_{0,\Omega},&\forall \bv\in\boldsymbol{\H}_{\Gamma_1}(\Omega),\\
		b(\bu,q)&=0,&\forall q\in \L^2(\Omega),
	\end{aligned}
	\right.
\end{equation*}
where $(\cdot,\cdot)_{0,\Omega}$ denotes the usual $\boldsymbol{\L}^2$ inner product.
\end{problem}

In the forthcoming analysis, we need a suitable weighted norm which depends on the physical parameters. With this in mind, given $(\bv,q)\in\boldsymbol{\H}_{\Gamma_1}(\Omega)\times \mathrm{L}^2(\Omega)$, we define the following norm:
\begin{equation*}
	\label{eq:weigthed-norm}
	\vertiii{(\bv,q)}^2:= \|\mathbb{K}^{-1/2} \bv\|_{0,\Omega}^2 +   \|\nu^{1/2} \nabla \bv \|_{0,\Omega}^2 + \| \div \bv \|_{0,\Omega}^2 + \Vert q\Vert_{0,\O}^2.
\end{equation*}

Let us define the kernel $\boldsymbol{\mathcal{K}}$ of $b(\cdot,\cdot)$ as follows
$$\boldsymbol{\mathcal{K}}:= \{ \bv \in \boldsymbol{\H}_{\Gamma_1}(\Omega): b(\bv, q)=0 \quad \forall q \in \mathrm{L}^2(\Omega) \} = \{ \bv \in\boldsymbol{\H}_{\Gamma_1}(\Omega): \div \bv = 0 \}. $$

Let us remark that the bilinear form $a(\cdot,\cdot)$ satisfies the following properties
\begin{align*}
	a(\bu, \bv)  &\le \max\{{\mathbb{K}^*,\nu}\}\Vert \bu\Vert_{1,\Omega}\Vert\bv\Vert_{1,\O}\quad \forall \bu,  \bv \in\boldsymbol{\H}_{\Gamma_1}(\Omega), \\
	a(\bv, \bv) &\ge \min\{{\mathbb{K}_*},\nu\}\Vert \bv\Vert_{1,\Omega}^2 \quad \forall \bv \in \boldsymbol{\mathcal{K}},
\end{align*}
whereas the form $b(\cdot,\cdot)$ satisfies  the following inf-sup condition
\begin{lemma}
There exists $\beta>0$ such that
$$
\sup_{\boldsymbol{0} \neq\bv \in \boldsymbol{\H}_{\Gamma_1}(\Omega)}\frac{b(\bv, q)}{\Vert\bv\Vert_{1,\Omega}}  \ge \beta \norm{q}_{0,\Omega}, \qquad \forall q \in \mathrm{L}^2(\Omega).
$$
\end{lemma}
\begin{proof}
Since we have a free-flow boundary condition, the proofs follows  by following the same arguments as in \cite[Remark 4.8 and Lemma 4.9]{MR2050138}.
\end{proof}

Now we define the following bilinear form
$$
\mathcal{A}((\bu,p),(\bv,q)):= a(\bu,\bv) + b(\bv,p) + b(\bu,q),\quad\forall\bu,\bv\in\boldsymbol{\H}_{\Gamma_1}(\Omega),\,\,\,\, \,\forall p,q\in \L^2(\O),
$$
which allows to rewrite Problem \ref{prob:continuous}  in the following manner:
\begin{problem}\label{prob:continuous-A}
	Find $\lambda\in\mathbb{R}$ and $(\boldsymbol{0},0)\neq(\bu,p)\in\boldsymbol{\H}_{\Gamma_1}(\Omega)\times \mathrm{L}^2(\Omega)$ such that
	\begin{equation*}
		\label{eq:porous_system_continuous-A}
			\mathcal{A}((\bu,p),(\bv,q))= \lambda(\bu,\bv)_{0,\Omega},\qquad\forall (\bv,q)\in\boldsymbol{\H}_{\Gamma_1}(\Omega)\times \L^2(\Omega).
	\end{equation*}
	\end{problem}
The following result establish a general inf-sup condition for the bilinear form $\mathcal{A}$.

\begin{lemma}
	\label{lemma:elliptic}
	For all $(\boldsymbol{0},0)\neq (\bu,p)\in\boldsymbol{\H}_{\Gamma_1}(\Omega)\times \L^2(\Omega)$, there exists $(\bv,q)\in\boldsymbol{\H}_{\Gamma_1}(\Omega)\times \L^2(\Omega)$ with $\vertiii{(\bv,q)}\leq C_1\vertiii{(\bu,p)}$ such that
	$$
	\mathcal{A}((\bu,p),(\bv,q))\geq C_2 \vertiii{(\bu,p)}^2,
	$$
	where $C_1$, $C_2$ are constants depending on the physical parameters.
\end{lemma}
\begin{proof}
The proof is followed by the same arguments as \cite[Lemma 4.3]{Paul2005347}
	\end{proof}

For the analysis of the eigenvalue problem, we introduce the so-called solution operators. These operators,  which we denote by $\bT$ and  $\boldsymbol{B}$,  are defined by 
\begin{align*}
\bT&:\boldsymbol{\L}^2(\O)\rightarrow\boldsymbol{\H}_{\Gamma_1}(\Omega),\,\,\,\bF\mapsto\bT\bF:=\widetilde{\bu}, \\
\boldsymbol{B}&:\boldsymbol{\L}^2(\O)\rightarrow\boldsymbol{\L}^2(\O),\,\,\,\bF\mapsto\boldsymbol{B}\bF:=\widetilde{p}, 
\end{align*}
where the pair $(\widetilde{\bu},\widetilde{p})\in\boldsymbol{\H}_{\Gamma_1}(\Omega)\times\L^2(\O)$ is the unique solution of the following source problem
\begin{equation} 
	\label{eq:source}
	\mathcal{A}((\widetilde{\bu},\widetilde{p}),(\bv,q))=(\bF,\bv)_{0,\O},\qquad \forall(\bv,q)\in \boldsymbol{\H}_{\Gamma_1}(\Omega)\times \L^2(\Omega),\\
\end{equation}
which is well posed thanks to Lemma \ref{lemma:elliptic}. 
Moreover, Lemma \ref{lemma:elliptic} together with  the Bab\v uska-Brezzi theory reveal that  operators $\bT$  and  $\boldsymbol{B}$  are well defined and bounded. 

Let $\mu$ be a real number such that  $\mu\neq 0$. Notice that $(\mu,\boldsymbol{u})\in \mathbb{R}\times\boldsymbol{\H}_{\Gamma_1}(\Omega)$ is an eigenpair of $\bT$ if and only if  there exists $p\in \L^2(\Omega)$ such that,  $(\l,(\boldsymbol{u},p))$ solves Problem \ref{prob:continuous-A}  with $\mu:=1/\l$.

Using the fact that the term $(\mathbb{K}^{-1}\bu,\bv)_{0,\O}$ is well defined, and taking advantage of the well known regularity results for the Stokes system   (see \cite{MR975121,MR1600081} for instance), we have  the following additional regularity result for the solution of the source problem \eqref{eq:source} and consequently, for the generalized eigenfunctions of $\bT$.
\begin{theorem}\label{th:regularidadfuente}
There exists $s>0$  that for all $\boldsymbol{f} \in \boldsymbol{\L}^2(\O)$, the solution $(\widetilde{\bu},\widetilde{p})\in\boldsymbol{\H}_{\Gamma_1}(\Omega)\times\L^2(\O)$ of problem \eqref{eq:source}, satisfies for the velocity $\widetilde{\bu}\in  \boldsymbol{\H}^{1+s}(\Omega)$, for the pressure $\widetilde{p}\in \H^s(\Omega)$, and
 \begin{equation*}
\|\widetilde{\bu}\|_{1+s,\O}+\|\widetilde{p}\|_{s,\O}\leq C_{\mathbb{K},\nu} \|\boldsymbol{f}\|_{0,\O},
 \end{equation*}
 where $C_{\mathbb{K},\nu}>0$ is a constant depending on the permeability and the viscosity.
\end{theorem}

It is important to note  that the estimate on Theorem \ref{th:regularidadfuente} holds for the source problem. On the other hand, for the eigenfunctions we have that there exists a positive constant $C$ depending on the physical parameters and the eigenvalue $\lambda$  such that 
 \begin{equation}\label{eq:reg_eigenfunction}
\|\bu\|_{1+r,\O}+\|p\|_{r,\O}\leq C \|\bu\|_{0,\O},
 \end{equation}
 where $r>0$.
Hence, because of the compact inclusion $\boldsymbol{\H}^{1+s}(\O)\hookrightarrow\boldsymbol{ \L}^2(\O)$, we 
conclude that $\bT$ is a compact operator and consequently, the following  spectral characterization of $\bT$ holds.
\begin{lemma}(Spectral Characterization of $\bT$).
The spectrum of $\bT$ is such that $\sp(\bT)=\{0\}\cup\{\mu_{k}\}_{k\in{N}}$ where $\{\mu_{k}\}_{k\in\mathbf{N}}$ is a sequence of real  eigenvalues that converge to zero, according to their respective multiplicities. 
\end{lemma}

\section{Finite element scheme}
\label{sec:fem}
In this section our aim is to describe a finite element discretization of problem \eqref{prob:continuous}. To do this task, we will introduce two families of inf-sup stable finite elements for the Stokes problem which trivially are also valid for the Brinkman load problem. Clearly,  the  properties of these spaces also hold for the eigenvalue problem. We begin by introducing some preliminary definitions and notations to perform the analysis. 


Let $\mathcal{T}_h=\{T\}$ be a conforming partition  of $\overline{\O}$ into closed simplices $T$ with size $h_T=\text{diam}(T)$. Define $h:=\max_{T\in\mathcal{T}_h}h_T$. Let us denote by $\boldsymbol{V}_h$ and $\mathcal{P}_h$ the finite element spaces that approximate the velocity field and the pressure, respectively. In particular, the study is focused in the following two elections:
\begin{itemize}
\item[(a)] The mini element \cite[Section 4.2.4]{MR2050138}: Here,
\begin{align*}
&\boldsymbol{V}_h=\{\bv_{h}\in\boldsymbol{C}(\overline{\O})\ :\ \bv_{h}|_T\in[\mathbb{P}_1(T)\oplus\mathbb{B}(T)]^{d} \ \forall \ T\in\mathcal{T}_h\}\cap\boldsymbol{\H}_{\Gamma_1}(\Omega),\\
&\mathcal{P}_h=\{ q_{h}\in C(\overline{\O})\ :\ q_{h}|_T\in\mathbb{P}_1(T) \ \forall \ T\in\mathcal{T}_h \}\cap \L^{2}(\O),
\end{align*}
where $\mathbb{B}(T)$ denotes the space spanned by local bubble functions.
\item[(b)] The lowest order Taylor--Hood element \cite[Section 4.2.5]{MR2050138}: In this case,
\begin{align*}
&\boldsymbol{V}_h=\{\bv_{h}\in\boldsymbol{C}(\overline{\O})\ :\ \bv_h|_T\in\mathbb{P}_2(T)^{d} \ \forall \ T\in\mathcal{T}_h\}\cap\boldsymbol{\H}_{\Gamma_1}(\Omega),\\
&\mathcal{P}_h=\{ q_{h}\in C(\overline{\O})\ :\ q_{h}|_T\in\mathbb{P}_1(T) \ \forall\ T\in\mathcal{T} \}\cap \L^{2}(\O).
\end{align*}
\end{itemize}

Let us remark that  from now and on, the discrete analysis will be performed in a general manner, in the sense that indistinctively we are considering both the mini element or the Taylor-Hood family.  If some difference must be claimed, we will point it out when is necessary. 

With the aforementioned discrete spaces at hand, the discretization of Problem \ref{prob:continuous} reads as follows:

\begin{problem}\label{prob:discrete}
	Find $\lambda_h\in\mathbb{R}$ and $(\boldsymbol{0},0)\neq(\bu_h,p_h)\in\boldsymbol{V}_h\times \mathcal{P}_h$ such that
\begin{equation*}
	\label{eq:porous_system_discrete}
	\left\{
	\begin{aligned}
		a(\bu_h,\bv_h) + b(\bv_h,p_h) &= \lambda_h(\bu_h,\bv_h)_{0,\Omega},&\forall \bv_h\in\boldsymbol{V}_h,\\
		b(\bu_h,q_h)&=0,&\forall q_h\in \mathcal{P}_h.
	\end{aligned}
	\right.
\end{equation*}
\end{problem}
Let $\boldsymbol{\mathcal{K}}_h$ be the discrete kernel of bilinear form  $b(\cdot,\cdot)$ defined by
$$\boldsymbol{\mathcal{K}}_h:= \{ \bv_h \in \boldsymbol{V}_h: b(\bv_h, q_h)=0 \quad \forall q_h \in \mathcal{P}_h \}. $$
For the elements on $\boldsymbol{\mathcal{K}}_h$, bilinear form  $a(\cdot,\cdot)$ satisfies
\begin{align*}
	a(\bv_h, \bv_h) &\ge \min\{{\mathbb{K}_*},\nu\}\Vert \bv_h\Vert_{1,\Omega}^2 \quad \forall \bv \in \boldsymbol{\mathcal{K}_h}.
\end{align*}
On the other hand, since mini element and Taylor-Hood are  inf-sup stable spaces, we have the there exists a constant  $\widehat{\beta}>0$ such that
$$
\sup_{\boldsymbol{0} \neq\bv_h \in \boldsymbol{V}_h}\frac{b(\bv_h, q_h)}{\Vert\bv_h\Vert_{1,\Omega}}  \ge \widehat{\beta} \norm{q_h}_{0,\Omega}, \qquad \forall q \in \mathcal{P}_h.
$$

Let us  introduce the discrete solution operators $\bT_h$ and $\boldsymbol{B}_h$ defined as follows
\begin{align*}\label{eq:operador_solucion_u_h}
	\bT_h&:\boldsymbol{\L}^2(\O)\rightarrow \boldsymbol{V}_h,\quad\boldsymbol{f}\mapsto \bT_h\boldsymbol{f}:=\widetilde{\bu}_h, \\
	\boldsymbol{B}_h&:\boldsymbol{\L}^2(\O)\rightarrow\mathcal{P}_h,\,\,\,\bF\mapsto\boldsymbol{B}_h\bF:=\widetilde{p}_h, 
\end{align*}
where the pair  $(\widetilde{\bu}_h, \widetilde{p}_h)\in\boldsymbol{V}_h\times \mathcal{P}_h$ is the unique solution of the following discrete source  problem
\begin{equation*}\label{def:brink_system_weak_disc_source}
	\left\{
	\begin{array}{rcll}
a(\widetilde{\bu}_h,\bv_h) + b(\bv_h,\widetilde{p}_h)&=& (\boldsymbol{f},\bv_h)_{0,\O}&\forall \bv_h\in\boldsymbol{V}_h,\\
b(\widetilde{\bu}_h,q_h)&=&0&\forall q_h\in\mathcal{P}_h.
\end{array}
	\right.
\end{equation*}

Moreover, if $\mathcal{L}_h$ represents the Lagrange interpolator operator and $\mathcal{P}_h$ represents the $\L^2$-projection,   the solutions of the continuous and discrete problems satisfy
the following error estimate
\begin{multline*}
\vertiii{(\widetilde{\bu}-\widetilde{\bu}_h,\widetilde{p}-\widetilde{p}_h)}^2\leq C\inf_{(\bv,q)\in\boldsymbol{V}_h\times\mathcal{P}_h}\vertiii{(\widetilde{\bu}-\bv,\widetilde{p}-q)}^2\\
\leq C\vertiii{(\widetilde{\bu}-\mathcal{L}_h\widetilde{\bu},\widetilde{p}-\mathcal{P}_h\widetilde{p})}^2
=\|\mathbb{K}^{1/2}(\widetilde{\bu}-\mathcal{L}_h\widetilde{\bu})\|_{0,\O}^2+\|\nu^{-1/2}\nabla(\widetilde{\bu}-\mathcal{L}_h\widetilde{\bu})\|_{0,\O}^2\\
+\|\div(\widetilde{\bu}-\mathcal{L}_h\widetilde{\bu})\|_{0,\O}^2
+\|\widetilde{p}-\mathcal{P}_h\widetilde{p}\|_{0,\O}^2\\
 \leq h^{2\min\{s,k\}}\max\{\mathbb{K}^*,\nu^*,1 \}\left(\|\widetilde{\bu}\|_{1+s,\O}^2+\|\widetilde{p}\|_{s,\O}^{2} \right).
\end{multline*}
Then, from Theorem \ref{th:regularidadfuente} we have 
\begin{equation}
\label{eq:error_triple}
\vertiii{(\widetilde{\bu}-\widetilde{\bu}_h,\widetilde{p}-\widetilde{p}_h)}\leq C_{\mathbb{K},\nu}h^{\min\{s,k\}}\left(\|\widetilde{\bu}\|_{1+s,\O}+\|\widetilde{p}\|_{s,\O}\right) \leq C_{\mathbb{K},\nu} h^{\min\{s,k\}}\Vert \boldsymbol{f}\Vert_{0,\O},
\end{equation}
where $k\in\{1,2\}$ depending on whether mini-element or Taylor-Hood family is being used.

Now we have all the necessary ingredients to guarantee the convergence in norm of  $\bT_h$ to $\bT$ as $h\rightarrow 0$.  
The following result establishes this convergence in the $\H^1$ and $\L^2$ norms.

\begin{theorem}\label{thm:erroroperator}
Let $\boldsymbol{f}\in \boldsymbol{\L}^2(\O)$ be such  that $\widetilde{\bu}:=\bT\boldsymbol{f}$, $\widetilde{p}:=\boldsymbol{B}\boldsymbol{f}$,  $\widetilde{\bu}_h:=\bT_h\boldsymbol{f}$ and  $\widetilde{p}_h:=\boldsymbol{B}_h\boldsymbol{f}$. Then, the follwing estimates hold
\begin{equation}
\label{eq:convergence_norms}
\|(\bT-\bT_h)\boldsymbol{f}\|_{1,\O}\leq C_{\mathbb{K},\nu}h^{\min\{s,k\}}\|\boldsymbol{f}\|_{0,\O},
\end{equation}
and
\begin{equation}
\label{eq:convergence_norms2}
 \|(\bT-\bT_h)\boldsymbol{f}\|_{0,\O}\leq C_{\mathbb{K},\nu}h^{2{\min\{s,k\}}}\|\boldsymbol{f}\|_{0,\O},
\end{equation}
where $s>0$ and $C_{\mathbb{K},\nu}$ are the constants involved in Theorem \ref{th:regularidadfuente}  and $k\in\{1,2\}$, depending on whether mini-element or Taylor-Hood family is being used.
\end{theorem}
\begin{proof}
The estimate of $\|(\bT-\bT_h)\boldsymbol{f}\|_{1,\O}$ in \eqref{eq:convergence_norms} is a consequence of  \eqref{eq:error_triple} and the regularity of Theorem \ref{th:regularidadfuente}. To prove the convergence in $\boldsymbol{\L}^2$ norm \eqref{eq:convergence_norms2}, we use a duality argument. Let us introduce the following auxiliary problem: 
find  $(\bphi,\psi)\in \boldsymbol{\H}_{\Gamma_1}(\Omega)\times \L^2(\Omega)$ such that
\begin{equation}
	\label{eq:auxiliar_problem}
	\left\{
	\begin{aligned}
		a(\bv,\bphi) + b(\bv,\psi) &= (\widetilde{\bu}-\widetilde{\bu}_h,\bv)_{0,\Omega},& \qquad \forall \bv\in\boldsymbol{\H}_{\Gamma_1}(\Omega),\\
		b(\bphi,q)&=0,&\forall q\in  \L^2(\Omega).
	\end{aligned}
	\right.
\end{equation}

Clearly \eqref{eq:auxiliar_problem} is well posed and its solution satisfies the following dependency on the data
\begin{equation}\label{eq_estimateaux}
\|\bphi\|_{1+s,\O}+\|\psi\|_{s,\O}\leq C_{\mathbb{K},\nu}\|\widetilde{\bu}-\widetilde{\bu}_h\|_{0,\O}.
\end{equation}
On the other hand, simple manipulations yield to the following identity
\begin{multline*}
\|\widetilde{\bu}-\widetilde{\bu}_h\|_{0,\O}^2=a(\widetilde{\bu}-\widetilde{\bu}_h,\bphi-\bphi_h)+b(\widetilde{\bu}-\widetilde{\bu}_h,\psi-\psi_h)+a(\widetilde{\bu}-\widetilde{\bu}_h,\bphi_h)+b(\widetilde{\bu}-\widetilde{\bu}_h,\psi_h)\\
=a(\widetilde{\bu}-\widetilde{\bu}_h,\bphi-\bphi_h)+b(\widetilde{\bu}-\widetilde{\bu}_h,\psi-\psi_h)-b(\bphi_h,\widetilde{p}-\widetilde{p}_h)\\
=a(\widetilde{\bu}-\widetilde{\bu}_h,\bphi-\bphi_h)+b(\widetilde{\bu}-\widetilde{\bu}_h,\psi-\psi_h)+b(\bphi-\bphi_h,\widetilde{p}-\widetilde{p}_h),
\end{multline*}
where for the last equality we have used that $b(\bphi,\widetilde{p}-\widetilde{p}_h)=0$. Hence,  taking $\bphi_h$ as the Scott-Zhang interpolant of $\bphi$,  $\psi_h$ as the $\L^2$-orthogonal projection of $\psi$ onto 
$\mathbb{P}_1$, applying Cauchy-Schwarz inequality,  approximation properties of the interpolator and the projector, estimate \eqref{eq_estimateaux}, multiplying and dividing by $\sqrt{\nu}$, and using the $ \boldsymbol{\H}^1$-error norm already proved in  \eqref{eq:convergence_norms}, we obtain 
$$\|(\bT-\bT_h)\boldsymbol{f}\|_{0,\O}=\|\widetilde{\bu}-\widetilde{\bu}_h\|_{0,\O}\leq C_{\mathbb{K},\nu}h^{2\min\{s,k\}}\|\boldsymbol{f}\|_{0,\O},$$
 where  $k\in\{1,2\}$, depending on whether mini-element or Taylor-Hood family is being used.
This concludes the proof
\end{proof}

With Theorem \ref{thm:erroroperator} at hand, and invoking the  well established theory of  \cite[Chapter IV]{MR0203473} and \cite[Theorem 9.1]{MR2652780}, we are in position  to conclude that  our numerical methods does not introduce spurious eigenvalues. This is stated in the following theorem.
\begin{theorem}
	\label{thm:spurious_free}
	Let $V\subset\mathbb{C}$ be an open set containing $\sp(\bT)$. Then, there exists $h_0>0$ such that $\sp(\bT_h)\subset V$ for all $h<h_0$.
\end{theorem}
 We first recall the definition of spectral projectors. Let $\mu$ be a nonzero isolated eigenvalue of $\bT$ with algebraic multiplicity $m$ and let $\Gamma$
be a disk of the complex plane centered in $\mu$, such that $\mu$ is the only eigenvalue of $\bT$ lying in $\Gamma$ and $\partial\Gamma\cap\sp(\bT)=\emptyset$. With these considerations at hand, we define the spectral projection of $\boldsymbol{E}$ associated to $\bT$ as follows:
\begin{itemize}
\item The spectral projector of $\bT$ associated to $\mu$ is $\displaystyle \boldsymbol{E}:=\frac{1}{2\pi i}\int_{\partial\Gamma} (z\boldsymbol{I}-\bT)^{-1}\,dz;$
\end{itemize}
where $\boldsymbol{I}$ represents the identity operator. Let us remark that $\boldsymbol{E}$ is  the projection onto the generalized eigenspace $R(\boldsymbol{E})$.
A consequence of Theorem \ref{thm:erroroperator} is that there exist $m$ eigenvalues, which lie in $\Gamma$, namely $\mu_{h,1},\ldots,\mu_{h,m}$, repeated according their respective multiplicities, that converge to $\mu$ as $h$ goes to zero. With this result at hand, we introduce the following spectral projection
\begin{equation*}
\boldsymbol{E}_h:=\frac{1}{2\pi i}\int_{\partial\Gamma} (z\boldsymbol{I}-\bT_h)^{-1}\,dz,
\end{equation*}
which is a projection onto the discrete invariant subspace $R(\boldsymbol{E}_h)$ of $\bT$, spanned by the generalized eigenvector of $\bT_h$ corresponding to 
 $\mu_{h,1},\ldots,\mu_{h,m}$.

Now we recall the definition of the \textit{gap} $\hdel(\cdot,\cdot)$ between two closed
subspaces $\mathfrak{X}$ and $\mathfrak{Y}$ of $\boldsymbol{\L}^2(\O)$:
$$
\hdel(\mathfrak{X},\mathfrak{Y})
:=\max\big\{\delta(\mathfrak{X},\mathfrak{Y}),\delta(\mathfrak{Y},\mathfrak{X})\big\}, \text{ where } \delta(\mathfrak{X},\mathfrak{Y})
:=\sup_{\underset{\left\|\boldsymbol{x}\right\|_{0,\O}=1}{\boldsymbol{x}\in\mathfrak{X}}}
\left(\inf_{\boldsymbol{y}\in\mathfrak{Y}}\left\|\boldsymbol{x}-\boldsymbol{y}\right\|_{0,\O}\right).
$$
With these definitions at hand, we derive the following error estimates for eigenfunctions and eigenvalues by following the results
from \cite[Theorem 13.8 and 13.10]{MR2652780}
\begin{theorem}
\label{thm:errors1}
The following estimates hold
\begin{equation*}
\hdel(R(\boldsymbol{E}),R(\boldsymbol{E}_h))\leq C  h^{\min\{r,k\}}\quad\text{and}\quad
|\mu-\mu_h|\leq C h^{\min\{r,k\}},
\end{equation*}
 where $r>0$ is the same as in \eqref{eq:reg_eigenfunction} and  $k\in\{1,2\}$, depending on whether mini-element or Taylor-Hood family is being used.

\end{theorem}

The error estimate for the eigenvalue $\mu\in (0,1)$ of $\bT$ leads to an analogous
estimate for the approximation of the eigenvalue $\l= \dfrac{1}{\mu}$ of Problem \ref{prob:continuous} by means of the discrete eigenvalues $\l_{h,i} := \dfrac{1}{\mu_{h,i}}$, $1 \leq i\leq  m$, of Problem \ref{prob:discrete}. 
We end this section proving error estimates for the eigenfunctions and eigenvalues.
\begin{lemma}
If $(\l_i,(\bu,p))$ is a solution of Problem \ref{prob:continuous}, then there exists \\
$(\l_{h,i},(\bu_{h},p_h))$ satisfying Problem \ref{prob:discrete}, with $\|\bu\|_{0,\O}=\|\bu_h\|_{0,\O}=1$ such that
\begin{align}\label{eq:1estimate}
\|\bu-\bu_{h}\|_{0,\O}&\leq C h^{2\min\{r,k\}},\\\label{eq:2estimate}
\|\bu-\bu_{h}\|_{1,\O}&\leq C h^{\min\{r,k\}},\\\label{eq:3estimate}
|\l_i-\l_{h,i}|&\leq C h^{2\min\{r,k\}},
\end{align}
where $r>0$ and $C$ is a constant depending on the physical constants and the corresponding eigenvalue given in \eqref{eq:reg_eigenfunction}  and $k\in\{1,2\}$, depending on whether mini-element or Taylor-Hood family is being used.
\end{lemma}
\begin{proof}
In order to prove the first estimate, we proceed in a similar way as in \cite[Lemma 3.5 and Lemma 3.6]{MR4728079}, therefore the details are omitted. Estimate \eqref{eq:1estimate}  is a consequence of Theorem \ref{thm:errors1}. Now for the estimate \eqref{eq:3estimate} let us consider the following well known algebraic identity
\begin{equation*}
\mathcal{A}((\bu-\bu_h,p-p_h),(\bu-\bu_h,p-p_h))+\lambda (\bu-\bu_h,\bu-\bu_h)=(\lambda-\lambda_h)(\bu_h,\bu_h)_{0,\O}.
\end{equation*}

Now, taking modulus on the above identity and using triangle inequality, Cauchy-Schwarz inequality, and Young's inequality, we have
\begin{multline*}
\left| \mathcal{A}((\bu-\bu_h,p-p_h),(\bu-\bu_h,p-p_h))\right|\\ \leq
 \|\mathbb{K}^{-1/2}(\bu-\bu_h)\|_{0,\O}^2+\|\nu^{1/2}\nabla(\bu-\bu_h)\|_{0,\O}^2+\|p-p_h\|_{0,\O}^2+\|\div(\bu-\bu_h)\|_{0,\O}^2\\
=\vertiii{(\bu-\bu_h,p-p_h)}^2 \leq C h^{2\min\{r,k\}}\Vert \boldsymbol{u}\Vert_{0,\O},
\end{multline*}
where the constant $C>0$ depends on the eigenvalue. Then, using that $\|\bu_h\|_{0,\O}=1$, we obtain
$$|\lambda-\lambda_h|\leq Ch^{2\min\{r,k\}}\Vert \boldsymbol{u}\Vert_{0,\O}.$$
This concludes the proof.
\end{proof}
\begin{remark}\label{remarkautovalores}
An important consequence from the previous proof is the following:  if $(\l_i,(\bu,p))$ is a solution of Problem \ref{prob:continuous}  and  $(\l_{h,i},(\bu_{h},p_h))$ is solution of Problem \ref{prob:discrete}, with $\|\bu\|_{0,\O}=\|\bu_h\|_{0,\O}=1$, then
$$|\l_i-\l_{h,i}|\leq \vertiii{(\bu-\bu_h,p-p_h)}^2.$$
\end{remark}

\section{A posteriori analysis}
\label{sec:apost}
In what follows, we focus our attention on an a posteriori error analysis. With this goal in mind, and by means of the finite elements previously introduced for the a priori analysis, we introduce a residual-base estimator which results to be fully computable. Let us remark that, for the forthcoming analysis, we consider eigenvalues with simple multiplicity, and their associated eigenfunctions. Lets us begin with some definitions.  For any element $T\in \CT_h$, we denote by $\CE_{T}$ the set of faces/edges of $T$
and 
$$\CE_h:=\bigcup_{T\in\CT_h}\CE_{T}.$$
We consider the decomposition $\CE_h:=\CE_{\O}\cup\CE_{\partial\O}$,
where  $\CE_{\partial\O}:=\{\ell\in \CE_h:\ell\subset \partial\O\}$
and $\CE_{\O}:=\CE\backslash\CE_{\partial \O}$.
For each inner face/edge $\ell\in \CE_{\O}$ and for any  sufficiently smooth  tensor
$\btau$, we define the jump on $\ell$ by $\jump{\btau}:=\btau\vert_T\boldsymbol{n}_{T}+\btau\vert_{T'}\boldsymbol{n}_{T'} ,$
where $T$ and $T'$ are  the two elements in $\CT_{h}$  sharing the
face/edge $\ell$ and $\boldsymbol{n}_{T}$ and $\boldsymbol{n}_{T'}$ are the respective outer unit normal vectors.

\subsection{Local and global indicators}
For an element $T\in\CT_h$ we introduce the following local error indicator:
\begin{multline*}
\eta_T^2:=h_T^2\|\lambda_h\bu_h+\nu\Delta \bu_{h}-\mathbb{K}^{-1}\bu_h-\nabla p_h\|_{0,T}^2\\
+\|\div\bu_h\|_{0,T}^2+\frac{h_e}{2}\sum_{e\in\CT_h}\|\jump{(\nu\nabla\bu_h-p_h\mathbb{I})\boldsymbol{n}}\|_{0,e}^2,
\end{multline*}
and the corresponding global estimator
\begin{equation}\label{eq_est_global}
\eta:=\left(\sum_{T\in\CT_h}\eta_T^2 \right)^{1/2}.
\end{equation}

We define  the errors between the continuous and discrete eigenfunctions by 
$$
\eu:=\bu - \bu_{h}, \quad \ep:=p - p_h.$$

Now, our task is to demonstrate that our estimator is reliable and efficient. Let us begin with the reliability analysis.
\subsection{Reliability}
The goal of this section is to derive a global reliability bound for our proposed estimator defined in \eqref{eq_est_global}. This is contained in the following result.
\begin{theorem}
\label{thrm:gonzalo1}
There exists a constant $C > 0$ independent of $h$ such that
$$
\vertiii{({\bu}-{\bu}_h,{p}-{p}_h)}\leq C \left( \eta+ \dfrac{1}{\sqrt{\nu}}|\l-\l_h|+\dfrac{|\l_h| }{\sqrt{\nu}}\|\bu-\bu_h\|_{0,\O}\right).
$$
\end{theorem}
\begin{proof}
Let $(\l,(\bu,p))$ be the  solution of Problem \ref{prob:continuous} and $(\l_h,(\bu_h,p_h))$ be its finite element approximation given as the solution of Problem \ref{prob:discrete}. Then we have:
\begin{equation*}
\mathcal{A}((\eu,\ep),(\bv_h,0))=(\l \bu-\l_h\bu_h,\bv_h)_{0,\O},
\end{equation*}
where $\bv_h$ is the Scott-Zhang  interpolant of $\bv$.
Let 
$\vertiii{(\bv,q)}\leq C_1\vertiii{(\eu,\ep)}$ where $C_1>0$ is the constant involved in Lemma \ref{lemma:elliptic}.  Then, from the previous estimate and using integration by parts, it follows that  
 \begin{multline*}
C_2\vertiii{(\eu,\ep)}^2\leq  \mathcal{A}((\eu,\ep),(\bv,q))
=\mathcal{A}((\eu,\ep),(\bv-\bv_h,q))+\mathcal{A}((\eu,\ep),(\bv_h,0))\\
=\l(\bu,\bv)_{0,\O}-\l_h(\bu_h,\bv_h)_{0,\O}-\mathcal{A}((\bu_h,p_h),(\bv-\bv_h,q))\\
=\l_h(\bu_h,\bv_h-\bv)_{0,\O}+(\l-\l_h)(\bu,\bv)_{0,\O}+\l_h(\bu-\bu_h,\bv)_{0,\O}-\mathcal{A}((\bu_h,p_h),(\bv-\bv_h,q))\\
= (\l-\l_h)(\bu,\bv)_{0,\O}+\l_h(\bu-\bu_h,\bv)_{0,\O}+\int_{\O}\div\bu_h q+\l_h\int_\O\bu_h\cdot(\bv_h-\bv)\\
 +\sum_{T\in\mathcal{T}_h}\left(\int_T\div(\bv-\bv_h)p_h-\int_{T} \mathbb{K}^{-1} \bu_h \cdot (\bv-\bv_h )- \nu \int_{T} \nabla \bu_h :\nabla(\bv-\bv_h) \right)\\
 = (\l-\l_h)(\bu,\bv)_{0,\O}+\l_h(\bu-\bu_h,\bv)_{0,\O}+\int_{\O}\div\bu_h q\\
 +\sum_{T\in\mathcal{T}_h}\left(\int_T\left(\l_h\bu_h+\nu\Delta\bu_h-\mathbb{K}^{-1} \bu_h-\nabla p_h\right)\cdot(\bv-\bv_h)\right)\\
 -\sum_{T\in\mathcal{T}_h}\int_{\partial T}\left(\nu\nabla \bu_h-p_h\mathbb{I}\right)\boldsymbol{n}\cdot(\bv-\bv_h)\\
 \leq C \left( \eta+ \dfrac{1}{\sqrt{\nu}}|\l-\l_h|+\dfrac{|\l_h| }{\sqrt{\nu}}\|\bu-\bu_h\|_{0,\O}\right)\vertiii{({\bv},{q})},
 \end{multline*}
 where for the last estimate we have used the approximation properties of the Scott-Zhang interpolator.
Hence, the proof is concluded  using the fact that  $\vertiii{(\bv,q)}\leq C_1\vertiii{(\eu,\ep)}$.
\end{proof}

Finally, we also have an upper bound for the eigenvalue approximation.
\begin{corollary}
There exists a constant $C>0$ independent of $h$ such that
$$|\l-\l_{h}|\leq  \left( \eta+ \dfrac{1}{\sqrt{\nu}}|\l-\l_h|+\dfrac{|\l_h| }{\sqrt{\nu}}\|\bu-\bu_h\|_{0,\O}\right)^2.$$
\end{corollary}
\begin{proof}
The result is a direct consequence of Remark \ref{remarkautovalores}  and the previous theorem.
\end{proof}

%

\subsection{Efficiency}
The efficiency analysis follows the same arguments reported in, for instance \cite{MR4728079}, which consists in arguments based in localization of bubble functions, as in \cite{MR3059294}. Hence, the following result can be derived for our context.
\begin{lemma}(Efficiency) The following estimate holds 
$$\eta\leq C(\|\bu-\bu_h\|_{1,\O}^2+\|p-p_h\|_{0,\O}^2+h.o.t),$$
where the constant $C>0$ is independent of mesh size and the discrete solution.
\end{lemma}

\section{Numerical experiments}
\label{sec:numerics}
This section is dedicated to conducting various numerical experiments to assess the performance of the scheme across different geometries and physical configurations. The implementations are using the DOLFINx software \cite{barrata2023dolfinx,scroggs2022basix}, where the SLEPc eigensolver \cite{hernandez2005slepc} and the MUMPS linear solver are employed to solve the resulting generalized eigenvalue problem. Meshes are generated using GMSH \cite{geuzaine2009gmsh} and the built-in generic meshes provided by DOLFINx. The convergence rates for each eigenvalue are determined using least-square fitting and highly refined meshes. 

An important fact of the proposed formulation is that it relies on the use of well-known inf-sup stable elements. Moreover, as the permeability parameter  $\mathbb{K}$ ranges in bounded parameters throughout the domain, the computational complexity remains the same as a usual Stokes eigenvalue problem with the same elements (see for example \cite[Section 5.4]{LRVSISC} for a benchmark).

In what follows, we denote the mesh resolution by $N$, which is connected to the mesh-size $h$ through the relation $h\sim N^{-1}$. We also denote the number of degrees of freedom by $\texttt{dof}$. The relation between $\texttt{dof}$ and the mesh size is given by $h\sim\texttt{dof}^{-1/d}$, with $d\in\{2,3\}$. 

Let us define $\err(\lambda_i)$ as the error on the $i$-th eigenvalue, with
$$
\err(\lambda_i):={\vert \lambda_{h,i}-\lambda_{i}\vert},
$$
where $\lambda_i$ is the extrapolated value. Similarly, the effectivity indexes with respect to $\eta$ and the eigenvalue $\lambda_{h,i}$ is defined by
$$
\eff(\lambda_i):=\frac{\err(\lambda_i)}{\eta^2}.$$

In order to apply the adaptive finite element method, we shall generate a sequence of nested conforming triangulations using the loop
\begin{center}
	\textrm{solve $\rightarrow$ estimate $\rightarrow$ mark $\rightarrow$ refine,} 
\end{center}
based on \cite{verfuhrt1996}:
\begin{enumerate}
	\item Set an initial mesh $\CT_{h}$.
	\item Solve Problem \eqref{prob:discrete} in the actual mesh to obtain $\lambda_{h,i}$ and $(\bu_{h},p_h)$. 	
	\item Compute $\eta_T$ for each $T\in\CT_{h}$ using the eigenfunctions $(\bu_{h},p_h)$. 
	\item Use the maximum marking strategy to refine each $T'\in \CT_{h}$ whose indicator $\eta_{T'}$ satisfies
	$$
	\eta_{T'}\geq 0.5\max\{\eta_{T}\,:\,T\in\CT_{h} \},
	$$
	\item Set $\CT_{h}$ as the actual mesh and go to step 2.
\end{enumerate}

The refinement algorithm is the one implemented by DOLFINx through the command \texttt{refine}, which implements Plaza and Carey's algorithms for 2D and 3D geometries. The algorithms use local refinement of simplicial grids based on the skeleton with local facet marking information. 

It is important to note that although theoretically we have assumed that the permeability tensor is positive definite, in the numerical experiments we will assume that in some parts of the domain it is $\boldsymbol{0}$. Consequently, the behavior of the eigenfunctions in these regions aligns with that of the Stokes problem.

For simplicity, we have taken $\nu=1$ in all the experiments. We will consider several definitions for $\mathbb{K}$ in order to study the physical behavior of the eigenvalues, as well as their convergence rate. The division of $\Omega$ into subdomains is such that there is mesh conformity between the regions, i.e, the subregions are delimited exactly by the facets of the domain. By the nature of the problem, the choice of a porous subdomain is known to be very general. Therefore, the experiments below will be focused on basic cases in order to provide a study on the behavior of eigenfunctions on porous media that can be compared in some sense with the existing literature. For full Dirichlet boundary conditions ($\Gamma_2=\emptyset$), we have that $p_h\in P_h^0$, with
$$\mathcal{P}_h^0=\{ q_{h}\in C(\overline{\O})\ :\ q_{h}|_T\in\mathbb{P}_1(T) \ \forall\ T\in\mathcal{T} \}\cap \L_0^{2}(\O),$$
which means that the condition $\int_{\O} p_h=0$ must be satisfied. The theory with the spaces $\L_0^2(\Omega)$ and $\mathcal{P}_h^0$ is exactly the same as in Sections \ref{sec:model_problem}--\ref{sec:fem}. The zero integral mean condition is implemented in Problem \eqref{prob:discrete} by introducing a real Lagrange multiplier. This is achieved by using the \texttt{scifem} package along FEniCSx.

\subsection{Square domain with internal porous region} 
\label{subsec:test-cuadrado}
In this experiment we consider the unit square with boundary conditions $\bu=\boldsymbol{0}$ over the whole boundary, The domain is such that there is a heterogeneous distribution of permeability given by
$$
\mathbb{K}^{-1}=\left\{
\begin{aligned}
	&\kappa\mathbb{I},&\text{if } (x,y)\in\Omega_D,\\
	&\boldsymbol{0}, & \text{if } (x,y)\in\Omega_S,
\end{aligned}
\right.
$$ 
where $\Omega:=(0,1)^2$, $\Omega_D:=(3/8,5/8)^2$ and $\Omega_S:=\Omega\backslash\Omega_D$. An example of the geometry used for this experiment is depicted in Figure \ref{fig:square-domain}. The permeability factor of the porous subdomain is taken as $\kappa=10^j, j=-8,-7,...,-1,0,1,2,...,7,8$. The choice for $\kappa$ allows to represent a wide range of possible porous conditions on $\Omega_D$, while the permeability chosen for the rest of the domain defines a zone with a Stokes-like behavior.  The mesh level on this mesh is such that the $h\approx 1 /N$. 
\begin{figure}[!hbt]\centering
	\includegraphics[scale=0.15]{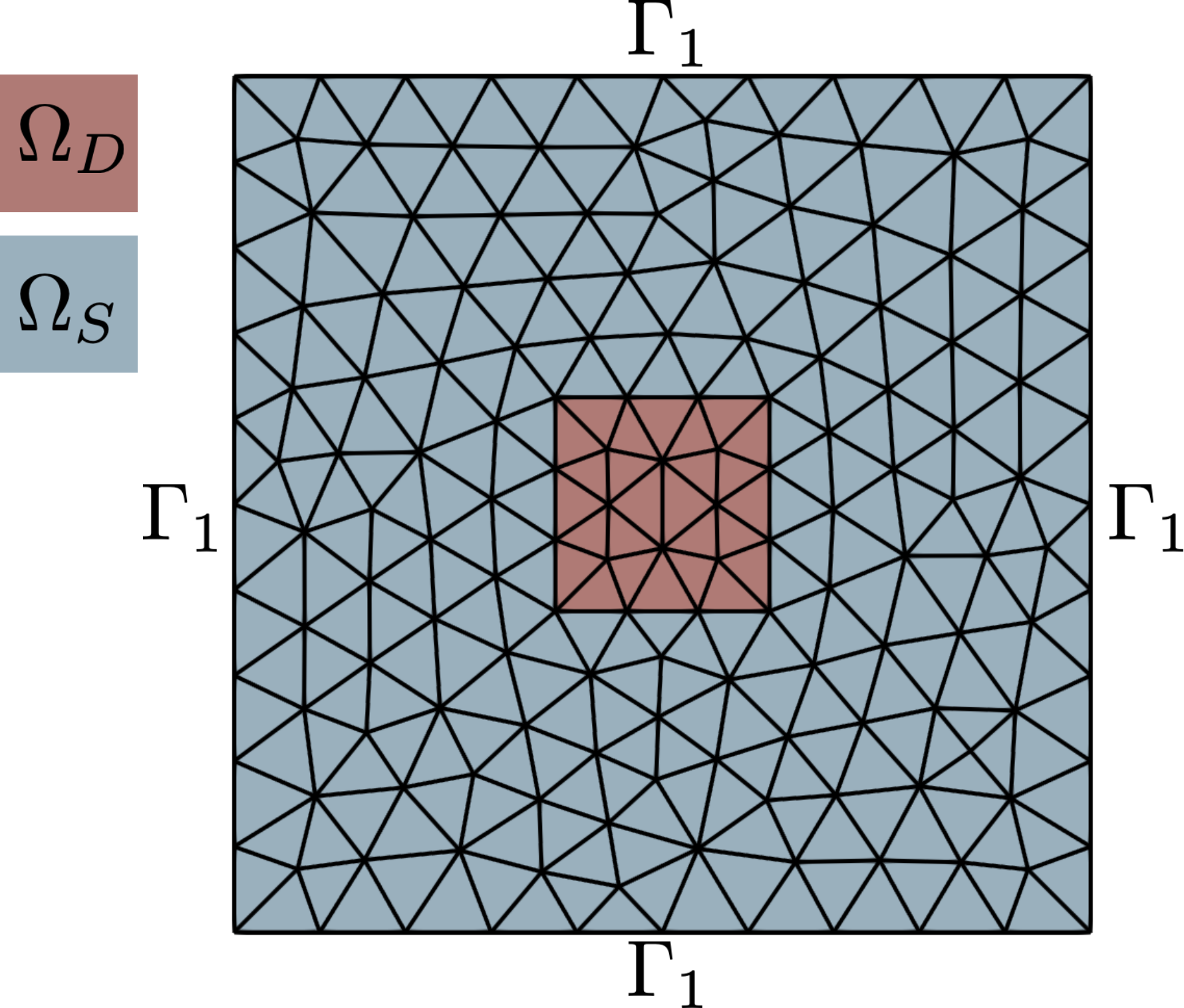}
	\caption{Test \ref{subsec:test-cuadrado}. Sample domain of a square with heterogeneous permeability parameter $\mathbb{K}$ and $N=10$.}
	\label{fig:square-domain}
\end{figure}

The convergence results for the first five computed eigenvalues with  values of $\kappa$ given by  $\kappa\in\{10^{-8},10^3,10^5,10^8\}$, using mini and Taylor-Hood elements, are given in Tables \ref{table-square2D-Stokes}-\ref{table-square2D-TH}. In Table \ref{table-square2D-Stokes} we observe that for $\mathbb{K}^{-1}\vert_{\Omega_D}=10^{-8}\mathbb{I}$, i.e., almost a full homogeneous full permeability, the eigenvalues are similar to those of Stokes (see \cite[Section 5]{turk2016stabilized}) with optimal rates for each element. As we start to decrease the permeability parameter, Tables \ref{table-square2D-MINI}-\ref{table-square2D-TH}  shows a considerable decrease of the convergence rate. For $\mathbb{K}^{-1}=10^5\mathbb{I}$ and onwards, we start to reach a computational impermeable limit, on which the regularity of the eigenfunction is affected because of the re-entrant corner generated on the subdomain. A further exploration of the spectrum is portrayed in Figure \ref{fig:convergence-evolution-square} where we have plotted the evolution of the first five eigenvalues as we change the permeability parameter. We observe that the first eigenvalue is the less affected by these changes. Also, note  that for low permeability values we have obtained eigenvalues similar to those of the Stokes eigenvalue problem with homogeneous boundary conditions in $\Gamma_1$, while for high permeability we have similar eigenvalues to those of Stokes on a unit square domain \cite[Section 5]{turk2016stabilized}. Samples of the first computed eigenfunctions for different permeabilities are depicted in Figure \ref{fig:square2D-uh1-allcases}. It is evident how a low permeability zone (obstacle) start to appear for small values of $\mathbb{K}$ on $\Omega_D$, with increasing gradients of pressures along $\partial\Omega_D$.

We finish this experiment by studying the a posteriori estimator on the case $\mathbb{K}^{-1}=10^5\mathbb{I}$ for the single multiplicity eigenvalues $\lambda_{h,1},\lambda_{h,2}$ and $\lambda_{h,5}$. A total of 15 adaptive iterations were performed with mini elements. The error and efficiency results are given on Figure \ref{fig:error-eff-square2D}. 
The slopes for the adaptive error lines behave like  $\mathcal{O}(\texttt{dof}^{-1.07})$, $\mathcal{O}(\texttt{dof}^{-1.01})$ and $\mathcal{O}(\texttt{dof}^{-1.0})$ for $\lambda_{h,1}, \lambda_{h,2}$ and $\lambda_{h,5}$, respectively. The effectivity indexes remains bounded above and beyond away from zero. Finally, sample meshes on the 14th iteration and the corresponding eigenfunctions are given in Figure \ref{fig:square2D-adaptive}.

\begin{table}[hbt!]
	\centering 
	{\footnotesize
		\begin{center}
			\caption{Test \ref{subsec:test-cuadrado}.Convergence behavior of the first five computed eigenvalues on the square domain for $\mathbb{K}^{-1} \vert_{\Omega_D} = 10^{-8} \mathbb{I}$.}
			\begin{tabular}{|c c c c |c| c| c|}
				\hline
				\hline
				$N=30$             &  $N=60$         &   $N=90$         & $N=120$ & Order & $\lambda_{\text{extr}}$ &\cite{turk2016stabilized} \\
				\hline
				\multicolumn{7}{|c|}{mini-element}\\
				\hline
				52.5163  &    52.3896  &    52.3641  &    52.3558  & 2.05 &    52.3448 & 52.3447 \\
				92.6447  &    92.2622  &    92.1841  &    92.1587  & 2.03 &    92.1243 & 92.1245 \\
				92.6472  &    92.2627  &    92.1843  &    92.1588  & 2.03 &    92.1241 &  92.1246  \\
				129.2619  &   128.4861  &   128.3276  &   128.2778  & 2.04 &   128.2085 & 128.2100  \\
				155.5077  &   154.4997  &   154.2890  &   154.2190  & 1.98 &   154.1206& 154.1260  \\
				\hline
				\multicolumn{7}{|c|}{Taylor-Hood}\\
				\hline
				52.3449  &    52.3447  &    52.3447  &    52.3447  & 4.19 &    52.3447  & 52.3447 \\
				92.1252  &    92.1244  &    92.1244  &    92.1244  & 4.16 &    92.1244 & 92.1245 \\
				92.1252  &    92.1245  &    92.1244  &    92.1244  & 4.05 &    92.1244 &  92.1246  \\
				128.2119  &   128.2097  &   128.2096  &   128.2096  & 4.13 &   128.2096 & 128.2100  \\
				154.1287  &   154.1257  &   154.1255  &   154.1255  & 4.02 &   154.1255  & 154.1260  \\
				\hline
				\hline             
			\end{tabular}
	\end{center}}
	\smallskip
	\label{table-square2D-Stokes}
\end{table}

\begin{table}[hbt!]
	\centering 
	{\footnotesize
		\begin{center}
			\caption{Test \ref{subsec:test-cuadrado}. Convergence behavior of the first five computed eigenvalues on the square domain using mini elements for various choices of $\mathbb{K}$.}
			\begin{tabular}{|c c c c |c| c|}
				\hline
				\hline
				$N=30$             &  $N=60$         &   $N=90$         & $N=120$ & Order & $\lambda_{\text{extr}}$ \\
				\hline
								\multicolumn{6}{|c|}{$\mathbb{K}^{-1}\vert_{\Omega_D}=10^{3}\mathbb{I}$}\\
								\hline
				  65.7509  &    65.4708  &    65.4108  &    65.3919  & 1.96 &    65.3633  \\
				  169.4892  &   168.2092  &   167.9457  &   167.8624  & 2.02 &   167.7440  \\
				  184.6488  &   183.1882  &   182.8878  &   182.7918  & 2.02 &   182.6572  \\
				  184.6618  &   183.1907  &   182.8884  &   182.7923  & 2.02 &   182.6564  \\
				  207.0402  &   205.1323  &   204.7262  &   204.5921  & 1.96 &   204.3997  \\
				\hline
				\multicolumn{6}{|c|}{$\mathbb{K}^{-1}\vert_{\Omega_D}=10^{5}\mathbb{I}$}\\
				\hline	
				
				  75.5617  &    74.8817  &    74.6664  &    74.5898  & 1.66 &    74.4516  \\
				  220.2530  &   216.1858  &   215.1133  &   214.7567  & 1.63 &   214.0401  \\
				  227.8248  &   223.9023  &   222.8955  &   222.5654  & 1.68 &   221.9303  \\
				  227.8493  &   223.9183  &   222.9108  &   222.5731  & 1.67 &   221.9274  \\
				  240.5004  &   236.6053  &   235.6493  &   235.3368  & 1.75 &   234.7770  \\

				\hline
				\multicolumn{6}{|c|}{$\mathbb{K}^{-1}\vert_{\Omega_D}=10^{8}\mathbb{I}$}\\
				\hline
				
				   76.6489  &    76.1506  &    75.9985  &    75.9427  & 1.38 &    75.8082  \\
				   225.2753  &   221.7454  &   220.7918  &   220.4577  & 1.58 &   219.7861  \\
				   232.1130  &   228.6457  &   227.7560  &   227.4332  & 1.64 &   226.8420  \\
				   232.1330  &   228.6763  &   227.7563  &   227.4487  & 1.62 &   226.8304  \\
				   244.0683  &   240.5808  &   239.6952  &   239.3881  & 1.67 &   238.8185  \\

				\hline
				\hline             
			\end{tabular}
	\end{center}}
	\smallskip
	
	\label{table-square2D-MINI}
\end{table}

\begin{table}[hbt!]
	\centering 
	{\footnotesize
		\begin{center}
			\caption{Test \ref{subsec:test-cuadrado}. Convergence behavior of the first five computed eigenvalues on the square domain using Taylor-Hood elements for various choices of $\mathbb{K}^{-1}$. }
			\begin{tabular}{|c c c c |c| c|}
				\hline
				\hline
				$N=30$             &  $N=60$         &   $N=90$         & $N=120$ & Order & $\lambda_{\text{extr}}$ \\
				\hline
				\multicolumn{6}{|c|}{$\mathbb{K}^{-1}\vert_{\Omega_D}=10^{3}\mathbb{I}$}\\
				\hline	
				   65.3662  &    65.3659  &    65.3658  &    65.3658  & 3.52 &    65.3658  \\
				   167.7502  &   167.7483  &   167.7481  &   167.7481  & 3.32 &   167.7480  \\
				   182.6633  &   182.6608  &   182.6605  &   182.6605  & 3.62 &   182.6605  \\
				   182.6634  &   182.6608  &   182.6606  &   182.6605  & 3.32 &   182.6605  \\
				   204.4187  &   204.4122  &   204.4118  &   204.4117  & 3.99 &   204.4117  \\

				\hline
				\multicolumn{6}{|c|}{$\mathbb{K}^{-1}\vert_{\Omega_D}=10^{5}\mathbb{I}$}\\
				\hline	
				
				   74.5266  &    74.4713  &    74.4567  &    74.4537  & 1.78 &    74.4455  \\
				   214.1172  &   214.1604  &   214.1703  &   214.1736  & 1.85 &   214.1789  \\
				   221.9733  &   222.0154  &   222.0271  &   222.0314  & 1.52 &   222.0403  \\
				   221.9793  &   222.0238  &   222.0312  &   222.0329  & 1.38 &   222.0389  \\
				   234.7379  &   234.8175  &   234.8396  &   234.8459  & 1.61 &   234.8608  \\

				\hline
				\multicolumn{6}{|c|}{$\mathbb{K}^{-1}\vert_{\Omega_D}=10^{8}\mathbb{I}$}\\
				\hline
				
				        75.5378  &    75.6628  &    75.7119  &    75.7314  & 1.47 &    75.7711  \\
				     218.3013  &   219.0216  &   219.2751  &   219.3717  & 1.45 &   219.5799  \\
				     225.5422  &   226.1622  &   226.3920  &   226.4724  & 1.48 &   226.6518  \\
				     225.5437  &   226.1838  &   226.3923  &   226.4781  & 1.45 &   226.6520  \\
				     237.6175  &   238.2001  &   238.4031  &   238.4801  & 1.44 &   238.6483  \\
				\hline
				\hline             
			\end{tabular}
	\end{center}}
	\smallskip
	\label{table-square2D-TH}
\end{table}

\begin{figure}[!hbt]\centering
	\begin{minipage}{0.32\linewidth}\centering
		{\footnotesize $\bu_{h,1},\mathbb{K}^{-1}\vert_{\Omega_D}=10^{-8}\mathbb{I}$}\\
		\includegraphics[scale=0.12,trim=23.2cm 5cm 23.2cm 5cm,clip]{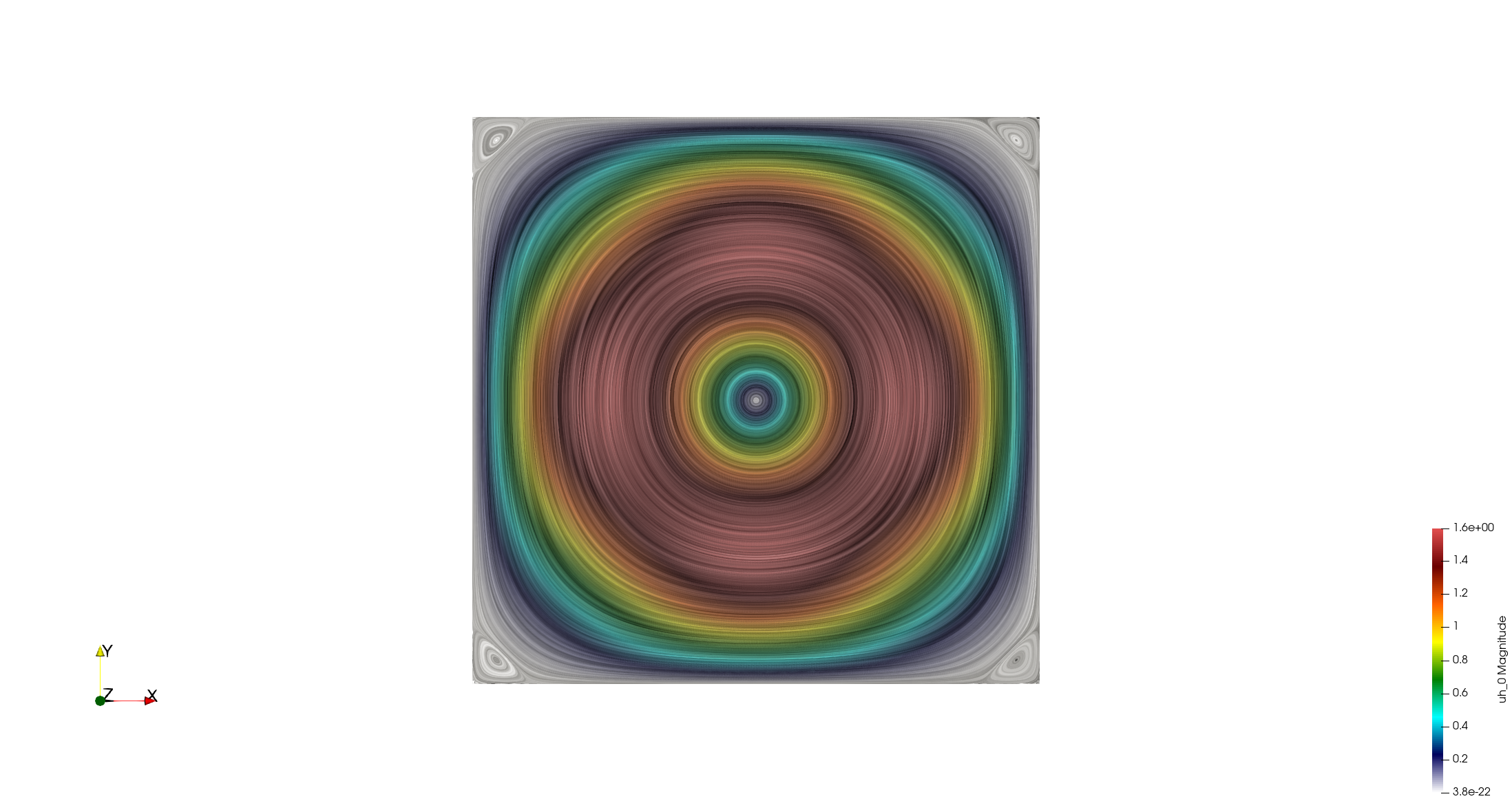}
	\end{minipage}
	\begin{minipage}{0.32\linewidth}\centering
		{\footnotesize $\bu_{h,1},\mathbb{K}^{-1}\vert_{\Omega_D}=10^{3}\mathbb{I}$}\\
		\includegraphics[scale=0.12,trim=23.2cm 5cm 23.2cm 5cm,clip]{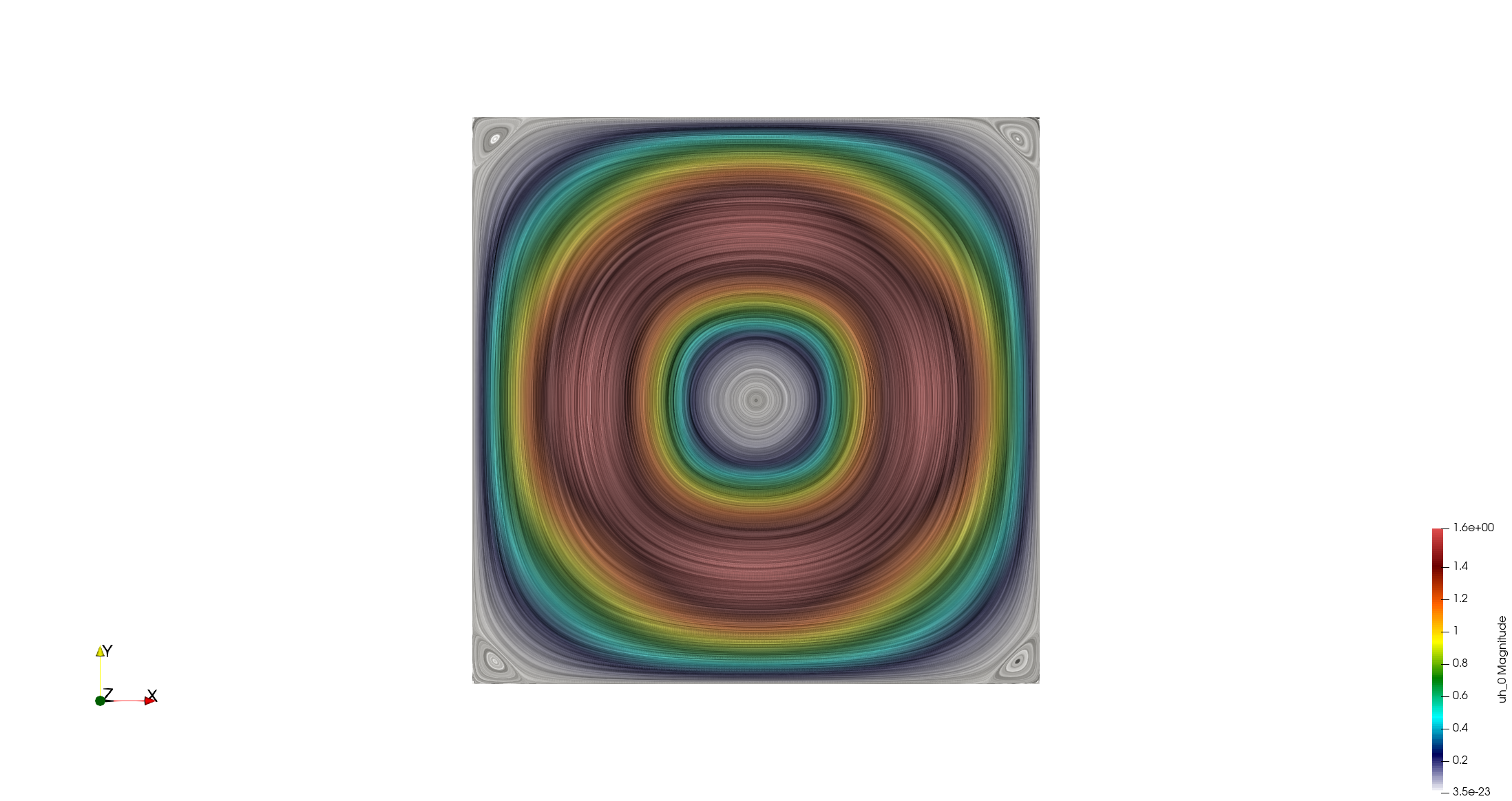}
	\end{minipage}
	\begin{minipage}{0.32\linewidth}\centering
		{\footnotesize $\bu_{h,1},\mathbb{K}^{-1}\vert_{\Omega_D}=10^{8}\mathbb{I}$}\\
		\includegraphics[scale=0.12,trim=23.2cm 5cm 23.2cm 5cm,clip]{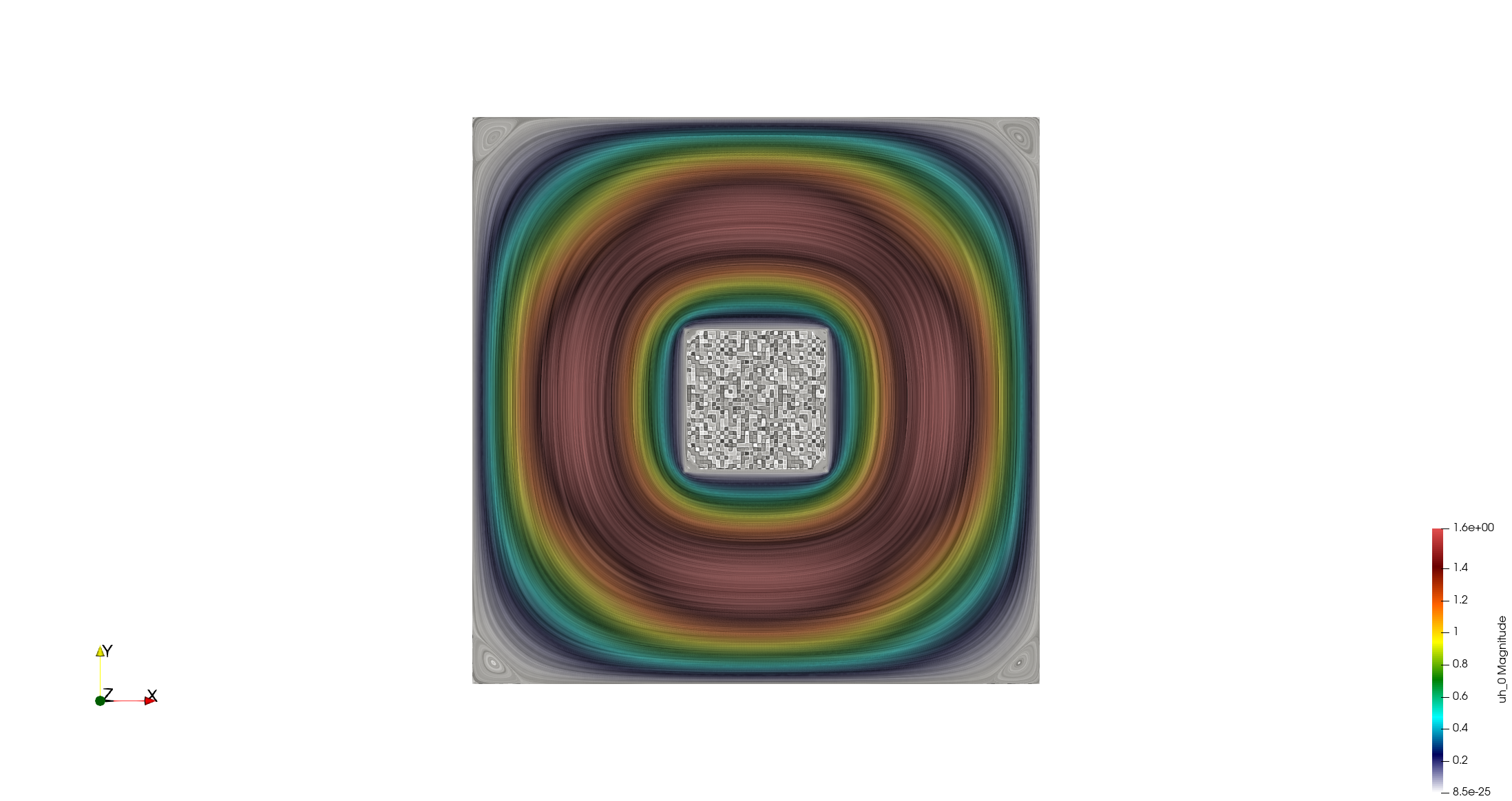}
	\end{minipage}\\
	\begin{minipage}{0.32\linewidth}\centering
		{\footnotesize $p_{h,1},\mathbb{K}^{-1}\vert_{\Omega_D}=10^{-8}\mathbb{I}$}\\
		\includegraphics[scale=0.12,trim=23.2cm 5cm 23.2cm 5cm,clip]{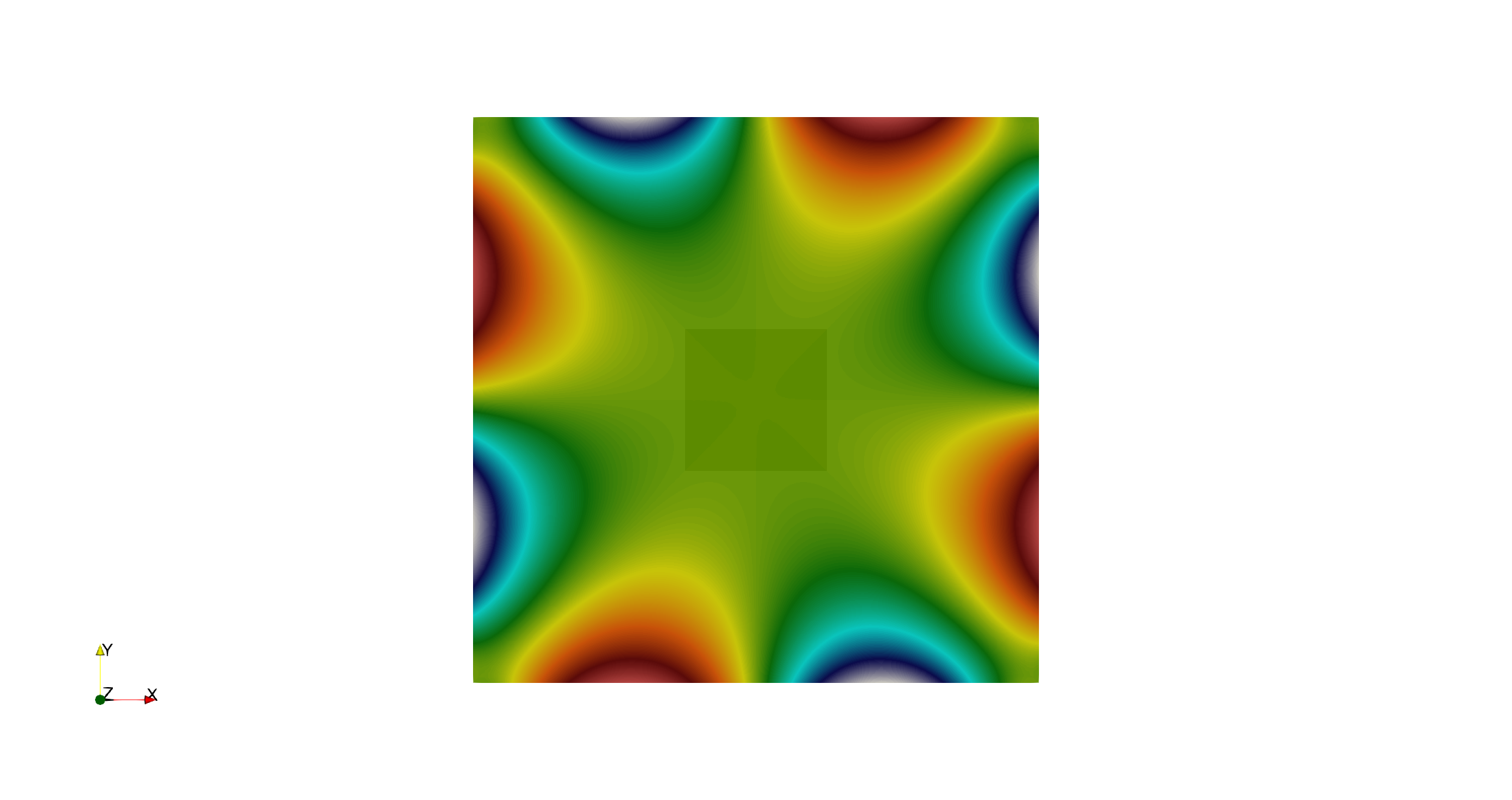}
	\end{minipage}
	\begin{minipage}{0.32\linewidth}\centering
		{\footnotesize $p_{h,1}, \mathbb{K}^{-1}\vert_{\Omega_D}=10^{3}\mathbb{I}$}\\
		\includegraphics[scale=0.12,trim=23.2cm 5cm 23.2cm 5cm,clip]{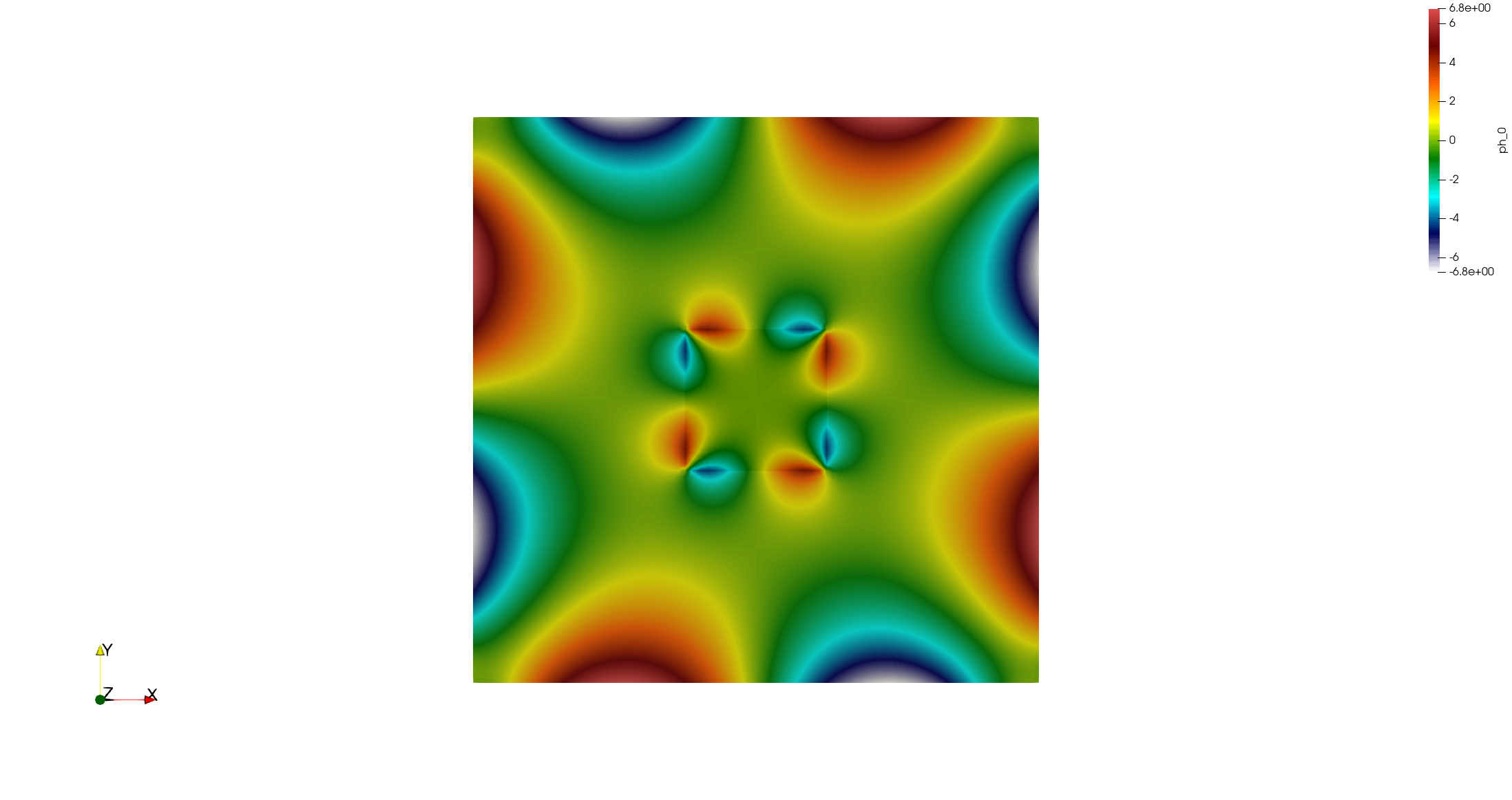}
	\end{minipage}
	\begin{minipage}{0.32\linewidth}\centering
		{\footnotesize $p_{h,1}, \mathbb{K}^{-1}\vert_{\Omega_D}=10^{8}\mathbb{I}$}\\
		\includegraphics[scale=0.12,trim=23.2cm 5cm 23.2cm 5cm,clip]{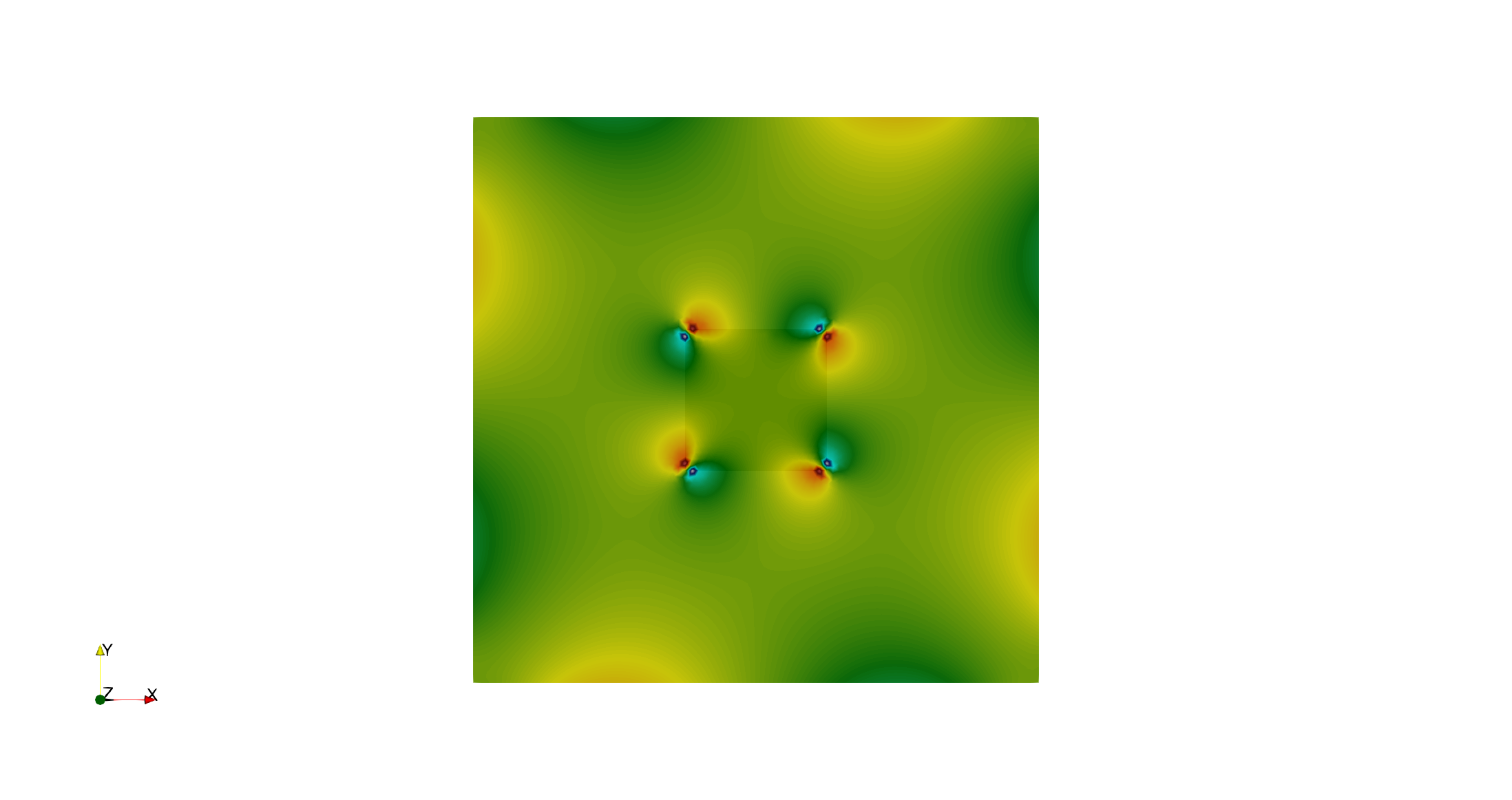}
	\end{minipage}\\
	\caption{Test \ref{subsec:test-cuadrado}. Comparison of the velocity streamlines, together with their corresponding pressure surface plots for the first eigenvale and different choices of $\mathbb{K}$. The skeleton of the geometry has been slightly superimposed on the pressure graph for better representation. }
	\label{fig:square2D-uh1-allcases}
\end{figure}

\begin{figure}[!hbt]\centering
\includegraphics[scale=0.43, trim=0cm 4cm 1.8cm 5cm,clip]{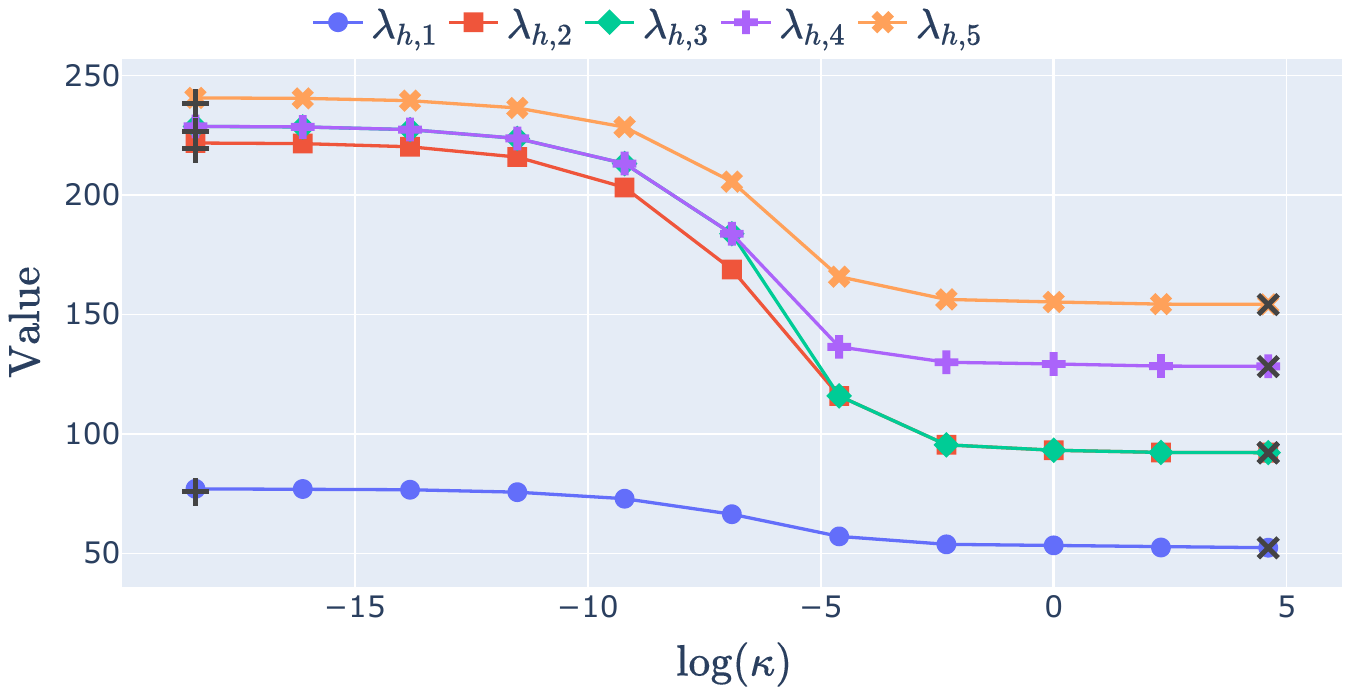}
\caption{Test \ref{subsec:test-cuadrado}. Spectrum changes on the first five computed eigenvalues for different values of the permeability parameter on $\Omega_D$ with $N=120$ compared with the Stokes eigenvalues on the square domain (\texttt{x} marker) and on the domain $\Omega\backslash\Omega_D$ ($\texttt{+}$  marker) with homogeneous boundary conditions.}
	\label{fig:convergence-evolution-square}
\end{figure}

\begin{figure}[!hbt]\centering
	\begin{minipage}{0.59\linewidth}\centering
		\includegraphics[scale=0.37,trim=0cm 0cm 2cm 2cm,clip]{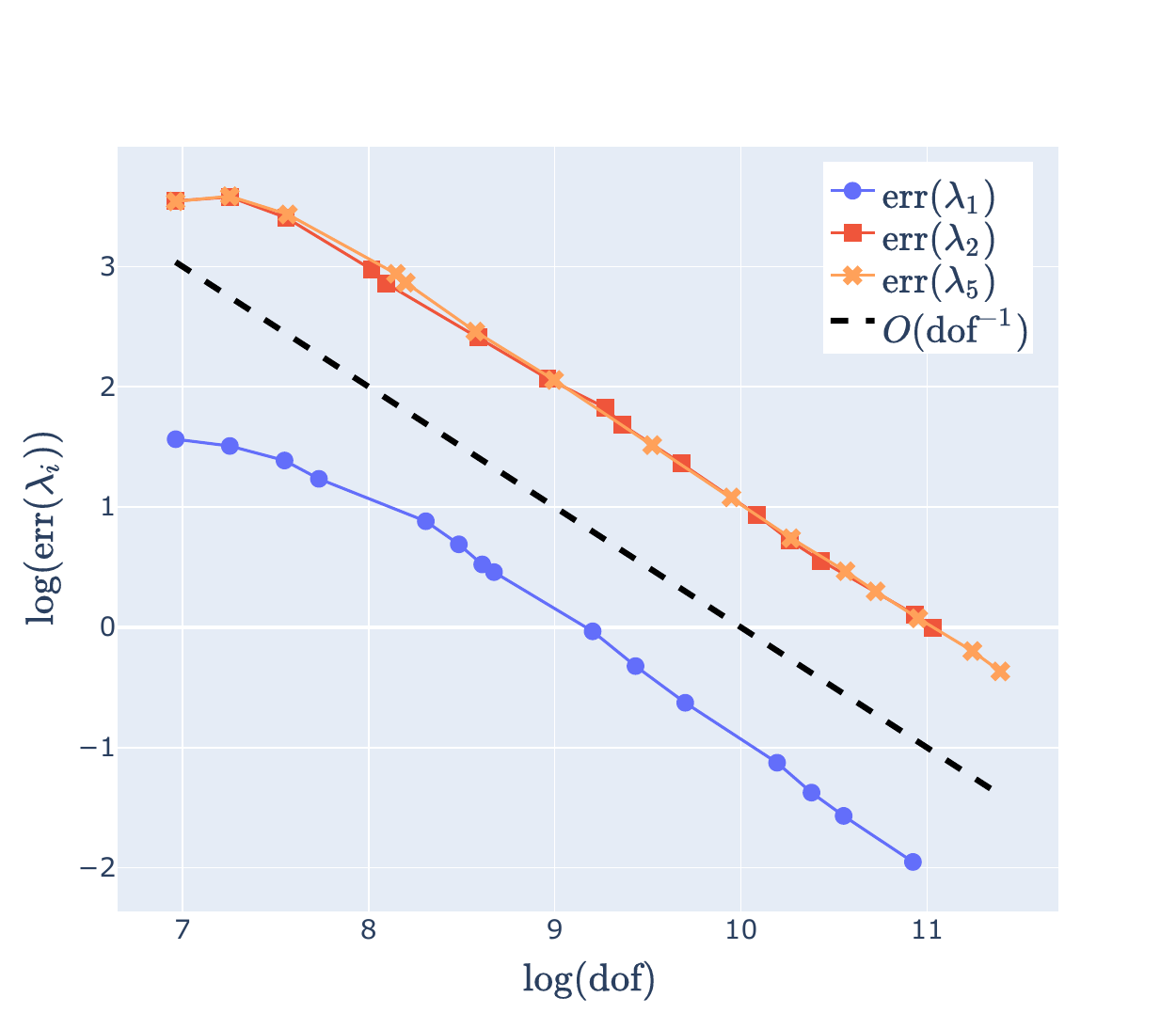}
	\end{minipage}
	\begin{minipage}{0.40\linewidth}\centering
		\includegraphics[scale=0.37,trim=0cm 0cm 2cm 2cm,clip]{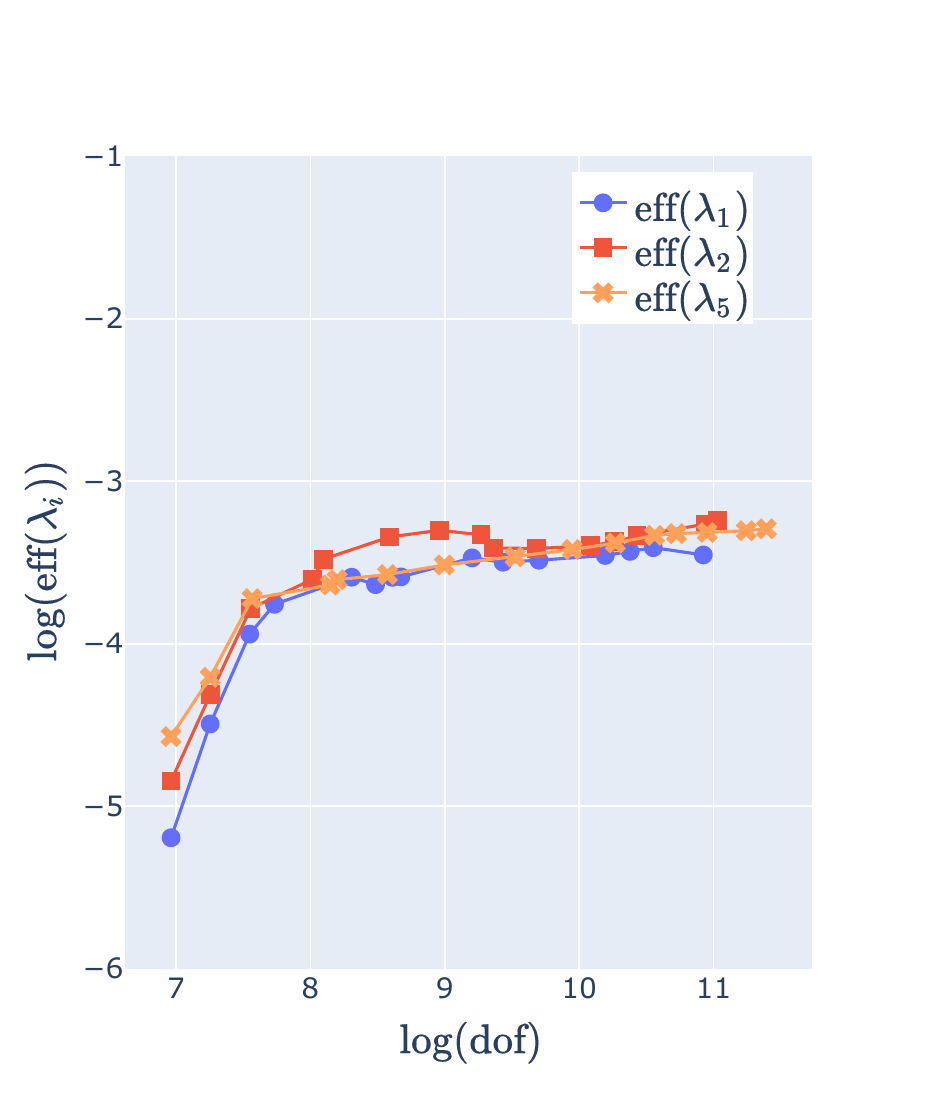}
	\end{minipage}
	\caption{Test \ref{subsec:test-cuadrado}. Error history in the adaptive refinements with mini elements for the first, second and fifth eigenvalue (left) together with their corresponding effectivity indexes (right).}
	\label{fig:error-eff-square2D}
\end{figure}

\begin{figure}[!hbt]\centering
	\begin{minipage}{0.32\linewidth}\centering
		{\footnotesize $\bu_{h,1},\mathbb{K}^{-1}\vert_{\Omega_D}=10^{5}\mathbb{I}$}\\
		\includegraphics[scale=0.12,trim=23.2cm 5cm 23.2cm 5cm,clip]{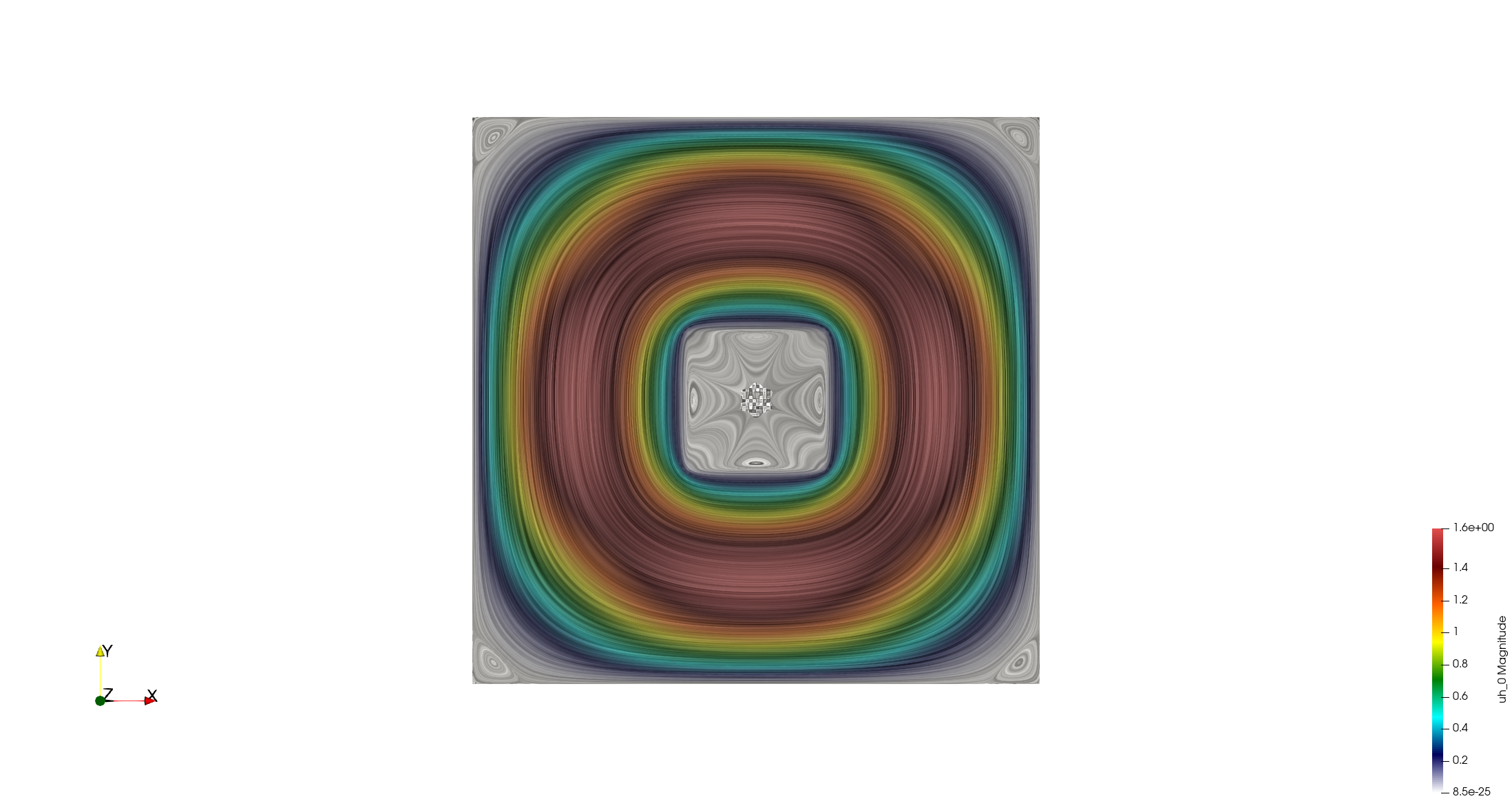}
	\end{minipage}
	\begin{minipage}{0.32\linewidth}\centering
		{\footnotesize $\bu_{h,2},\mathbb{K}^{-1}\vert_{\Omega_D}=10^{5}\mathbb{I}$}\\
		\includegraphics[scale=0.12,trim=23.2cm 5cm 23.2cm 5cm,clip]{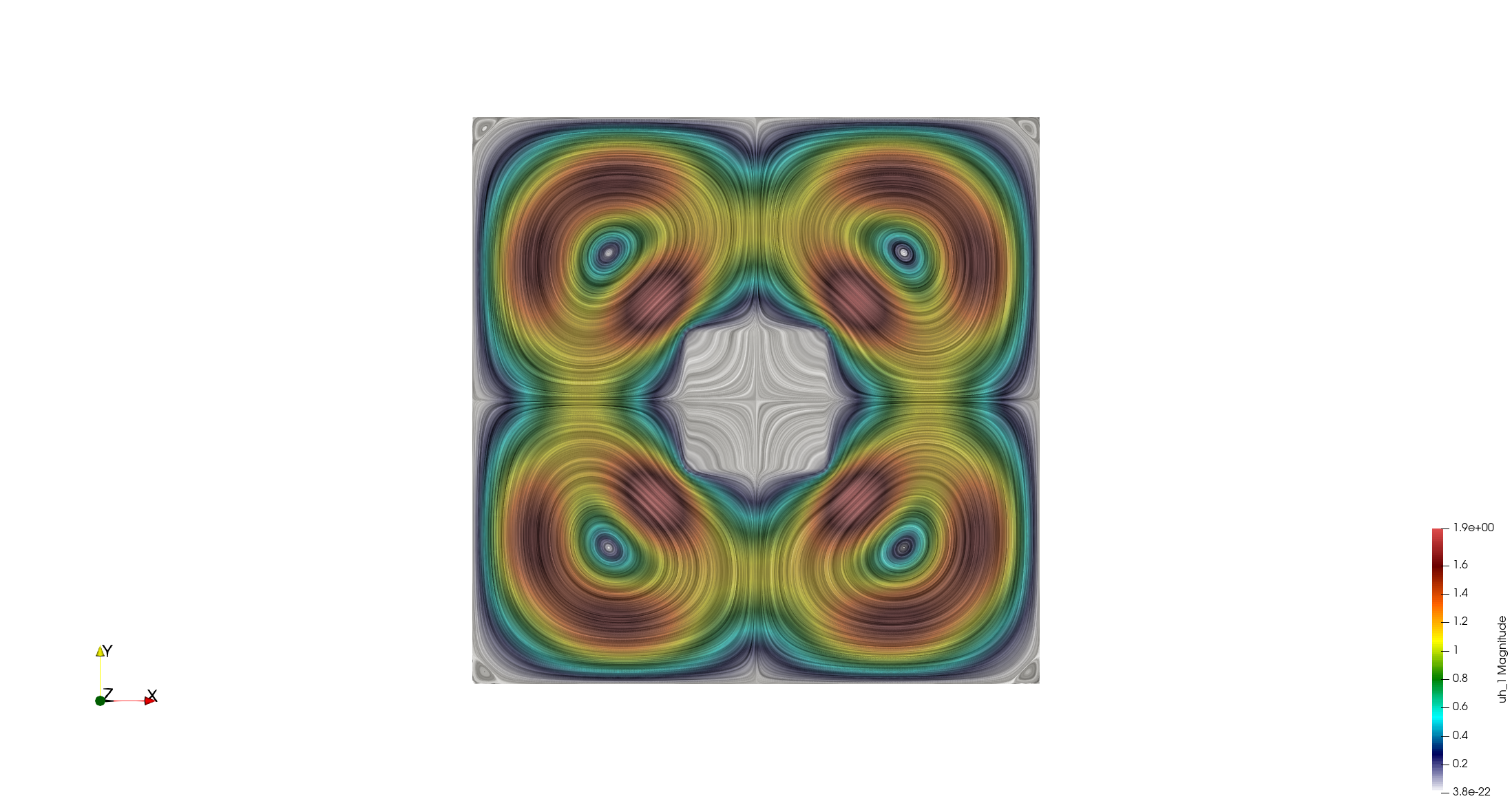}
	\end{minipage}
	\begin{minipage}{0.32\linewidth}\centering
		{\footnotesize $\bu_{h,5},\mathbb{K}^{-1}\vert_{\Omega_D}=10^{5}\mathbb{I}$}\\
		\includegraphics[scale=0.12,trim=23.2cm 5cm 23.2cm 5cm,clip]{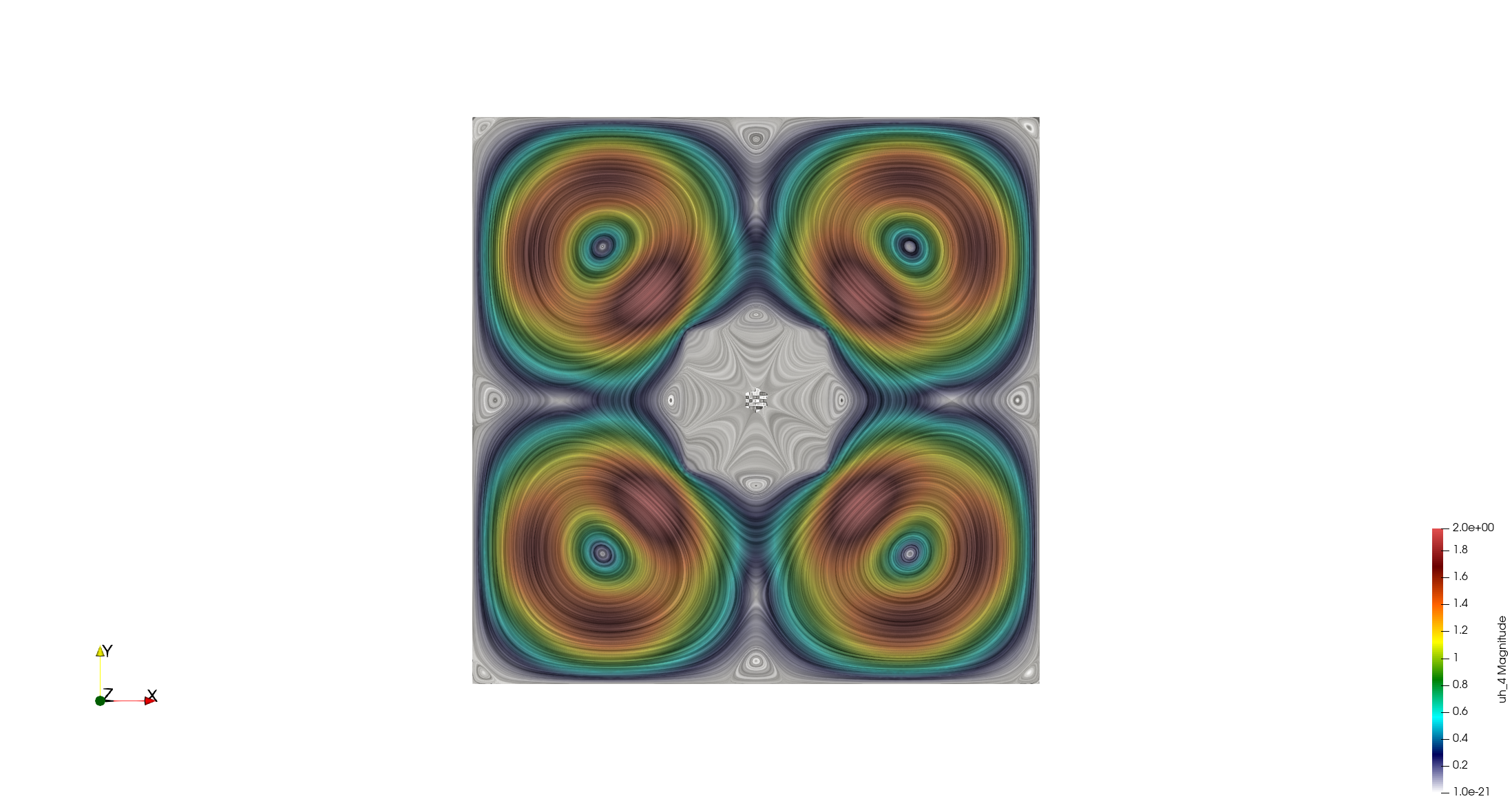}
	\end{minipage}\\
	\begin{minipage}{0.32\linewidth}\centering
		{\footnotesize $p_{h,1}, \mathbb{K}^{-1}\vert_{\Omega_D}=10^{5}\mathbb{I}$}\\
		\includegraphics[scale=0.12,trim=23.2cm 5cm 23.2cm 5cm,clip]{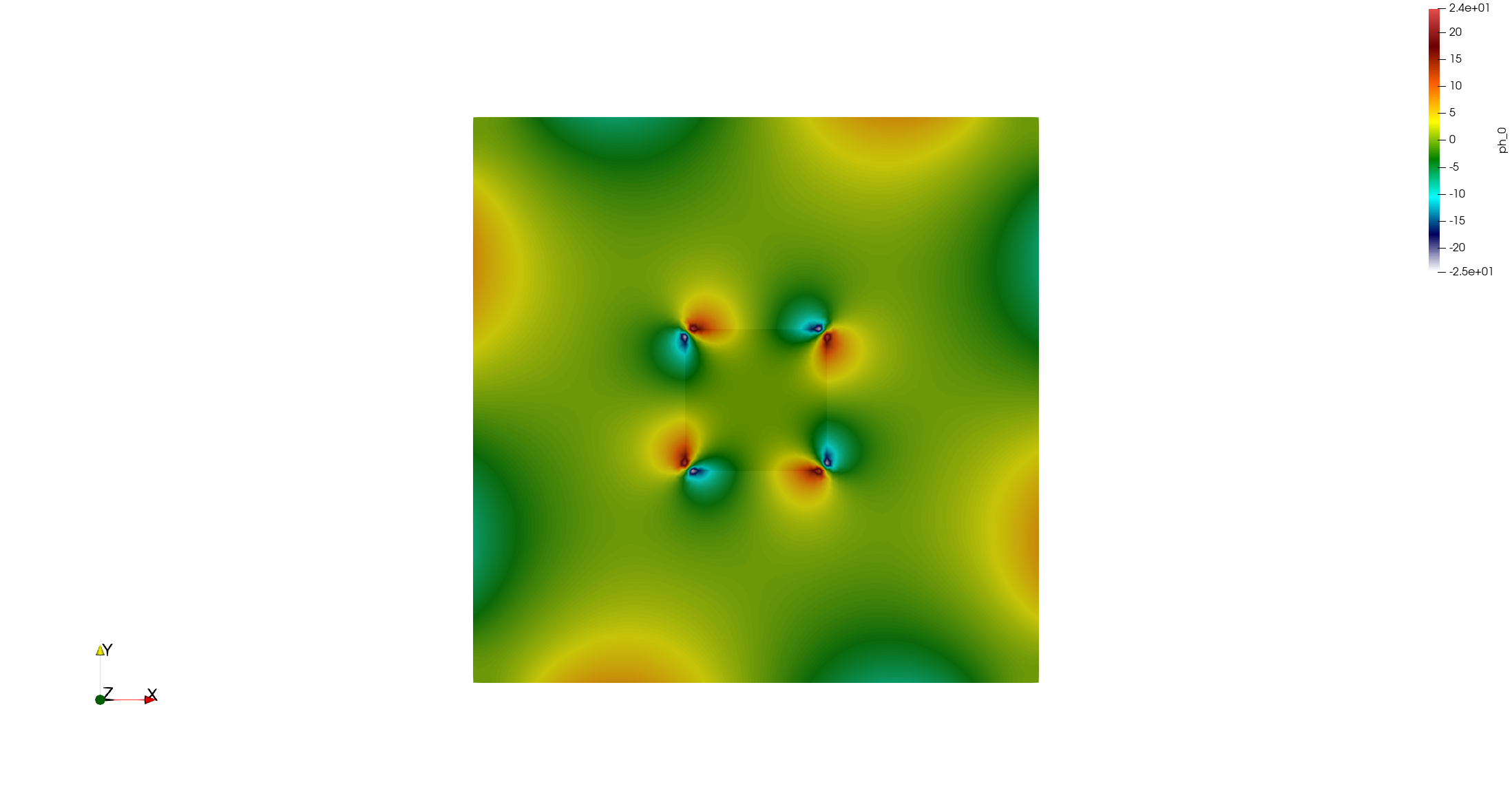}
	\end{minipage}
	\begin{minipage}{0.32\linewidth}\centering
		{\footnotesize $p_{h,2}, \mathbb{K}^{-1}\vert_{\Omega_D}=10^{5}\mathbb{I}$}\\	
		\includegraphics[scale=0.12,trim=23.2cm 5cm 23.2cm 5cm,clip]{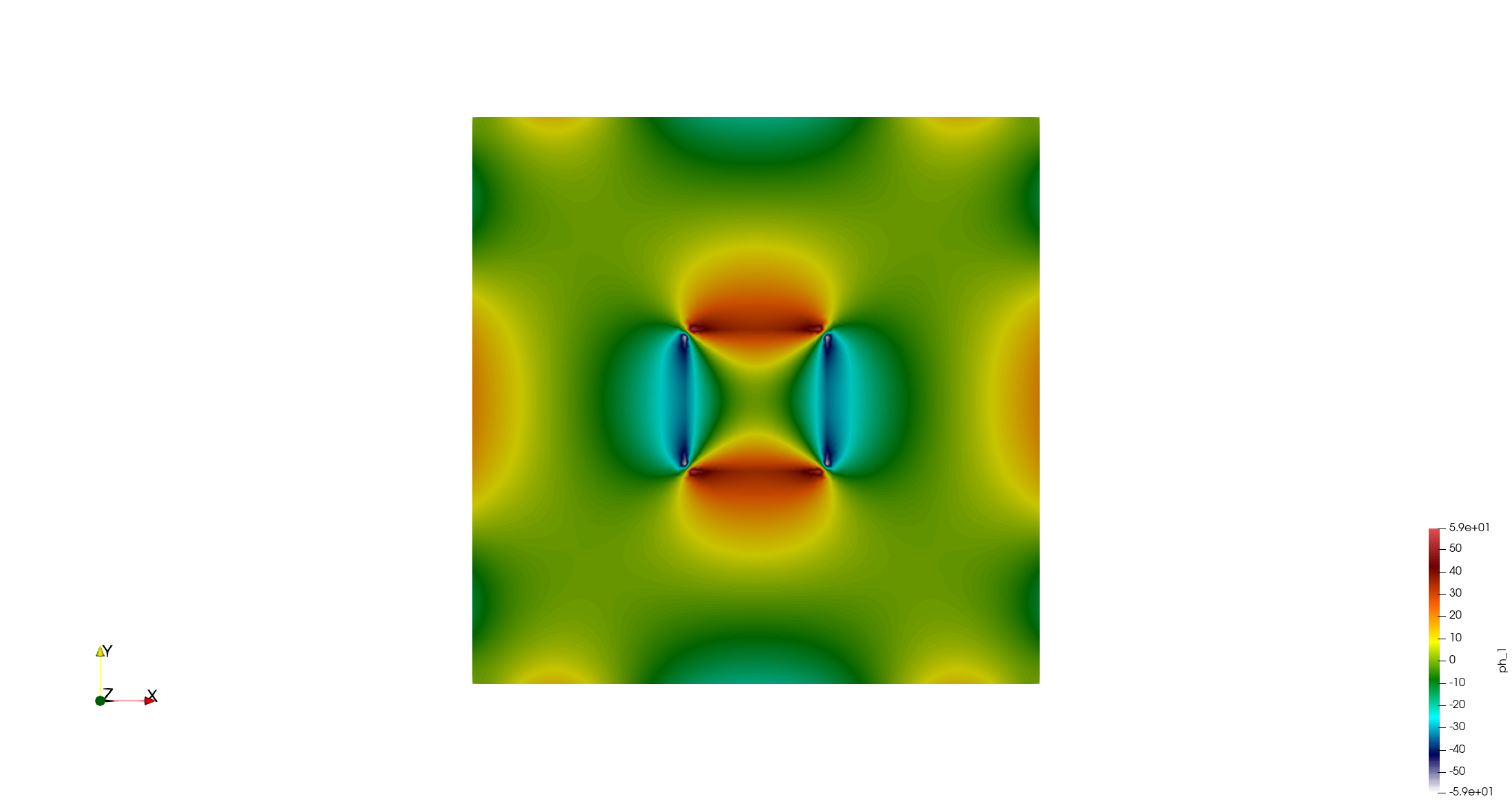}
	\end{minipage}
	\begin{minipage}{0.32\linewidth}\centering
		{\footnotesize $p_{h,5}, \mathbb{K}^{-1}\vert_{\Omega_D}=10^{5}\mathbb{I}$}\\
		\includegraphics[scale=0.12,trim=23.2cm 5cm 23.2cm 5cm,clip]{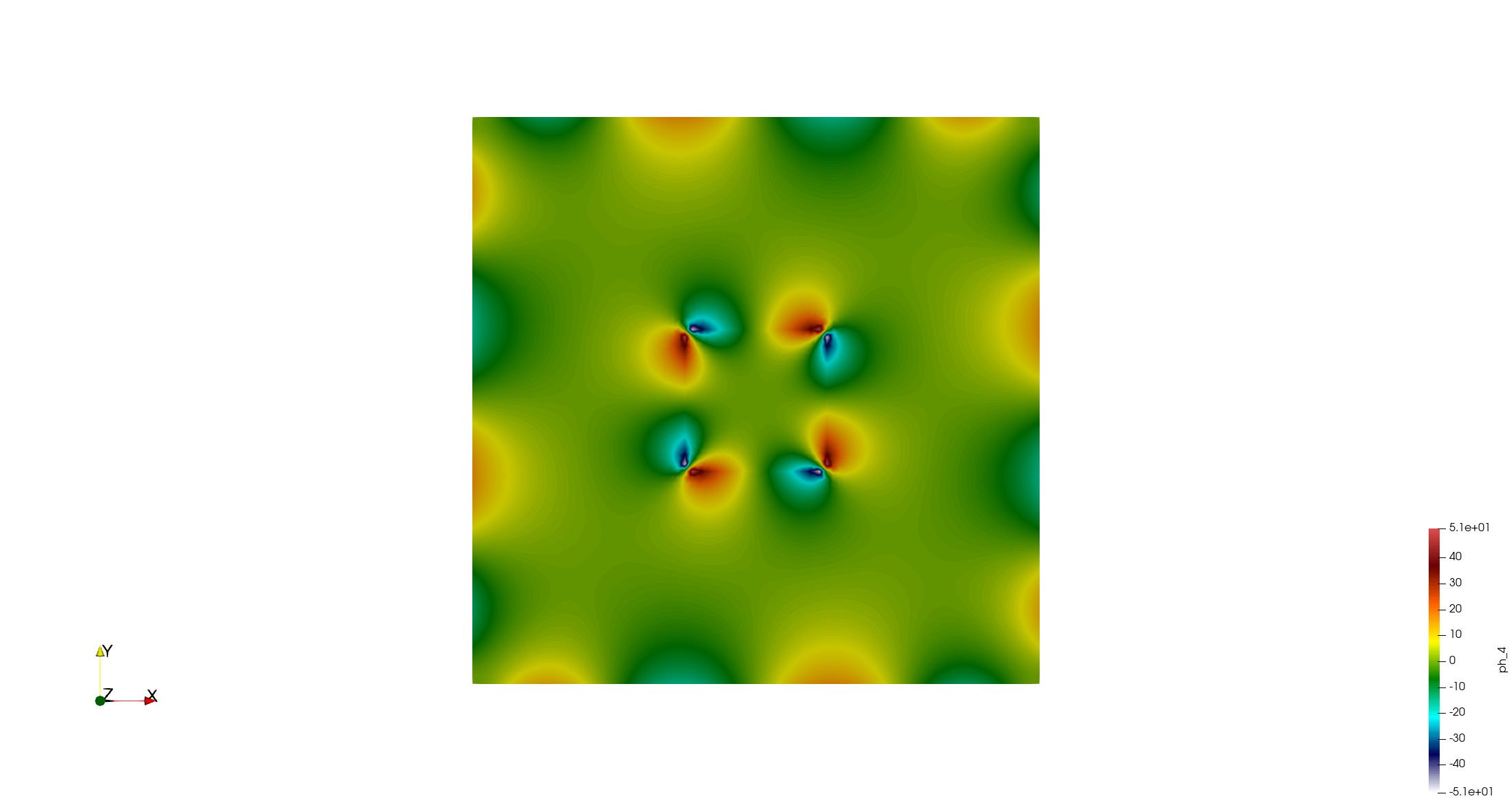}
	\end{minipage}\\
	\begin{minipage}{0.32\linewidth}\centering
		{\footnotesize $\mathbb{K}^{-1}\vert_{\Omega_D}=10^{5}\mathbb{I}, \texttt{dof}=38245$}\\
		\includegraphics[scale=0.12,trim=23.2cm 5cm 23.2cm 5cm,clip]{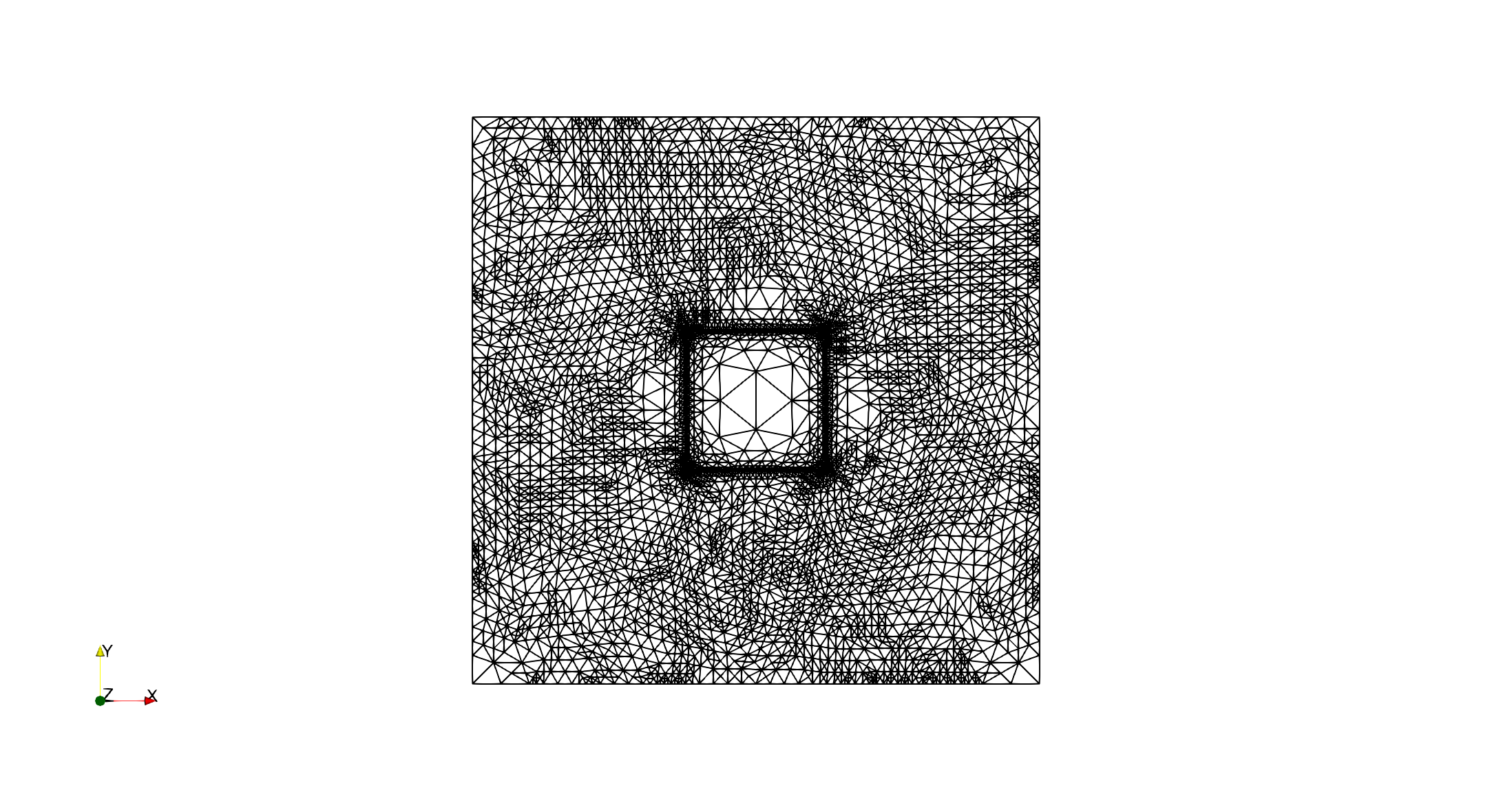}
	\end{minipage}
	\begin{minipage}{0.32\linewidth}\centering
		{\footnotesize $\mathbb{K}^{-1}\vert_{\Omega_D}=10^{5}\mathbb{I},\texttt{dof}=56153$}\\	
		\includegraphics[scale=0.12,trim=23.2cm 5cm 23.2cm 5cm,clip]{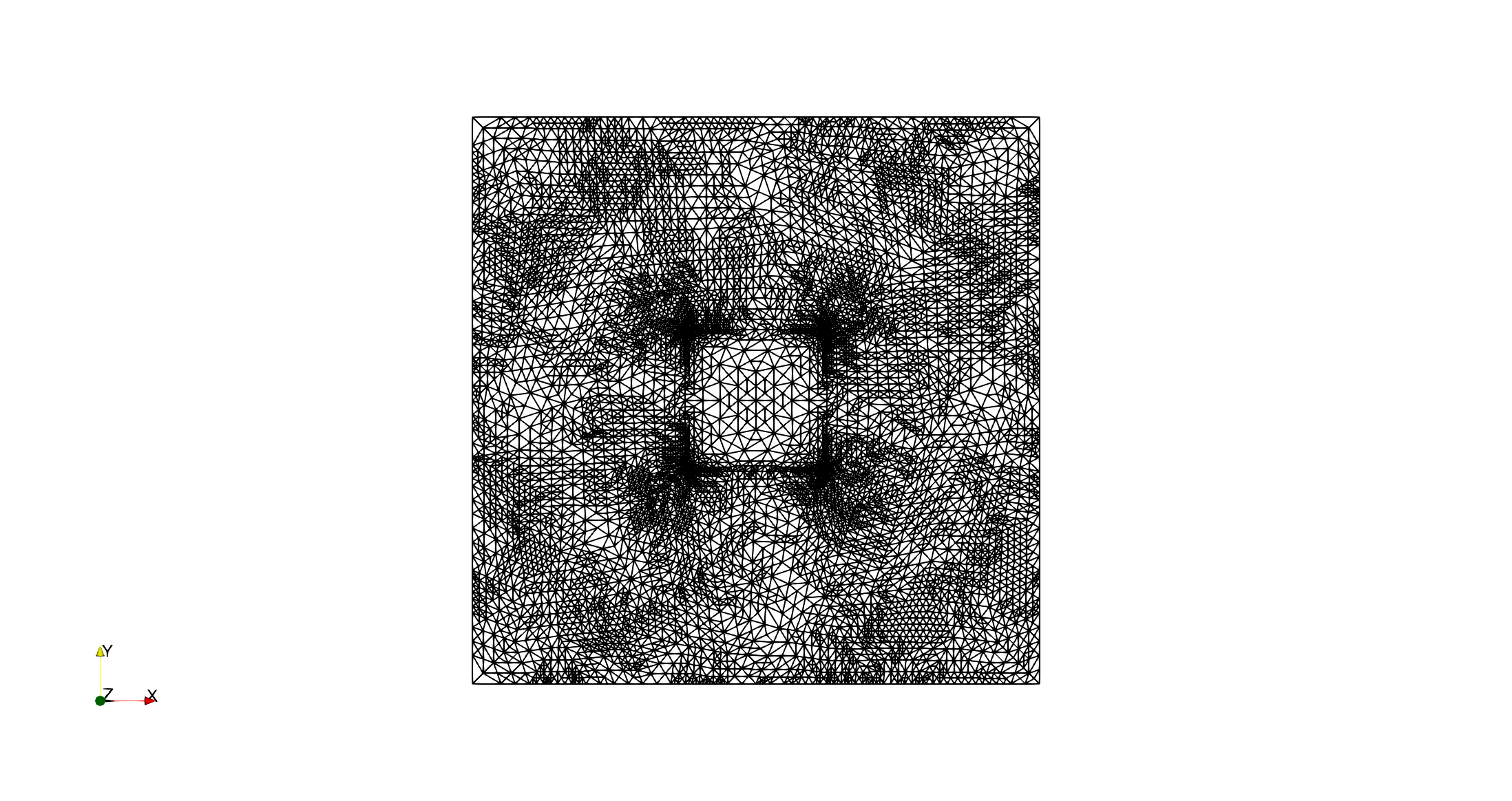}
	\end{minipage}
	\begin{minipage}{0.32\linewidth}\centering
		{\footnotesize $\mathbb{K}^{-1}\vert_{\Omega_D}=10^{5}\mathbb{I}, \texttt{dof}=76375$}\\
		\includegraphics[scale=0.12,trim=23.2cm 5cm 23.2cm 5cm,clip]{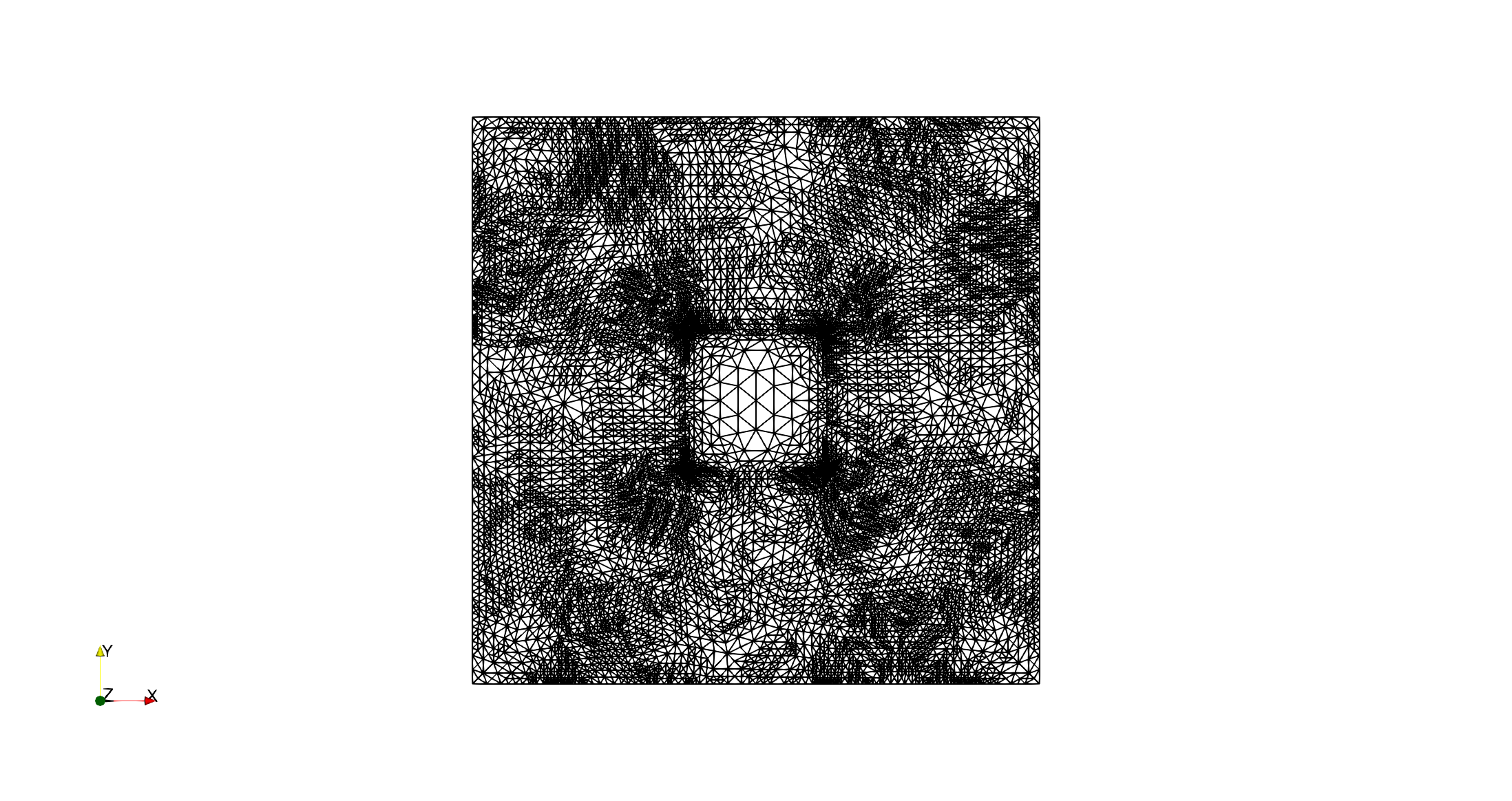}
	\end{minipage}\\
	\caption{Test \ref{subsec:test-cuadrado}. Comparison of the computed eigenfunctions for the first, second and fifth eigenvalue together with their corresponding adaptive meshes on the 14th iteration.}
	\label{fig:square2D-adaptive}
\end{figure}

\subsection{2D Lshaped porous domain with mixed boundary conditions}\label{subsec:lshape2D}
In this experiment we aim to test the scheme when solving the eigenvalue problem on a two-dimensional domain with a singularity and with different permeability zones. To this end we consider the Lshape domain given by $\Omega:=(0,1)^2\backslash\left((0.5,0)\times(1,0.5)\right)$ with chessboard-like configuration for $\Omega_S$ and $\Omega_D$. A sample of the meshed geometry is depicted in Figure \ref{fig:lshape-sample-domain}. Non-slip and \textit{do-nothing} boundary conditions are assumed on $\Gamma_1$ and  $\Gamma_2$, respectively. We take $\mathbb{K}$ such that $\mathbb{K}^{-1}=\boldsymbol{0}$ on $\Omega_S$, while $\mathbb{K}^{-1}=10^3\mathbb{I}$ on $\Omega_D$.

Convergence results for uniform refinements in the first four eigenvalues are given in Table \ref{table-lshape2D-TH-MINI}. Because of the singularity, we expect that the first eigenvalue converge with suboptimal rate. However, we also note that the fourth eigenvalue is singular. The second and third eigenvalues are regular up to $\mathcal{O}(h^{3.5})$ approximately. Streamlines of the velocity and surface plots of pressure for the computed eigenfunctions are portrayed in Figure \ref{fig:lshape2d:eigenmodes}. Note that the fluid tends to move in the zones that minimize the energy of the system. That is, it tends to moves in the Stokes zones, where there is less resistance to movement.

Focusing our attention on the two singular eigenmodes. Note that high pressure gradients are present on Figure \ref{fig:lshape2d:eigenmodes}.  When performing adaptive refinements, we present the error and efficiency results, for mini and Taylor-Hood elements, on Figure \ref{fig:error-eff-lshape2d}. For mini elements we have performed 15 adaptive iterations, while 20 iterations were considered for Taylor-Hood. Here we observe that optimal rates are attained in comparison with uniform refinements. Also, the estimator stays properly bounded. Some oscilation in the first iterations were observed with Taylor-Hood elements, but keeping the error curves always below of those of mini elements. When observing the adaptive meshes, depicted in Figures \ref{fig:lshape2d-meshes}-\ref{fig:lshape2d-meshes_TH}, we note that the estimator detects the singularity, as well as the pressure gradients between Stokes and Darcy subdomains, and refines accordingly. This is clearly observable in the first eigenmode. Also, we note that a considerable lower of elements are marked and refined to achieve the optimal rate $\mathcal{O}(h^4)$ with Taylor-Hood elements in the same number of iterations.

\begin{figure}[!hbt]\centering
	\includegraphics[scale=0.15]{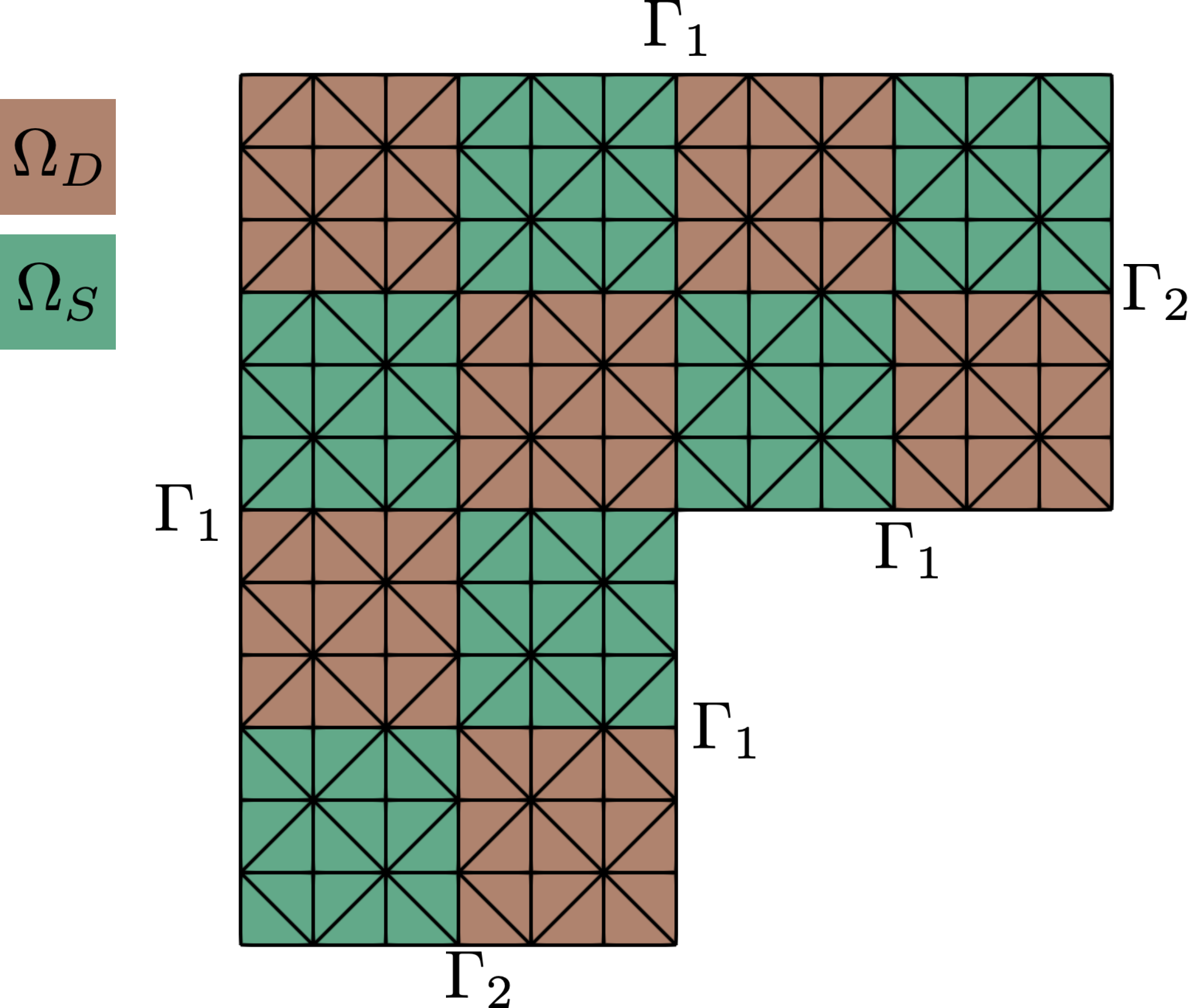}
	\caption{Test \ref{subsec:lshape2D}. Sample geometry of an Lshape domain with alternated permeability region and $N=10$.}
	\label{fig:lshape-sample-domain}
\end{figure}

\begin{table}[hbt!]
	\centering 
	{\footnotesize
		\begin{center}
			\caption{Test \ref{subsec:lshape2D}. Convergence behavior of the first four computed eigenvalues on the 2D Lshaped geometry with mixed boundary conditions and $\mathbb{K}^{-1} := 10^{3} \mathbb{I}$ in $\Omega_D$. }
			\begin{tabular}{|c c c c |c| c|}
				\hline
				\hline
				$N=20$             &  $N=40$         &   $N=60$         & $N=80$ & Order & $\lambda_{\text{extr}}$ \\
				\hline
				\multicolumn{6}{|c|}{mini-element}\\
				\hline
%
  309.8523  &   304.9781  &   303.8658  &   303.3933  & 1.79 &   302.7206  \\
401.8870  &   380.5371  &   376.4772  &   375.0585  & 2.15 &   373.3541  \\
414.5703  &   392.5575  &   388.3612  &   386.8954  & 2.15 &   385.1414  \\
489.8924  &   479.1417  &   476.4018  &   475.1874  & 1.59 &   473.1085  \\
				\hline
				\multicolumn{6}{|c|}{Taylor-Hood}\\
				\hline	
301.5767  &   302.3611  &   302.4922  &   302.5450  & 1.70 &   302.6652  \\
373.6313  &   373.2335  &   373.1915  &   373.1800  & 3.00 &   373.1708  \\
385.3491  &   385.0002  &   384.9620  &   384.9517  & 3.15 &   384.9452  \\
466.6914  &   470.7840  &   471.6521  &   472.0178  & 1.53 &   472.8589  \\
				\hline
				\hline             
			\end{tabular}
	\end{center}}
	\smallskip
	\label{table-lshape2D-TH-MINI}
\end{table}

\begin{figure}[!hbt]\centering
\begin{minipage}{0.24\linewidth}\centering
	{\footnotesize $\bu_{h,1}$}\\
	\includegraphics[scale=0.08,trim=20.2cm 2cm 20.2cm 2cm,clip]{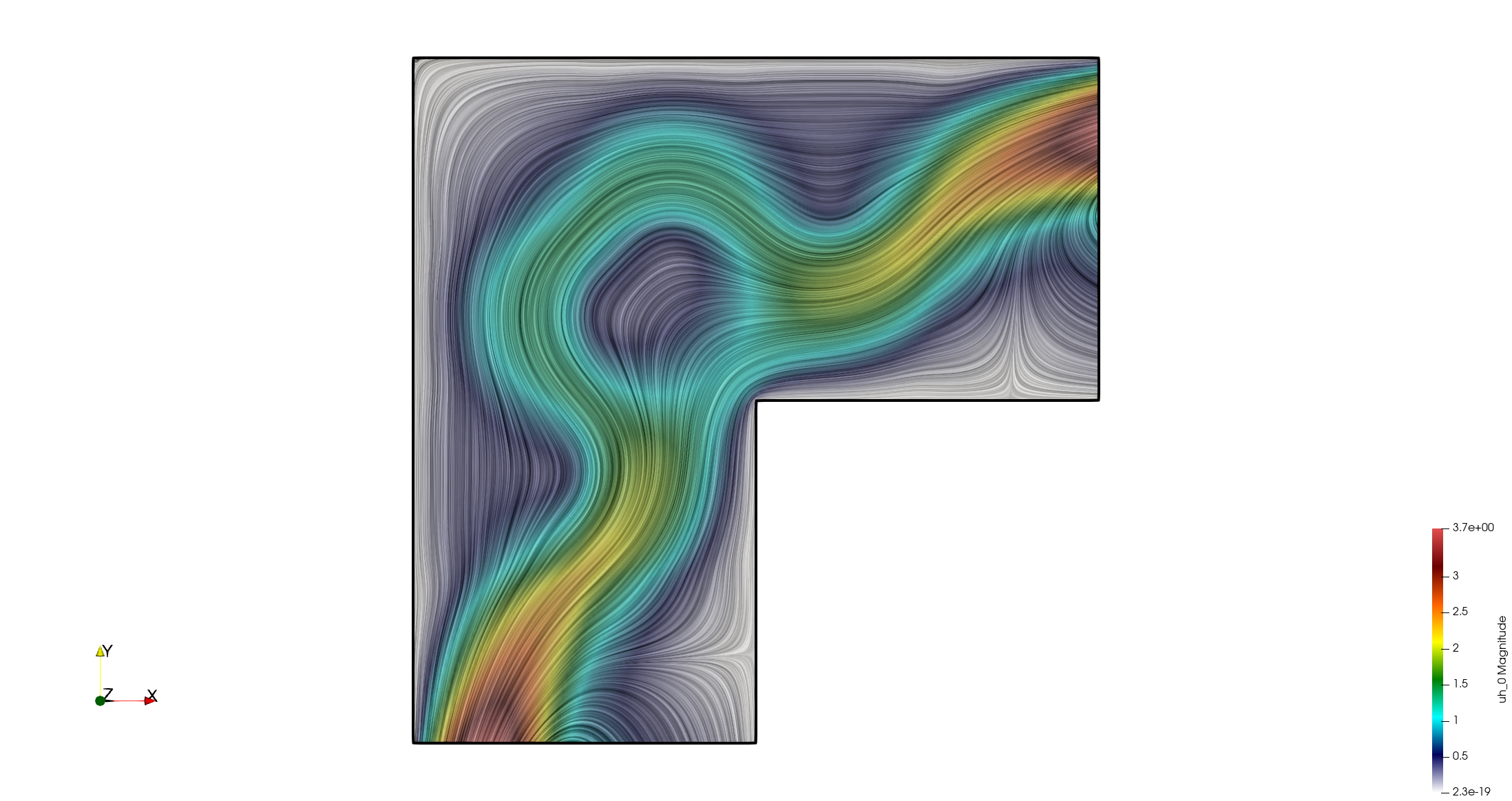}
\end{minipage}
\begin{minipage}{0.24\linewidth}\centering
	{\footnotesize $\bu_{h,2}$}\\
	\includegraphics[scale=0.08,trim=20.2cm 2cm 20.2cm 2cm,clip]{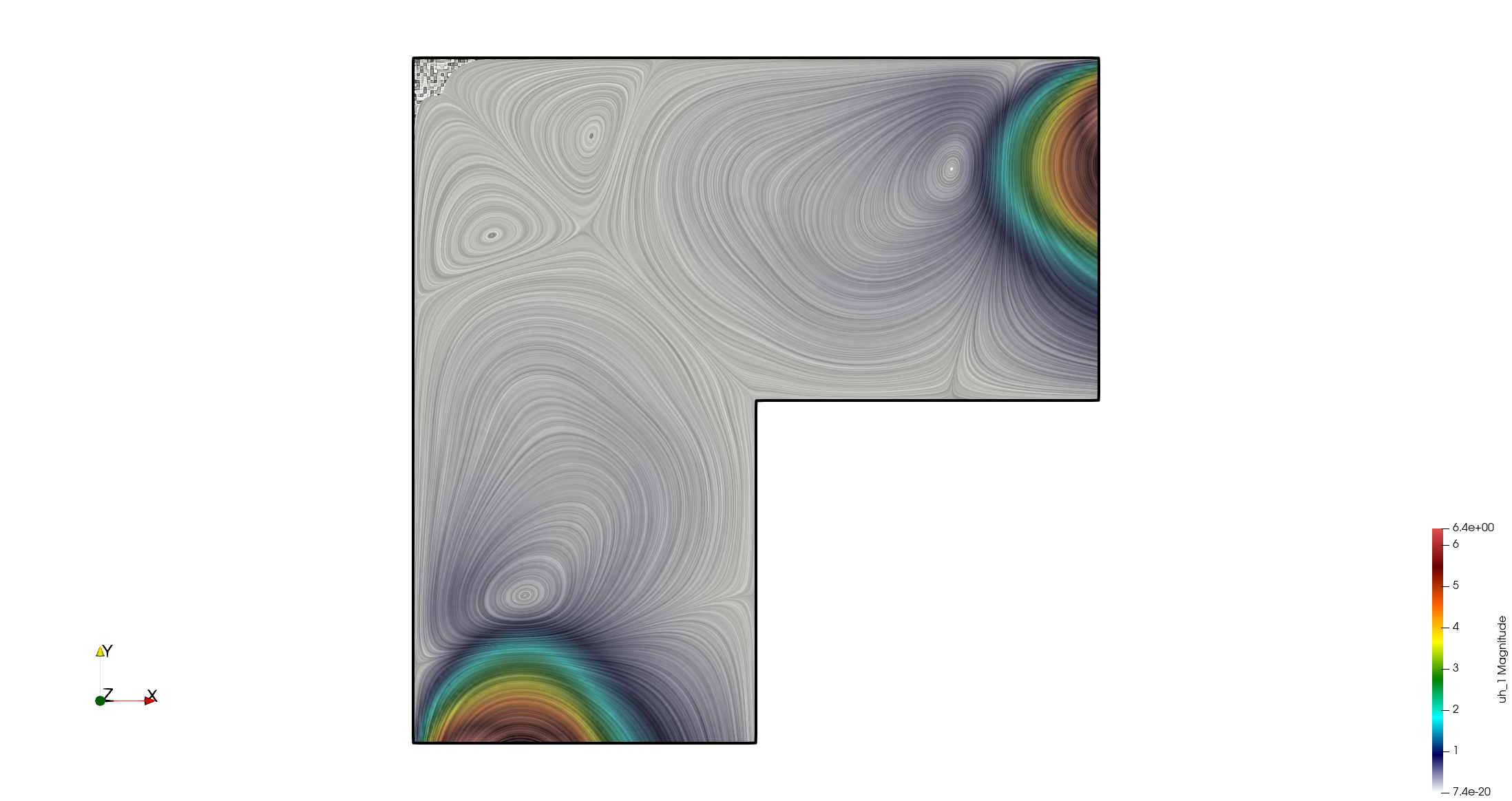}
\end{minipage}
\begin{minipage}{0.24\linewidth}\centering
	{\footnotesize $\bu_{h,3}$}\\
	\includegraphics[scale=0.08,trim=20.2cm 2cm 20.2cm 2cm,clip]{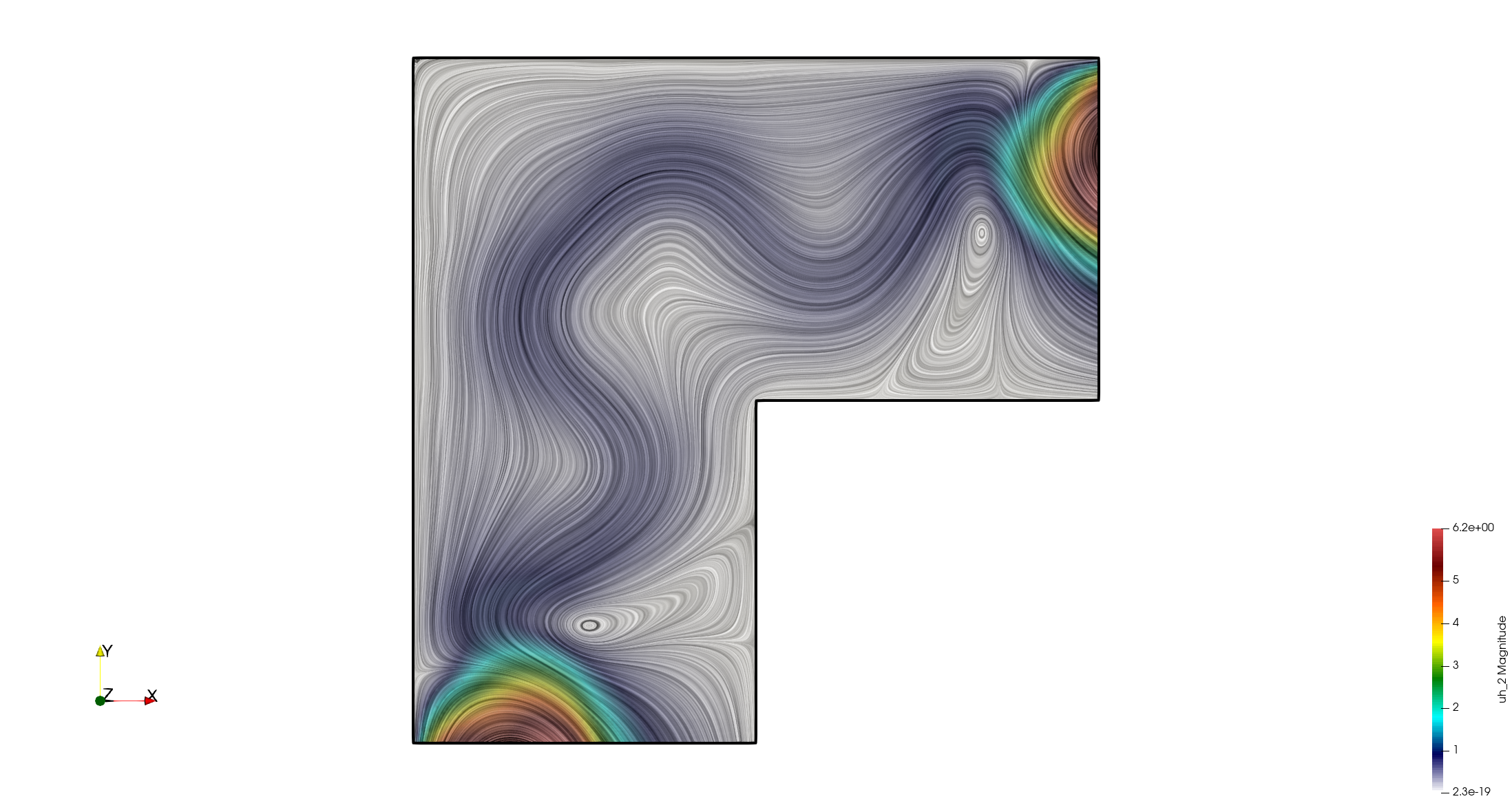}
\end{minipage}
\begin{minipage}{0.24\linewidth}\centering
	{\footnotesize $\bu_{h,4}$}\\
	\includegraphics[scale=0.08,trim=20.2cm 2cm 20.2cm 2cm,clip]{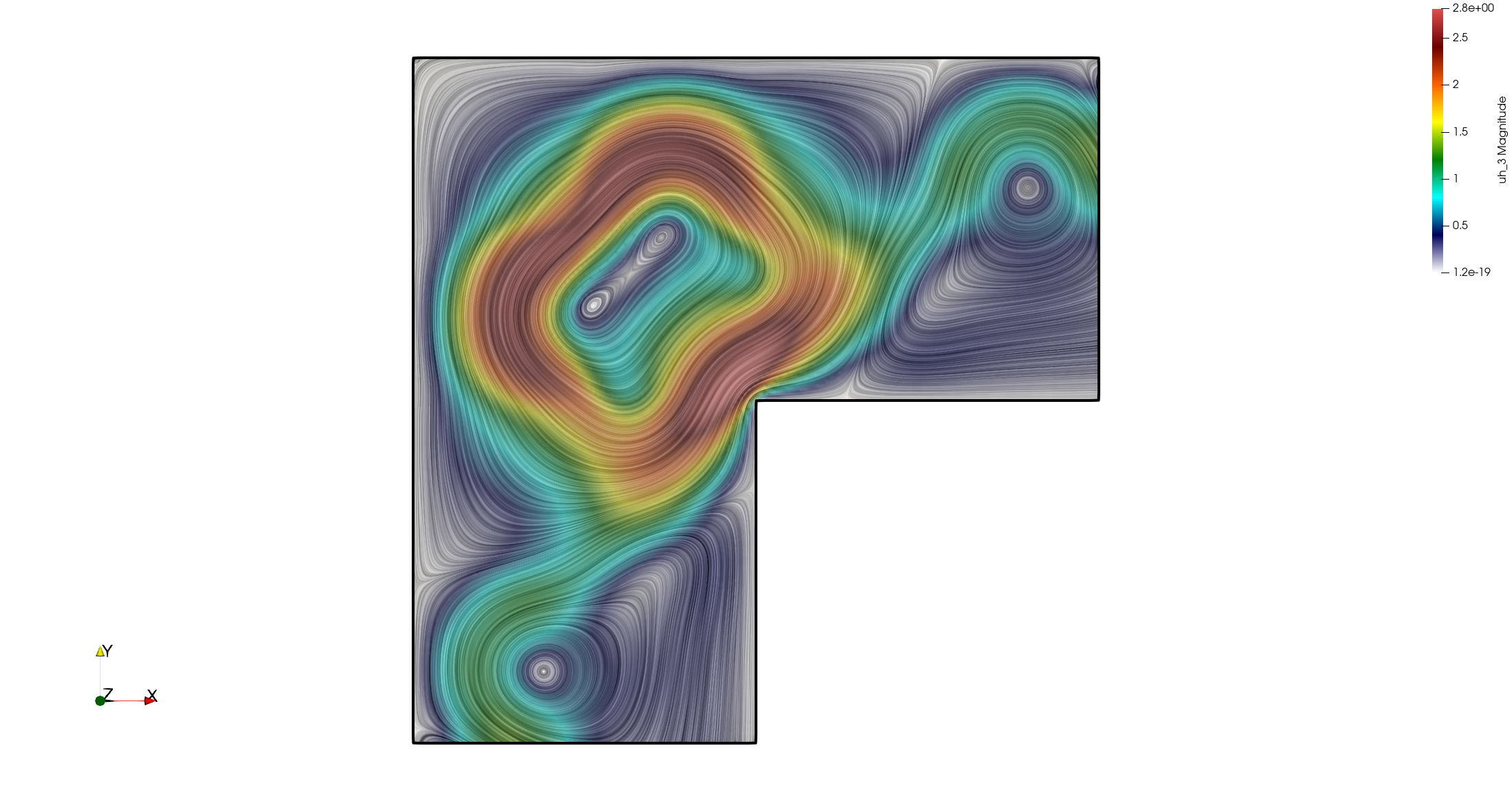}
\end{minipage}\\
\begin{minipage}{0.24\linewidth}\centering
	{\footnotesize $p_{h,1}$}\\
	\includegraphics[scale=0.08,trim=20.2cm 0cm 20.2cm 2cm,clip]{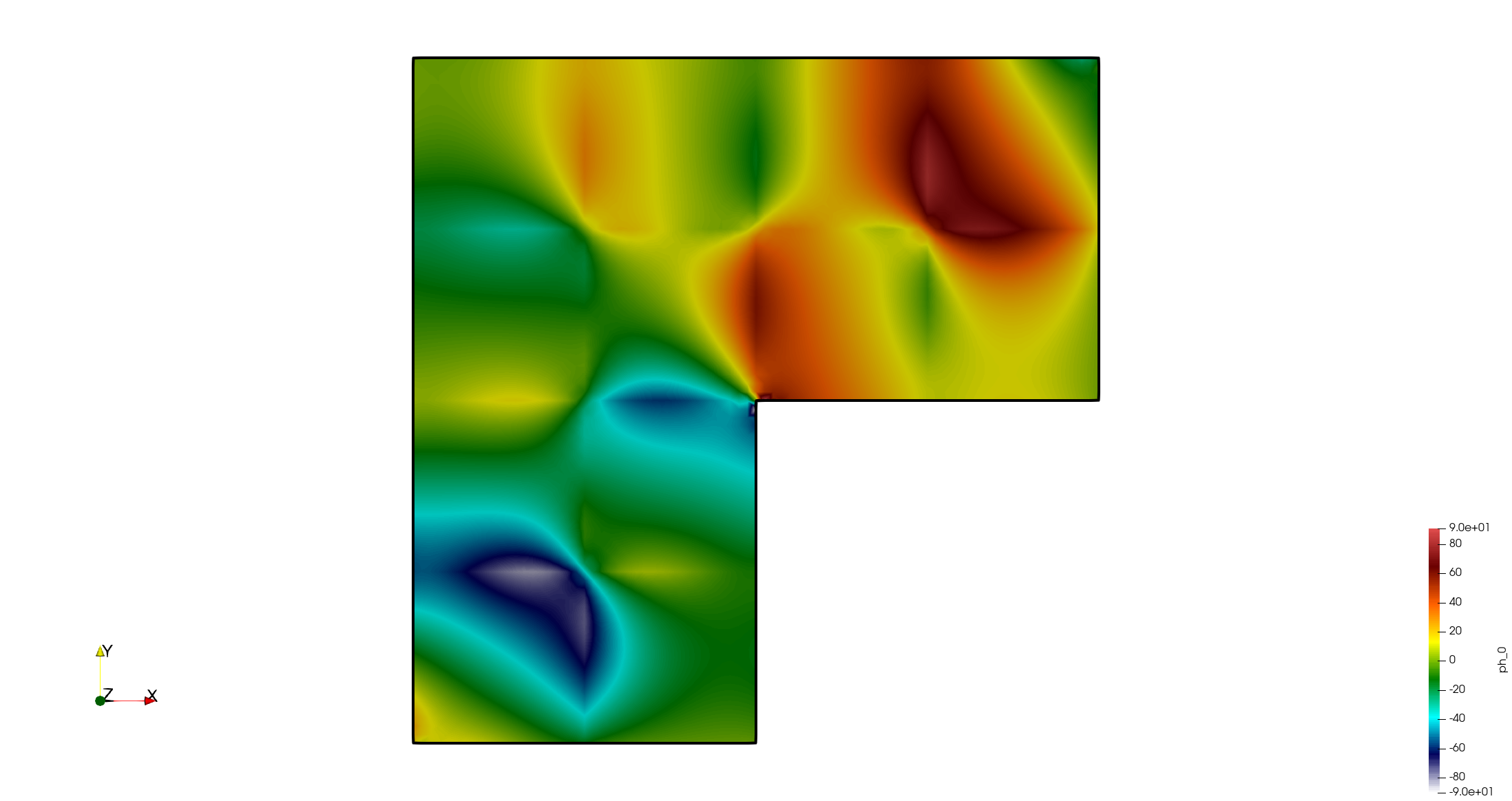}
\end{minipage}
\begin{minipage}{0.24\linewidth}\centering
	{\footnotesize $p_{h,2}$}\\
	\includegraphics[scale=0.08,trim=20.2cm 0cm 20.2cm 2cm,clip]{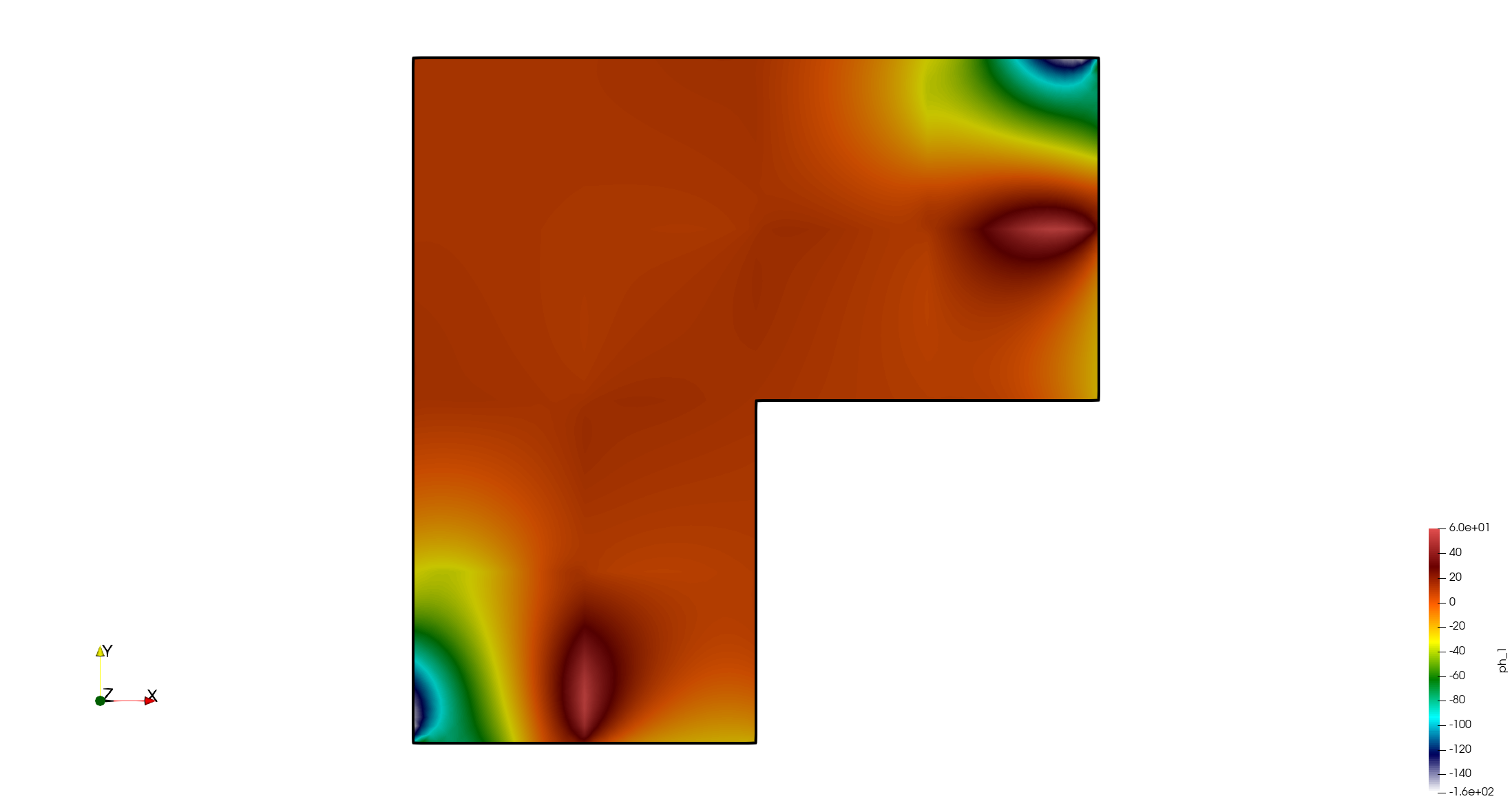}
\end{minipage}
\begin{minipage}{0.24\linewidth}\centering
	{\footnotesize $p_{h,3}$}\\
	\includegraphics[scale=0.08,trim=20.2cm 0cm 20.2cm 2cm,clip]{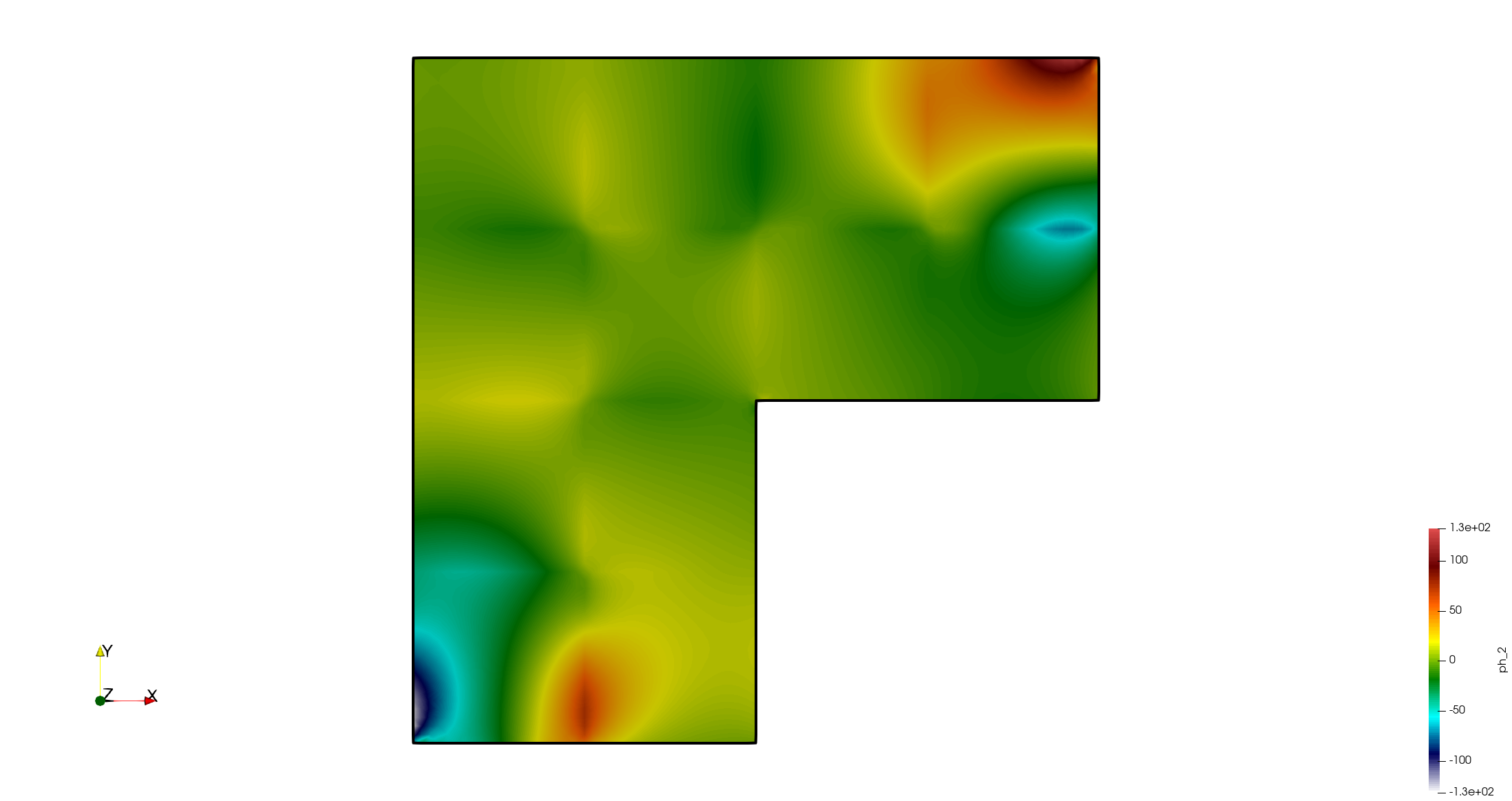}
\end{minipage}
\begin{minipage}{0.24\linewidth}\centering
	{\footnotesize $p_{h,4}$}\\
	\includegraphics[scale=0.08,trim=20.2cm 0cm 20.2cm 2cm,clip]{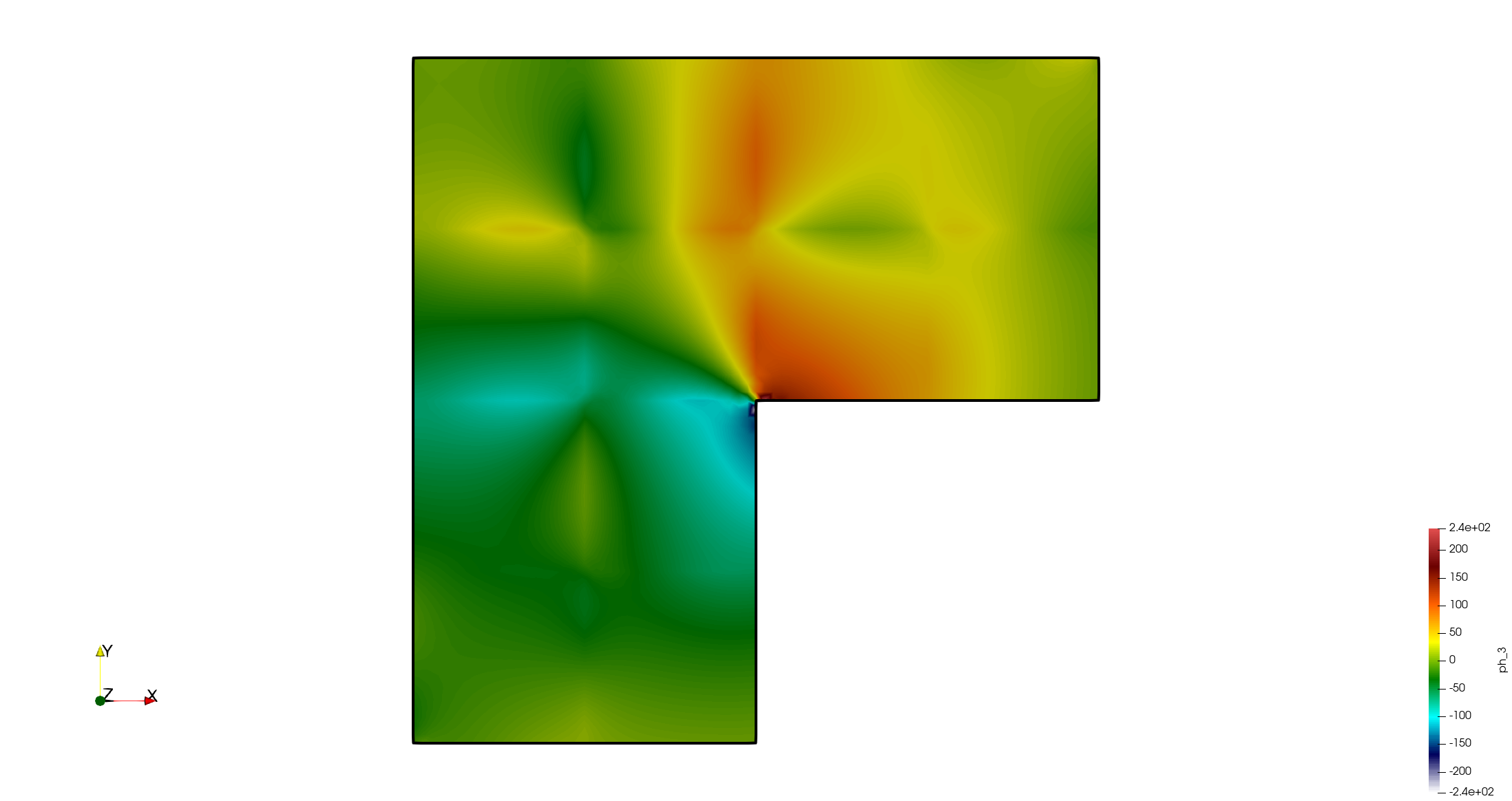}
\end{minipage}\\
\caption{Test \ref{subsec:lshape2D}. Eigenmodes of the first fourth computed eigenvalues on the Lshaped geometry with mixed boundary conditions and $\mathbb{K}^{-1}:=10^{3}\mathbb{I}$ on $\Omega_D$.  }
\label{fig:lshape2d:eigenmodes}
\end{figure}

\begin{figure}[!hbt]\centering
	\begin{minipage}{0.59\linewidth}\centering
		\includegraphics[scale=0.37,trim=0cm 0cm 2cm 2cm,clip]{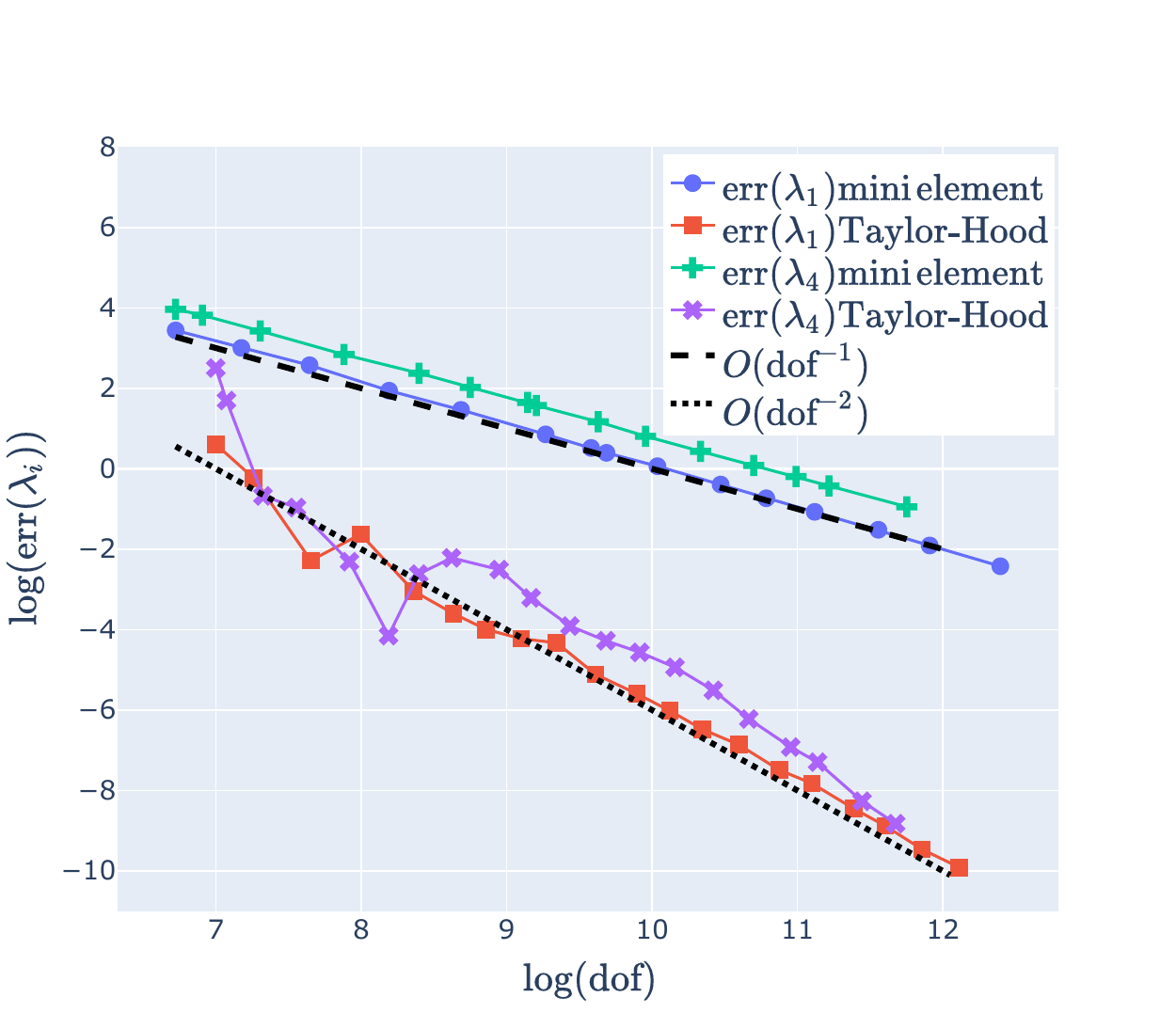}
	\end{minipage}
	\begin{minipage}{0.40\linewidth}\centering
		\includegraphics[scale=0.37,trim=0cm 0cm 2cm 2cm,clip]{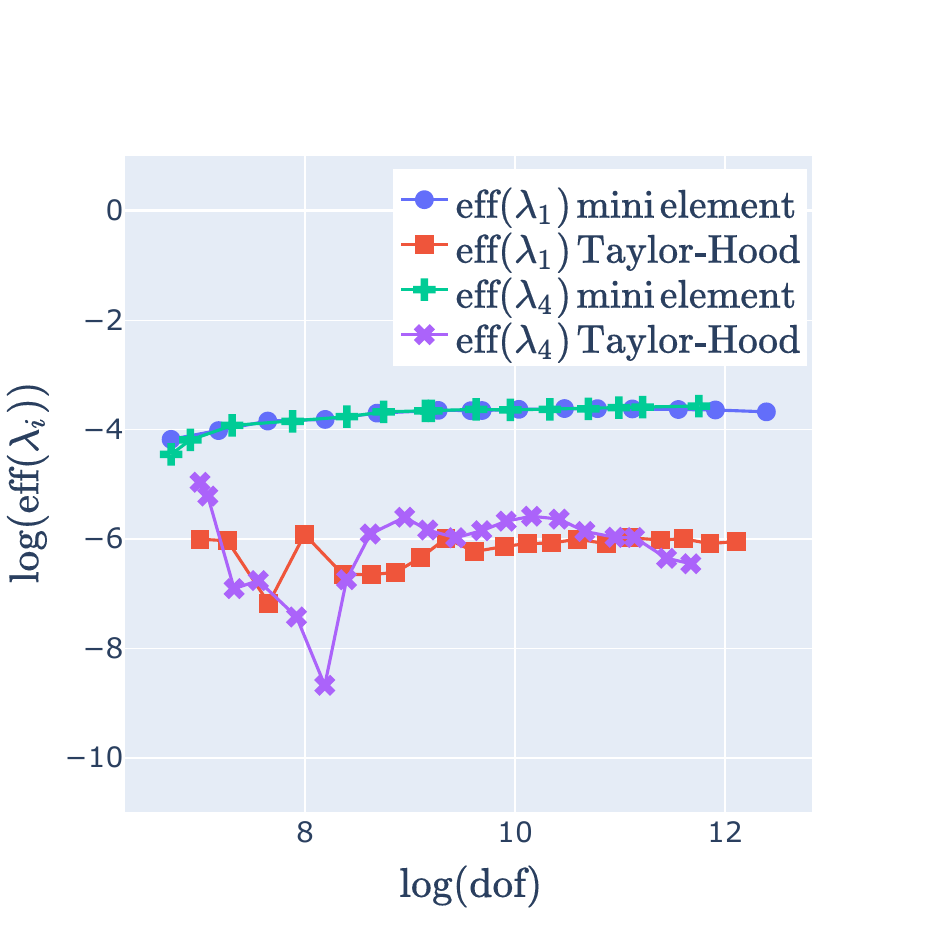}
	\end{minipage}
	\caption{Test \ref{subsec:lshape2D}. Error history for the uniform and adaptive refinements for the first and fourth eigenvalue (left) together with their corresponding effectivity indexes (right).}
	\label{fig:error-eff-lshape2d}
\end{figure}

\begin{figure}[!hbt]\centering
	\begin{minipage}{0.32\linewidth}\centering
	{\footnotesize $\lambda_{h,1}, \texttt{dof}=16089$ iter=8}\\
	\includegraphics[scale=0.1,trim=20.2cm 0cm 20.2cm 2cm,clip]{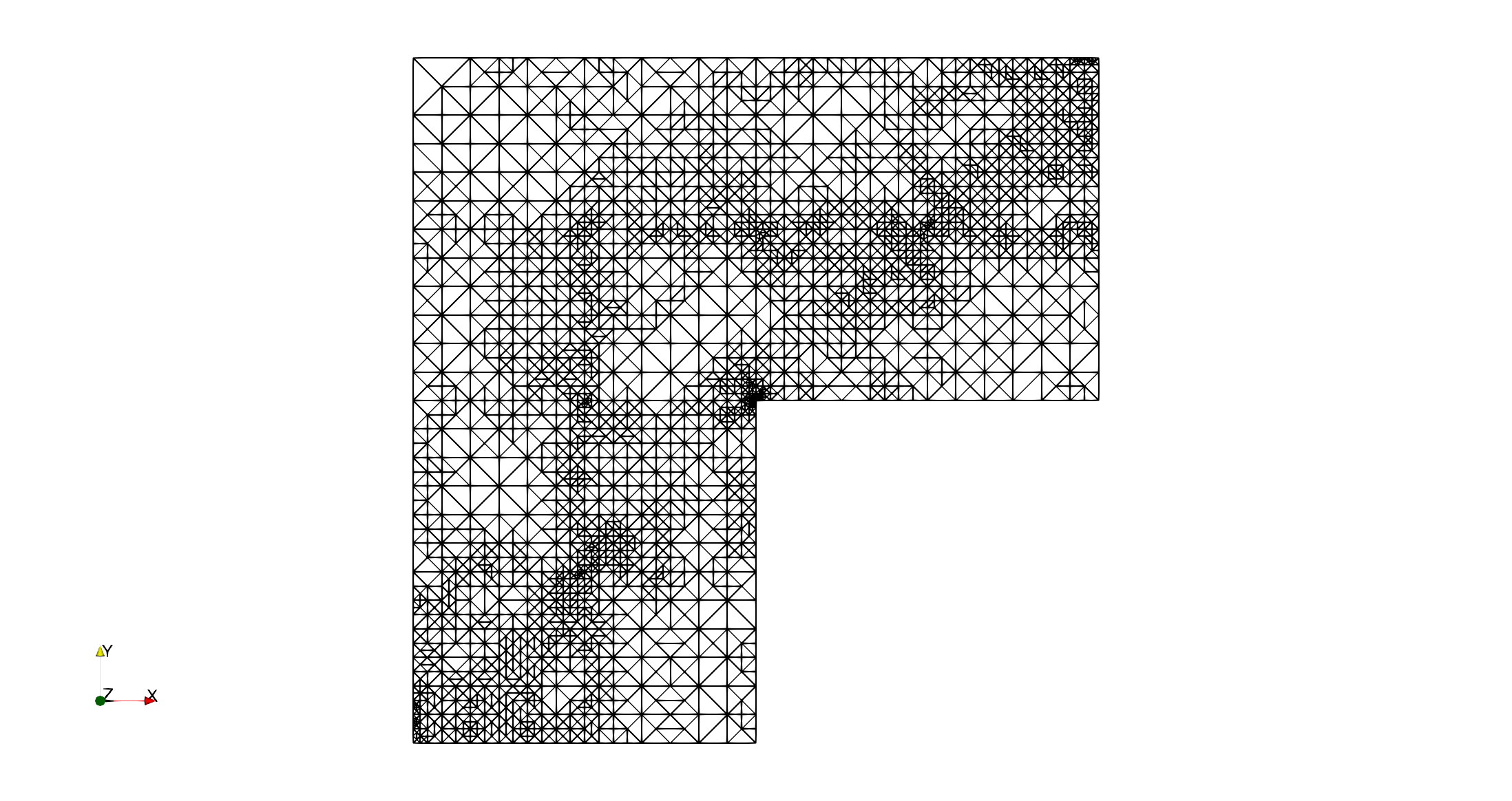}
	\end{minipage}
	\begin{minipage}{0.32\linewidth}\centering
		{\footnotesize $\lambda_{h,1}, \texttt{dof}=35289$ iter=10}\\
		\includegraphics[scale=0.1,trim=20.2cm 0cm 20.2cm 2cm,clip]{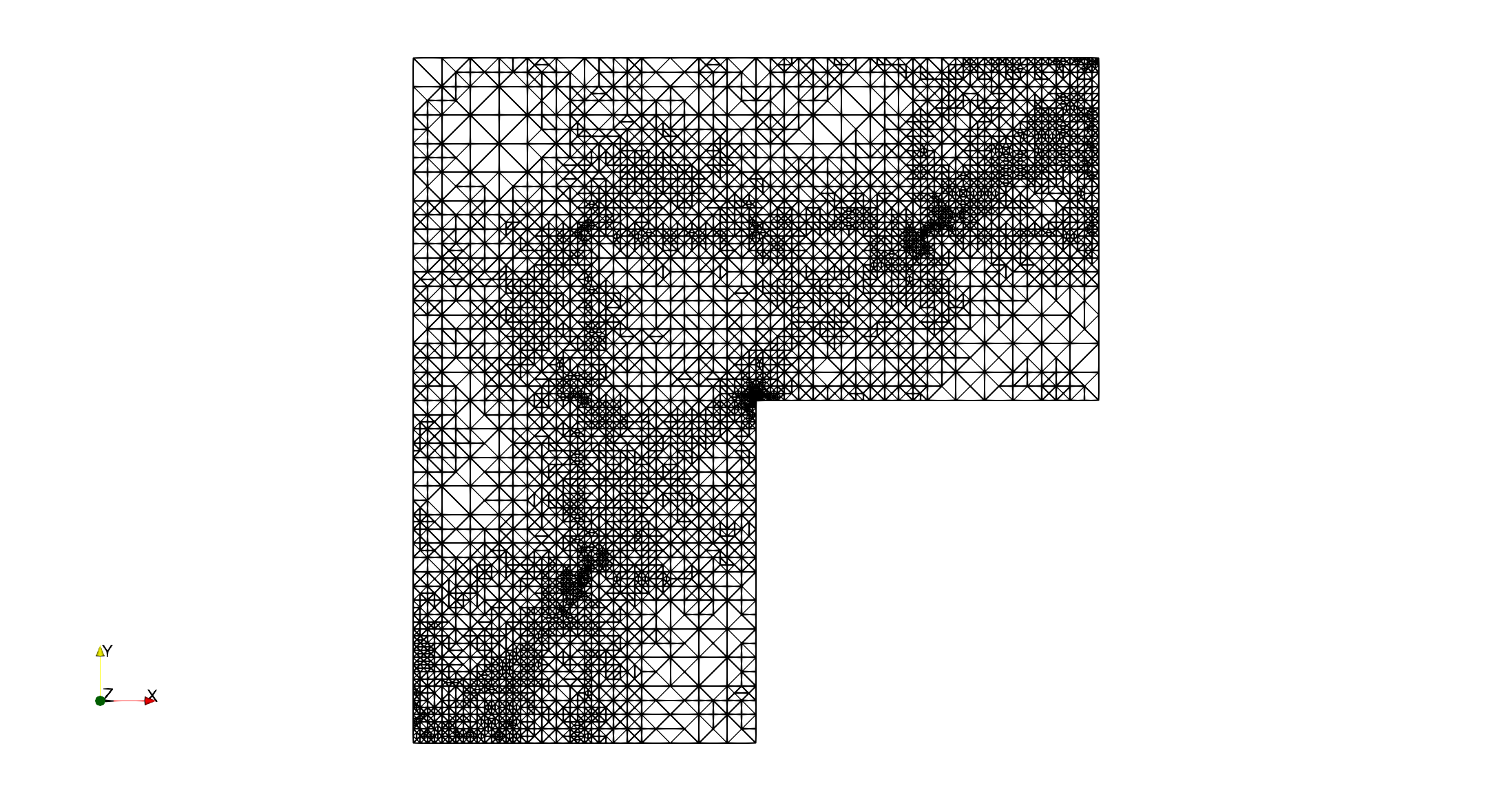}
	\end{minipage}
	\begin{minipage}{0.32\linewidth}\centering
		{\footnotesize $\lambda_{h,1}, \texttt{dof}=67391$ iter=12}\\
		\includegraphics[scale=0.1,trim=20.2cm 0cm 20.2cm 2cm,clip]{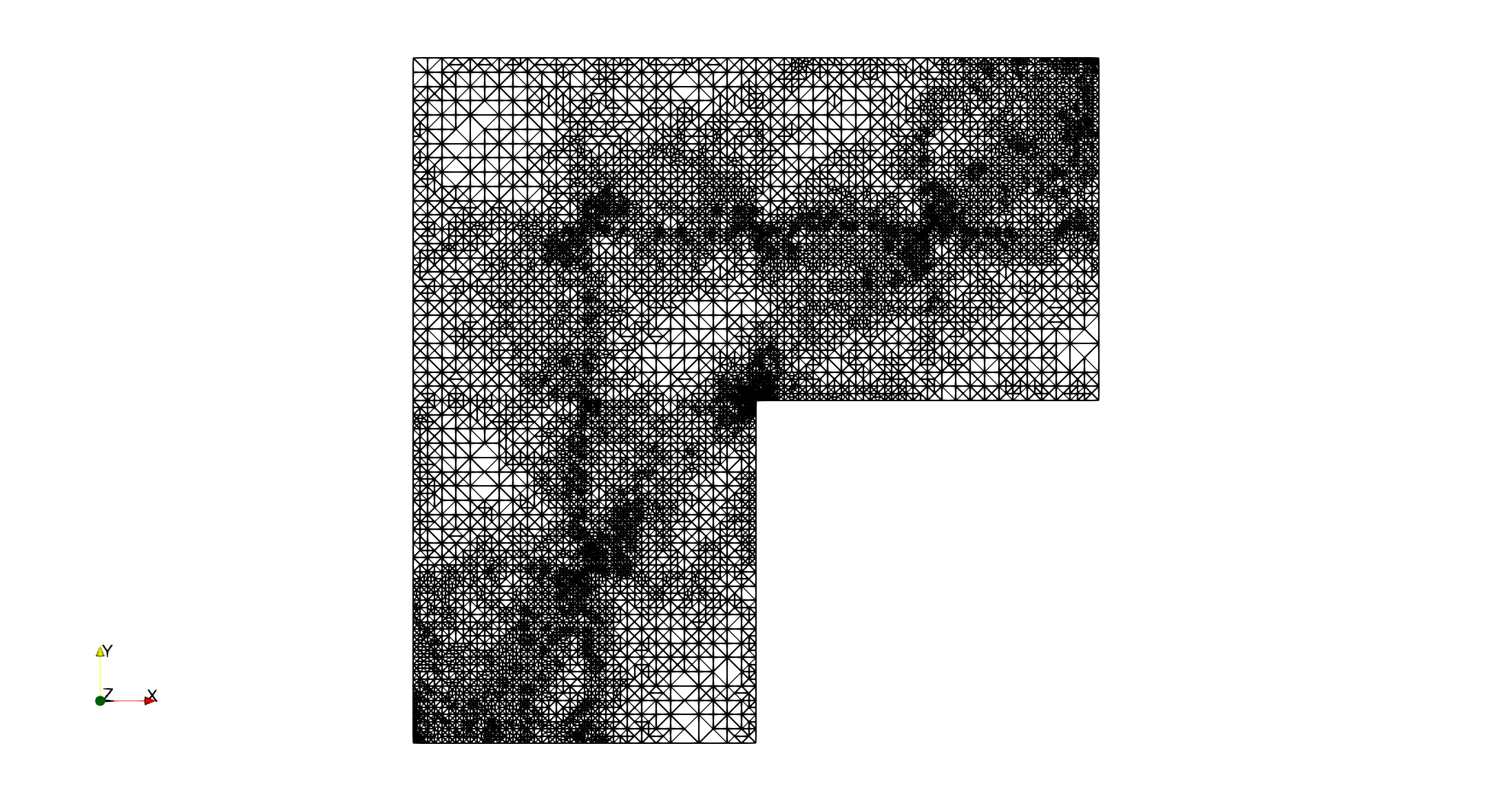}
	\end{minipage}\\
	\begin{minipage}{0.32\linewidth}\centering
		{\footnotesize $\lambda_{h,4}, \texttt{dof}=9937$ iter=8}\\
		\includegraphics[scale=0.1,trim=20.2cm 0cm 20.2cm 2cm,clip]{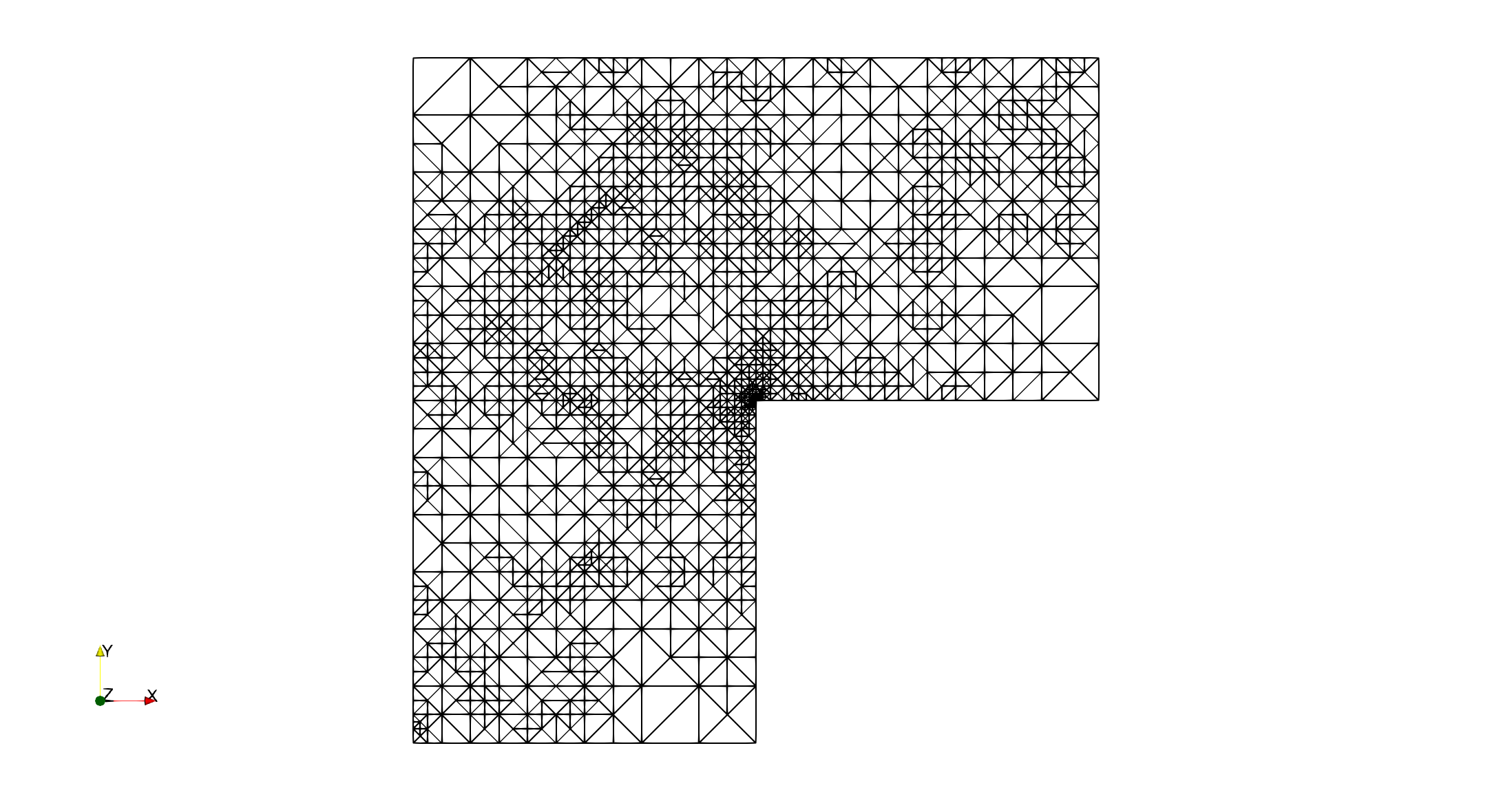}
	\end{minipage}
	\begin{minipage}{0.32\linewidth}\centering
		{\footnotesize $\lambda_{h,4}, \texttt{dof}=21065$ iter=10}\\
		\includegraphics[scale=0.1,trim=20.2cm 0cm 20.2cm 2cm,clip]{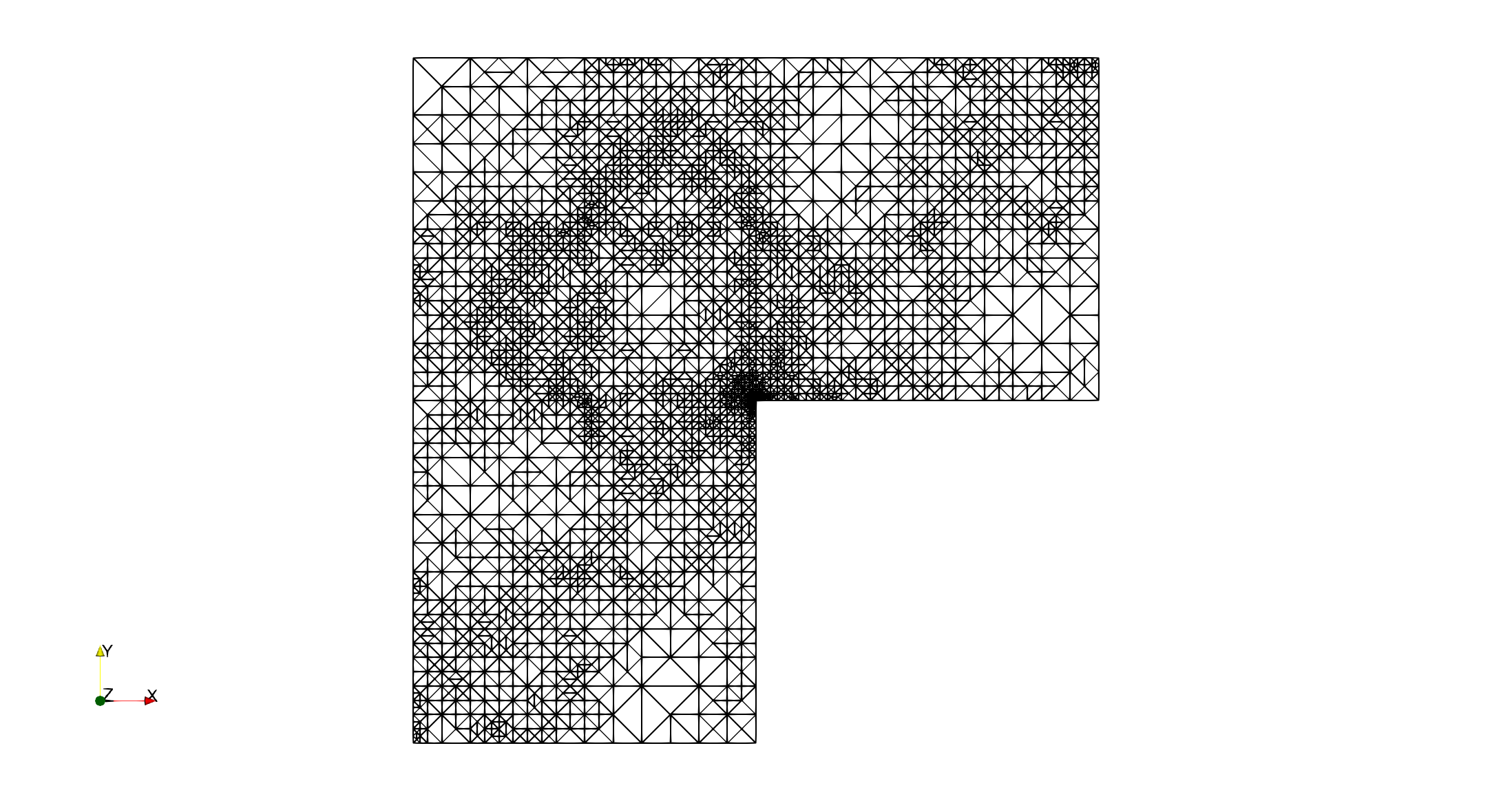}
	\end{minipage}
	\begin{minipage}{0.32\linewidth}\centering
		{\footnotesize $\lambda_{h,4}, \texttt{dof}=44331$ iter=12}\\
		\includegraphics[scale=0.1,trim=20.2cm 0cm 20.2cm 2cm,clip]{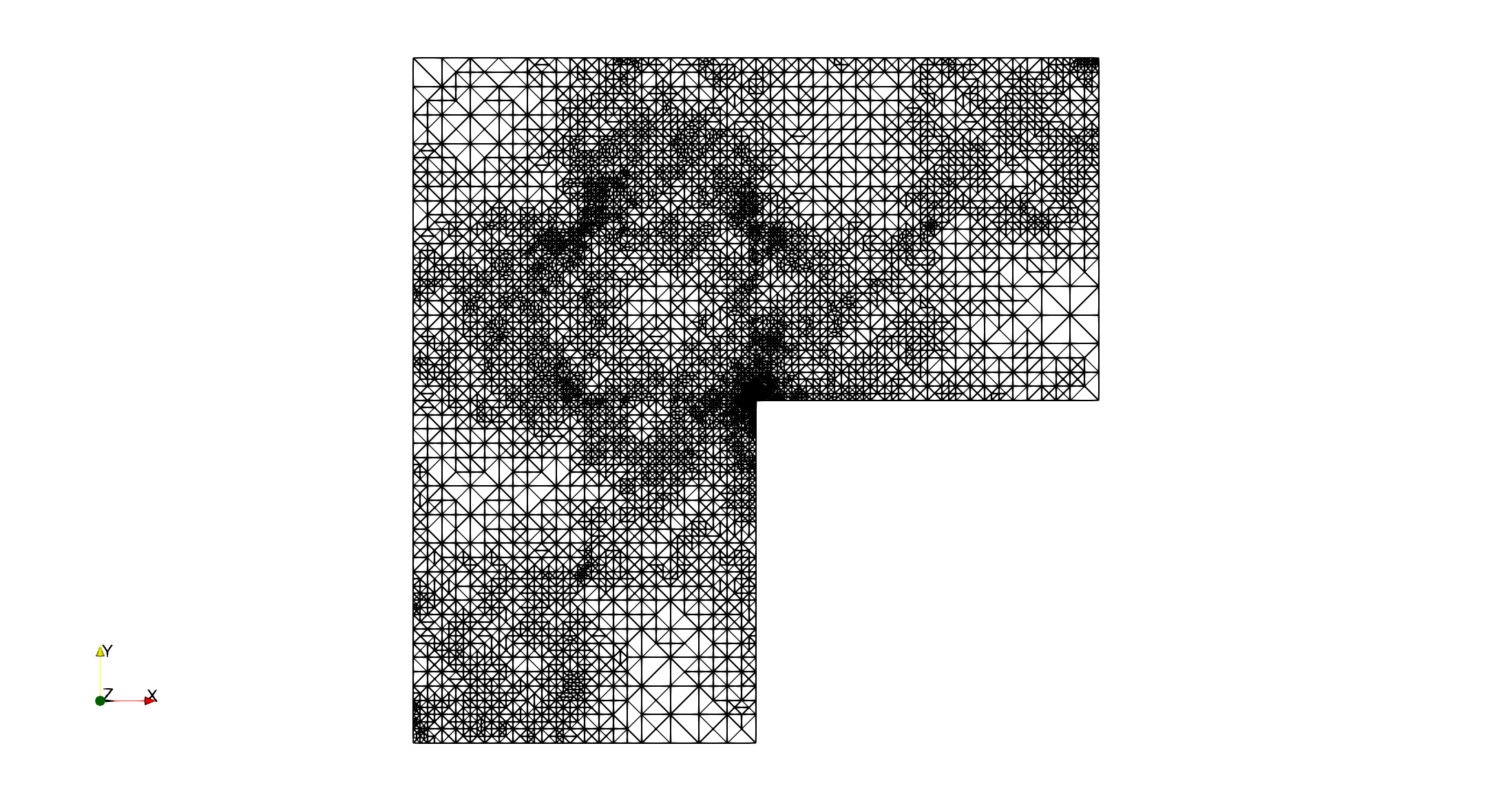}
	\end{minipage}\\
	\caption{Test \ref{subsec:lshape2D}. Intermediate meshes for the first and fourth computed eigenvalue with mini elements on the Lshaped geometry with mixed boundary conditions and $\mathbb{K}^{-1}:=10^{3}\mathbb{I}$ on $\Omega_D$.}
	\label{fig:lshape2d-meshes}
\end{figure}

\begin{figure}[!hbt]\centering
	\begin{minipage}{0.32\linewidth}\centering
		{\footnotesize $\lambda_{h,1}, \texttt{dof}=8952,$ iter=8}\\
		\includegraphics[scale=0.1,trim=20.2cm 0cm 20.2cm 2cm,clip]{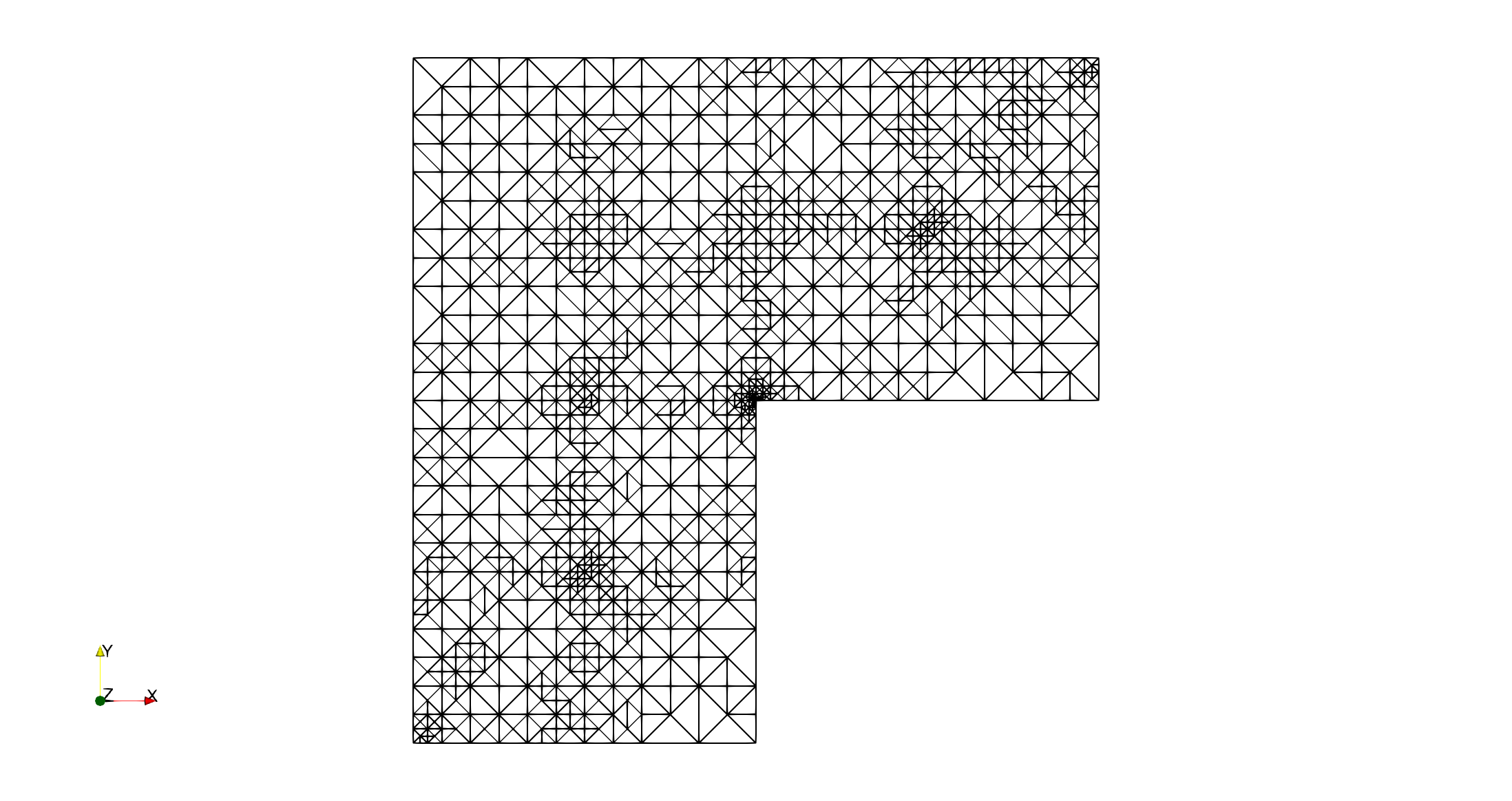}
	\end{minipage}
	\begin{minipage}{0.32\linewidth}\centering
		{\footnotesize $\lambda_{h,1}, \texttt{dof}=14930$ iter=10}\\
		\includegraphics[scale=0.1,trim=20.2cm 0cm 20.2cm 2cm,clip]{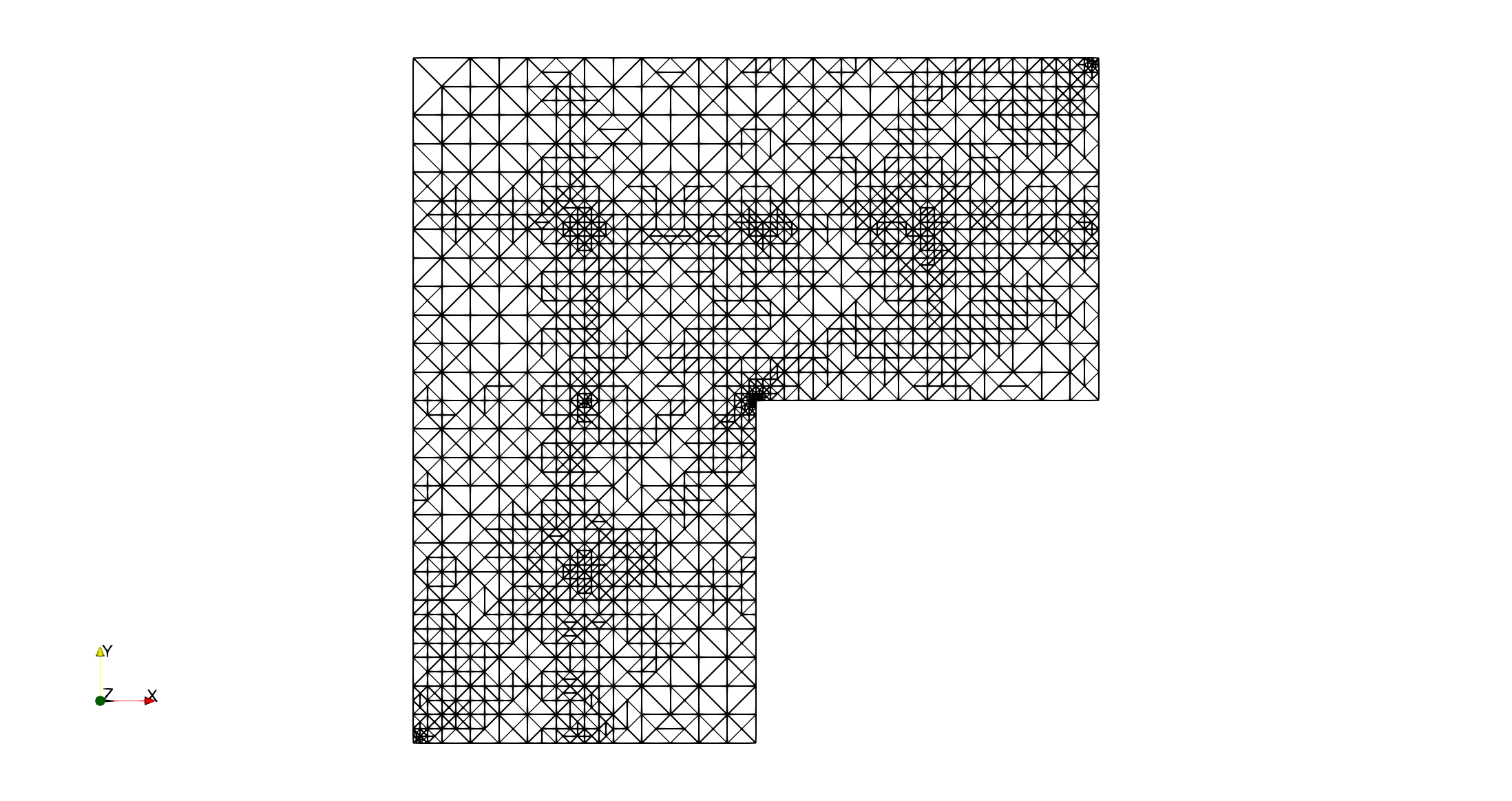}
	\end{minipage}
	\begin{minipage}{0.32\linewidth}\centering
		{\footnotesize $\lambda_{h,1}, \texttt{dof}=24876$ iter=12}\\
		\includegraphics[scale=0.1,trim=20.2cm 0cm 20.2cm 2cm,clip]{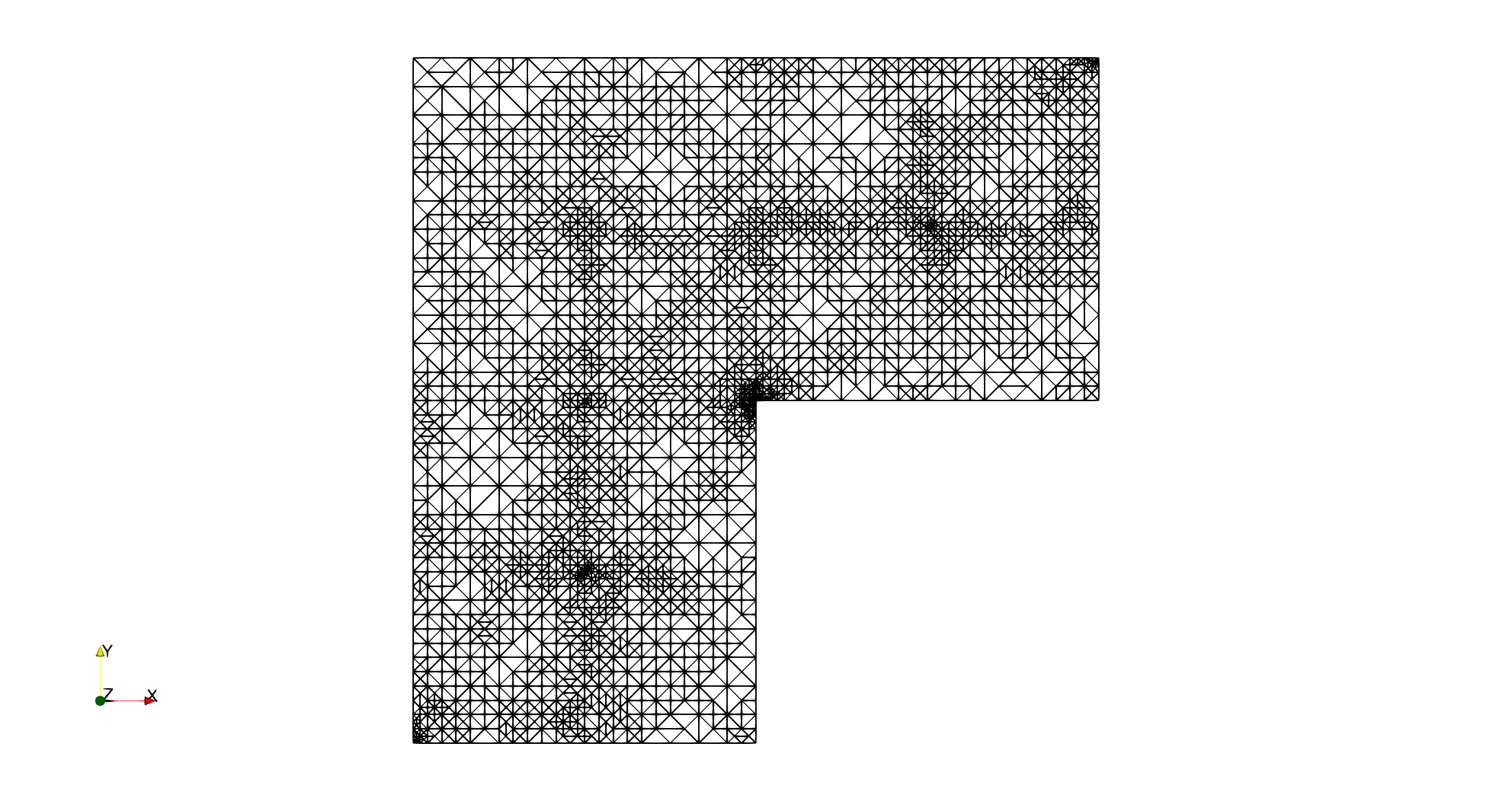}
	\end{minipage}\\
	\begin{minipage}{0.32\linewidth}\centering
		{\footnotesize $\lambda_{h,4}, \texttt{dof}=5544$ iter=8}\\
		\includegraphics[scale=0.1,trim=20.2cm 0cm 20.2cm 2cm,clip]{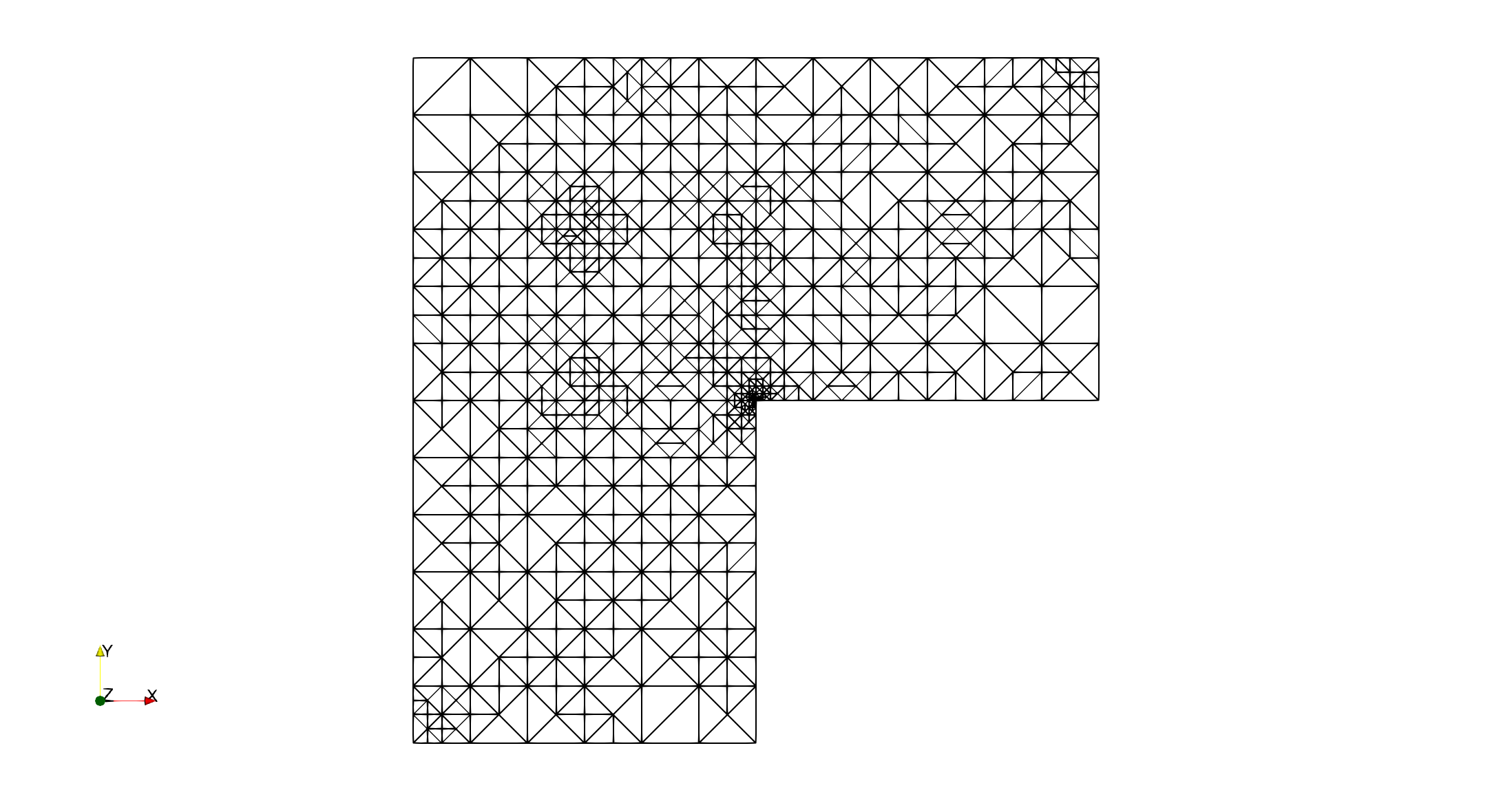}
	\end{minipage}
	\begin{minipage}{0.32\linewidth}\centering
		{\footnotesize $\lambda_{h,4}, \texttt{dof}=9585$ iter=10}\\
		\includegraphics[scale=0.1,trim=20.2cm 0cm 20.2cm 2cm,clip]{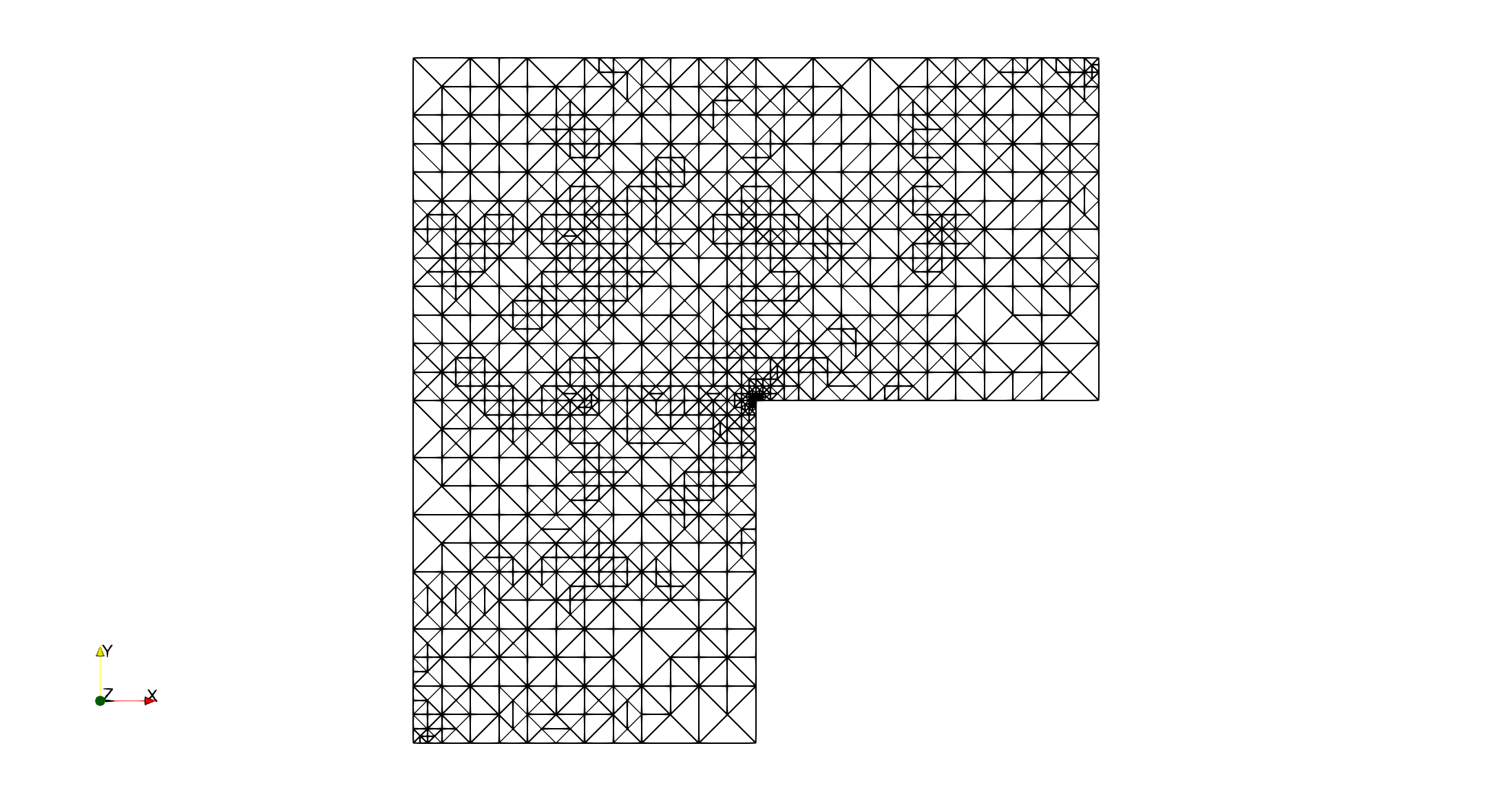}
	\end{minipage}
	\begin{minipage}{0.32\linewidth}\centering
		{\footnotesize $\lambda_{h,4}, \texttt{dof}=16036$ iter=12}\\
		\includegraphics[scale=0.1,trim=20.2cm 0cm 20.2cm 2cm,clip]{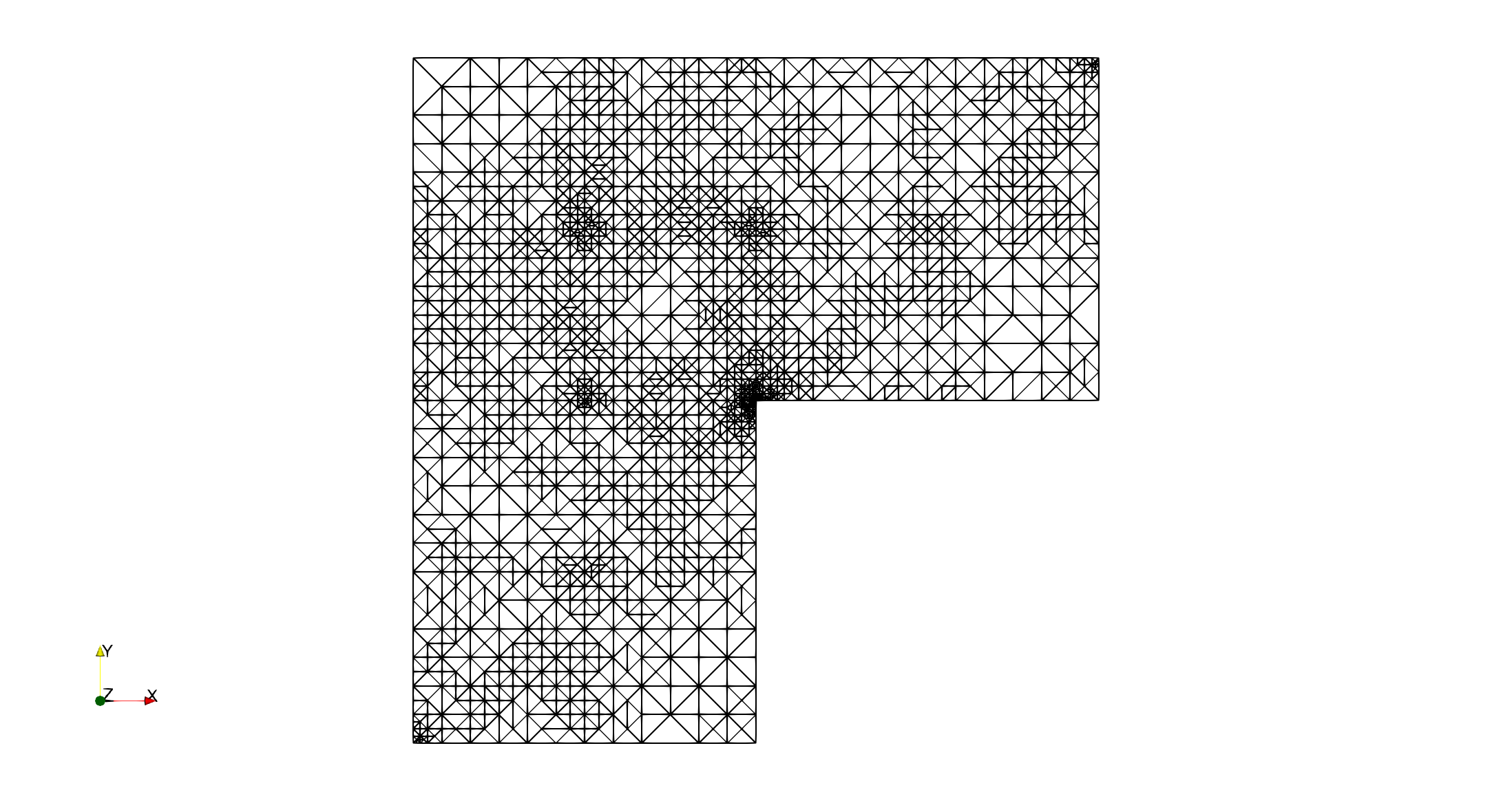}
	\end{minipage}\\
	\caption{Test \ref{subsec:lshape2D}. Intermediate meshes for the first and fourth computed eigenvalue with Taylor-Hood elements on the Lshaped geometry with mixed boundary conditions and $\mathbb{K}^{-1}:=10^{3}\mathbb{I}$ on $\Omega_D$.}
	\label{fig:lshape2d-meshes_TH}
\end{figure}

\subsection{The 3D Lshaped domain}\label{subsec:3D-Lshape}
For this final experiment test we have considered a 3D geometry with heterogeneous media  with singularities. The domain is a 3D Lshaped domain defined by
$$
\Omega:=(-1,1)\times(-1,1)\times(-1,0)\backslash\bigg((-1,0)\times(-1,0)\times(-1,0) \bigg).
$$

The domain contains a dihedral singularity along the line $(0,0,z)$, for $z\in[-1,0]$. The subdomain division as well as the initial mesh for uniform and adaptive refinement is depicted in Figure \ref{fig:3D-lshape-sample}. For simplicity, homogeneous boundary conditions are considered on the whole boundary, thus $\Gamma_2=\emptyset$. Because of the nature of the geometry, we focus our attention on the behavior of the first eigenvalue.

Convergence results for almost permeable ($\mathbb{K}^{-1}\vert_{\Omega_D}=10^{-8}\mathbb{I}$) up to mild porous region ($\mathbb{K}^{-1}\vert_{\Omega_D}=10^{2}\mathbb{I}$) are presented in Table \ref{table:lshape3d-convergence}.  It is observed that there is a suboptimal convergence rate ranging between $\mathcal{O}(h^{1.4})$ and $\mathcal{O}(h^{1.6})$ for the first eigenvalue, independent of the element used. This behavior is expected because of the dihedral singularity. A visual description of the eigenfunctions is depicted in Figure \ref{fig:3D-lshape-uh-ph}. Here we note that for $\mathbb{K}^{-1}\vert_{\Omega_D}=10^{2}\mathbb{I}$ the eigenmode start to get confined in the cube-shaped Stokes region of the domain, however, the pressure gradients are still high across the singularity. Also, we note that the pressure gradients are more pronounced when we increase the permeability of $\Omega_D$, which is expected since we are approaching to the first eigenvalue in a Stokes eigenvalue problem with similar geometry (see for example \cite[Section 5]{lepe2023posteriori}).


When studying the a posteriori estimator on the three singular modes from Table \ref{table:lshape3d-convergence}, we observed that the estimator detects the geometrical singularity, which also contains high pressure gradients, and refines accordingly. With mini elements, an optimal rate $\mathcal{O}(h^2)\approx\mathcal{O}(\texttt{dof}^{-0.66})$ is attained, while $\mathcal{O}^(h^4)\approx\mathcal{O}(\texttt{dof}^{-1.33}$) is obtained with Taylor-Hood elements.  By observing Figures \ref{fig:3D-lshape-meshes}--\ref{fig:3D-lshape-meshes_TH} we note that for $\mathbb{K}^{-1}\vert_{\Omega_D}=10^{2}\mathbb{I}$ there are less refinements across the domain and more on the singularity. Subtle differences between choosing $10^{-8}\mathbb{I}$ and $10^{1}\mathbb{I}$ for $\mathbb{K}^{-1}$ are observed since the estimator still gets most of the error values from the singularity. Also, similar to the 2D case, Taylor-Hood adaptive refinements mark less elements outside the singularity. The associated error curves and efficiency for the aforementioned results are presented on Figure \ref{fig:error-eff-lshape3d}. It is clear that optimal rates are attained for the two families of elements and the three cases under study. The efficiency of the estimator is also evidenced, with more oscillations when Taylor-Hood elements are used.
\begin{figure}[!hbt]\centering
	\includegraphics[scale=0.17]{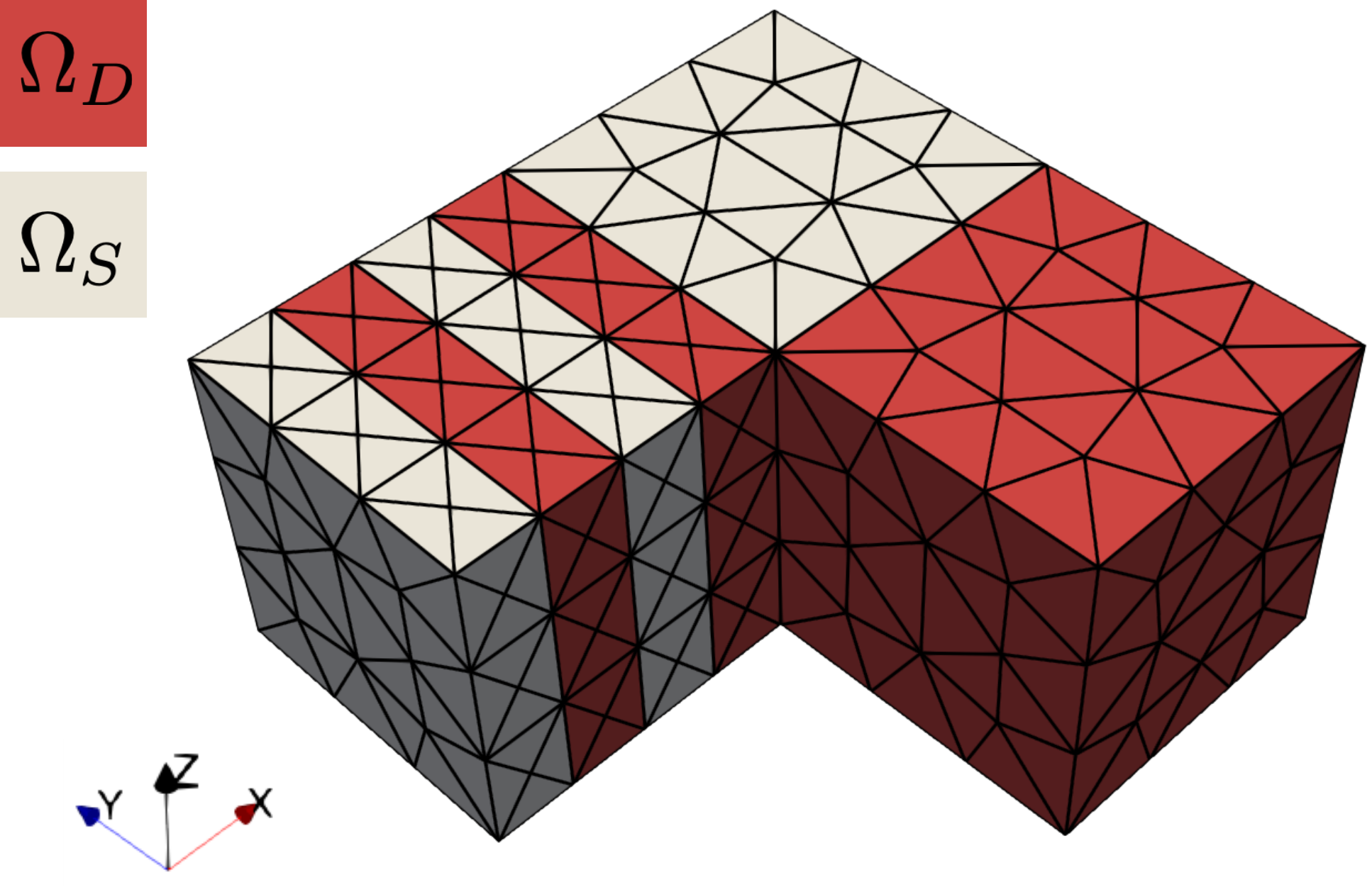}
	\caption{Test \ref{subsec:3D-Lshape}. Sample meshed geometry of the 3D Lshaped domain with heterogeneous porosity. The orientation axes for the current geometry visualization are included.}
	\label{fig:3D-lshape-sample}
\end{figure}

\begin{table}[hbt!]
	\centering 
	{\footnotesize
		\caption{Test \ref{subsec:3D-Lshape}. Convergence history for the first computed eigenvalue using uniform refinements with mini- and Taylor-Hood elements when different permeability parameters on $\Omega_D$ are considered. The convergence order is given for $\mathcal{O}(h^r)\approx\mathcal{O}(\texttt{dof}^{-r/3})$.}
		\label{table:lshape3d-convergence}
		\begin{tabular}{|c|c c c c |c| c|}
			\hline
			\hline
						\multicolumn{7}{|c|}{mini elements}\\
			\hline
			$\mathbb{K}^{-1}$&$\texttt{dof}=5442$             &  $\texttt{dof}=40956$         &   $\texttt{dof}=318900$         & $\texttt{dof}=2519268$ & Order & $\lambda_{1}$\\
			\hline
			$10^{-8}\mathbb{I}$&52.8392  &    48.7989  &    43.8184  &    42.4277  & 1.48 &    41.8418   \\
			 $10^{1}\mathbb{I}$&56.8641  &    50.6683  &    45.7282  &    44.3610  & 1.58 &    43.8090  \\
			$10^{2}\mathbb{I}$&73.1243  &    54.1100  &    60.0979  &    52.4376  & 1.42 &    51.8884  \\
			\hline
			\multicolumn{7}{|c|}{Taylor-Hood elements}\\
			\hline
			$\mathbb{K}^{-1}$&$\texttt{dof}=7215$             &  $\texttt{dof}=50514$         &   $\texttt{dof}=377976$         & $\texttt{dof}=2924556$ & Order & $\lambda_{1}$\\
			\hline
			$10^{-8}\mathbb{I}$&440.8226  &    41.7259  &    41.7618  &    41.8038  & 1.53 &    41.8418  \\
			$10^{1}\mathbb{I}$&43.0893  &    43.7377  &    43.7446  &    43.7774  & 1.41 &    43.8090  \\
			$10^{2}\mathbb{I}$&52.8033  &    51.8876  &    52.0338  &    51.8825  & 1.46 &    51.8884  \\
			
			\hline
			\hline             
	\end{tabular}}
\end{table}

\begin{figure}[!hbt]\centering
	\begin{minipage}{0.3275\linewidth}\centering
			{\footnotesize $\bu_{h,1},\mathbb{K}^{-1}\vert_{\Omega_D}=10^{-8}\mathbb{I}$}\\
		\includegraphics[scale=0.108,trim=19cm 9cm 19cm 4cm,clip]{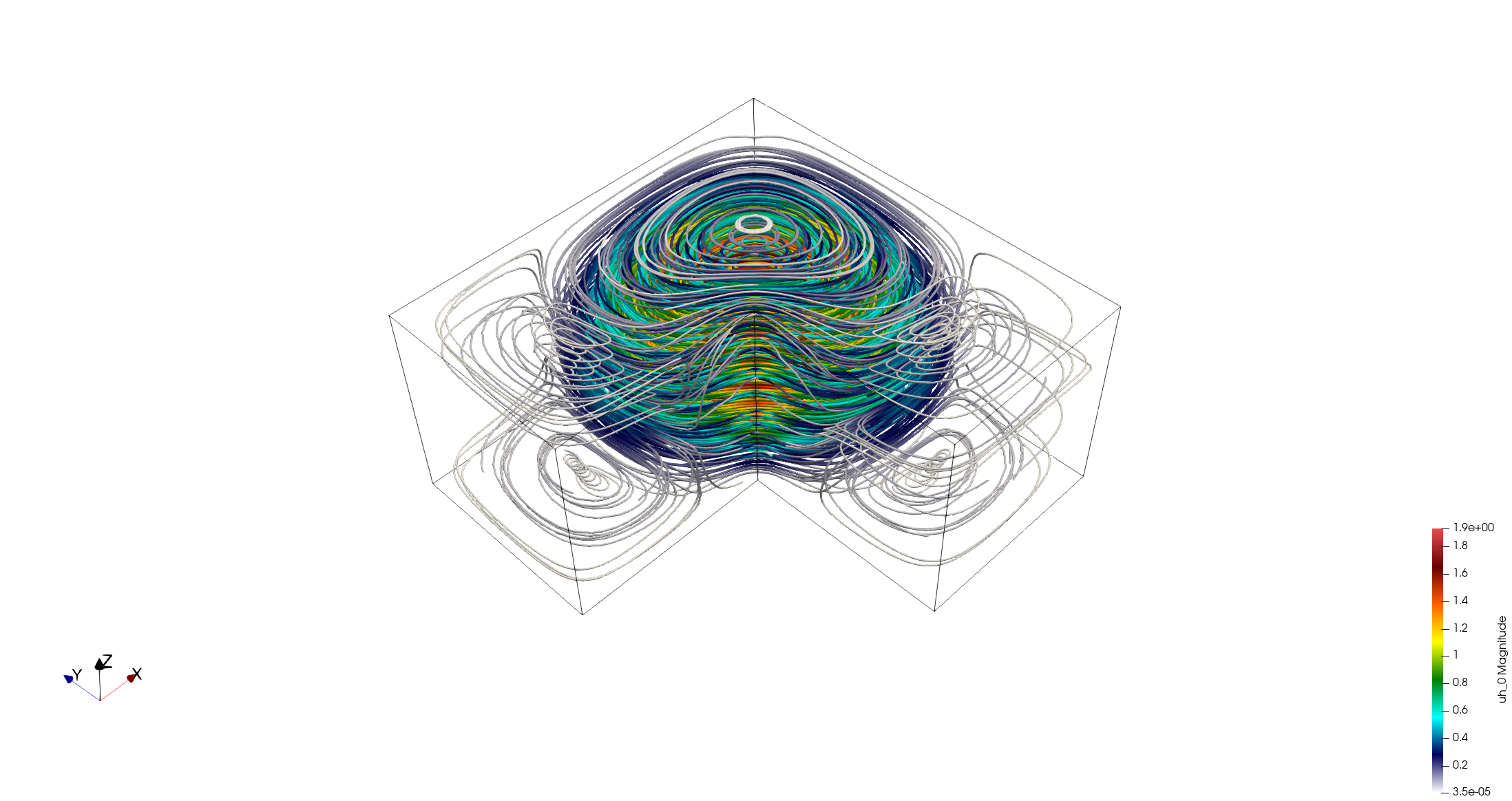}
	\end{minipage}
	\begin{minipage}{0.3275\linewidth}\centering
			{\footnotesize $\bu_{h,1},\mathbb{K}^{-1}\vert_{\Omega_D}=10^{1}\mathbb{I}$}\\
		\includegraphics[scale=0.1082,trim=19cm 9cm 19cm 4cm,clip]{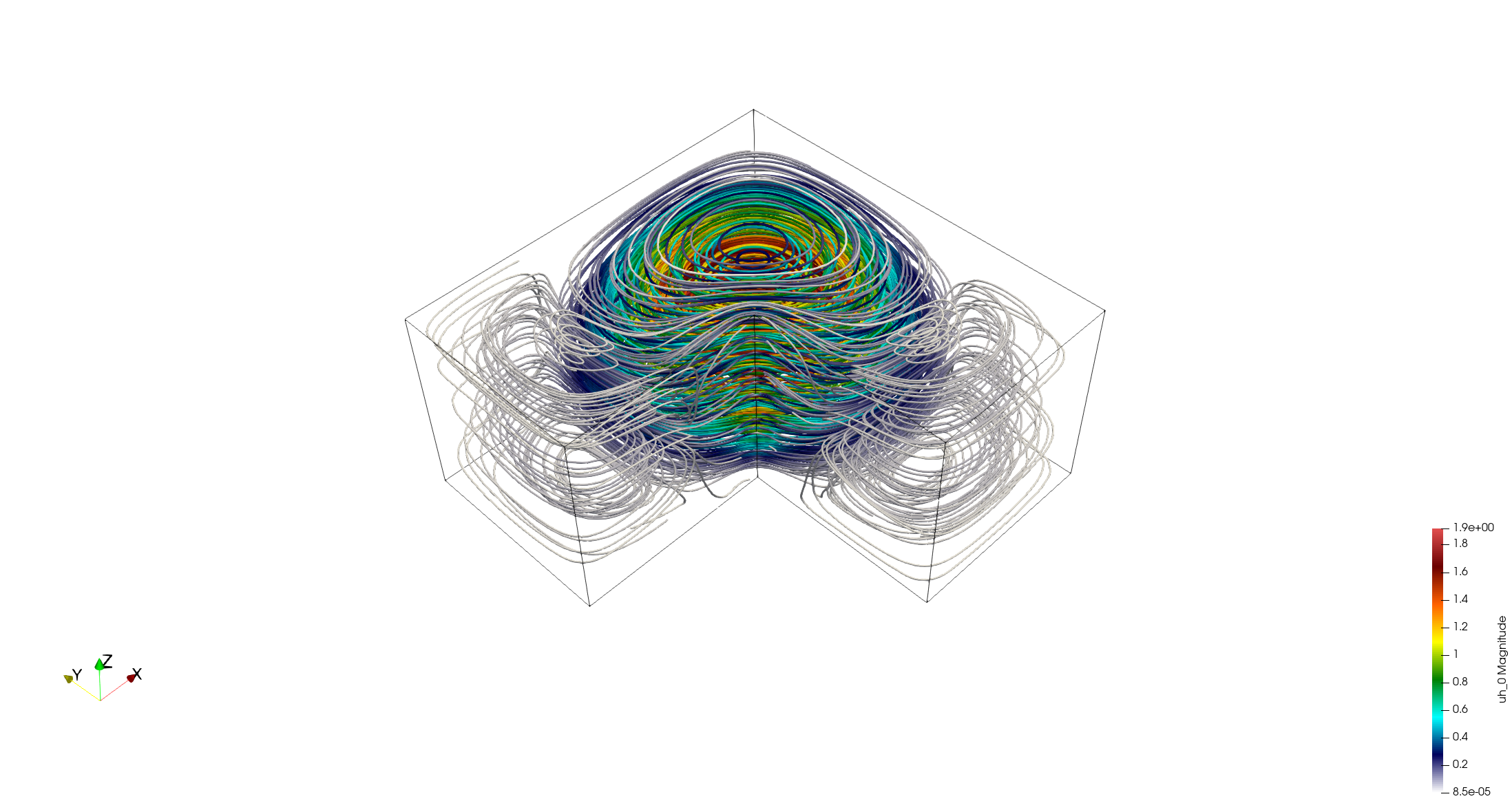}
	\end{minipage}
	\begin{minipage}{0.3275\linewidth}\centering
			{\footnotesize $\bu_{h,1},\mathbb{K}^{-1}\vert_{\Omega_D}=10^{2}\mathbb{I}$}\\
		\includegraphics[scale=0.108,trim=19cm 9cm 19cm 4cm,clip]{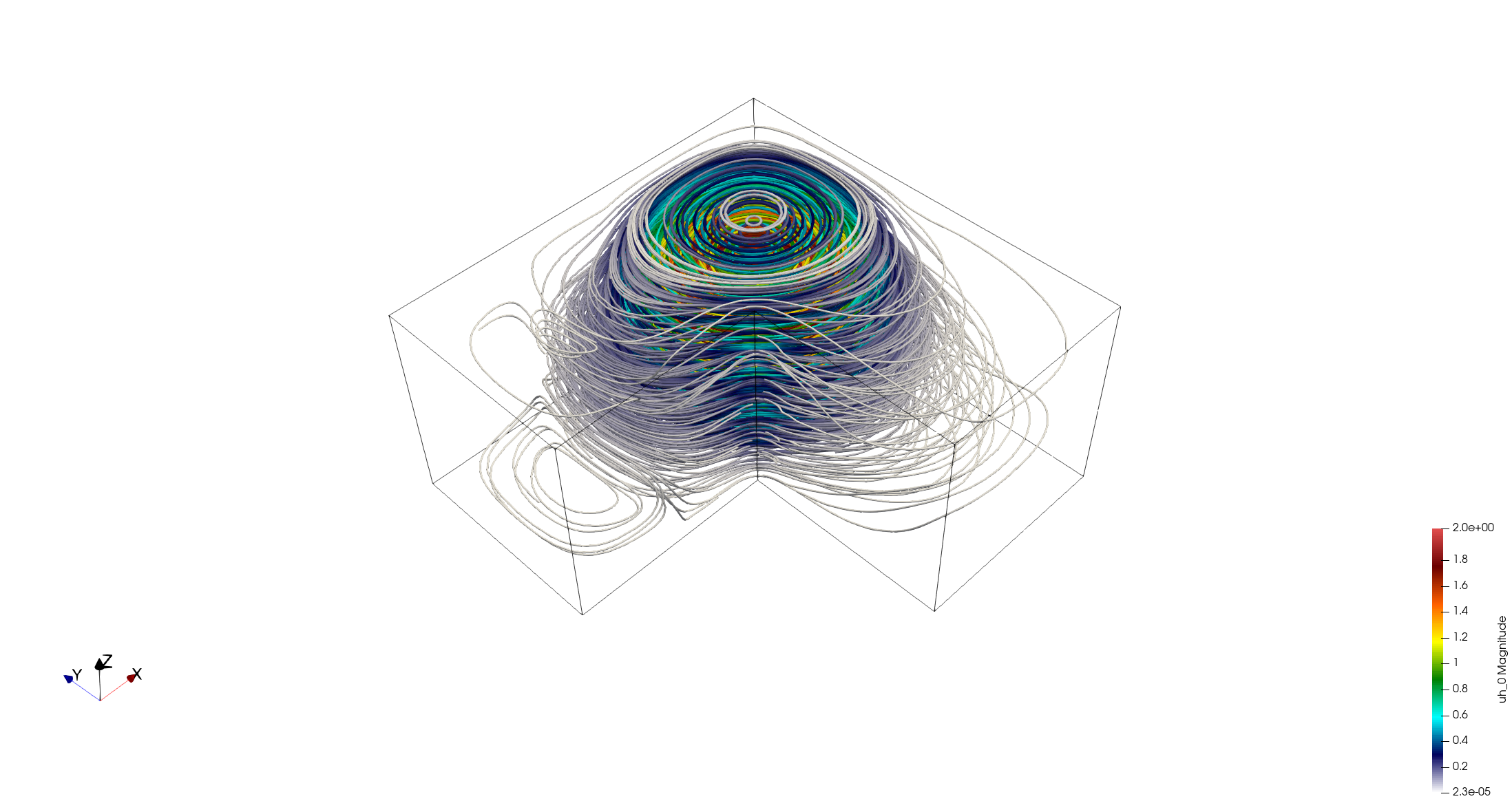}
	\end{minipage}\\
		\begin{minipage}{0.3275\linewidth}\centering
		{\footnotesize $p_{h,1},\mathbb{K}^{-1}\vert_{\Omega_D}=10^{-8}\mathbb{I}$}\\
		\includegraphics[scale=0.108,trim=19cm 9cm 19cm 4cm,clip]{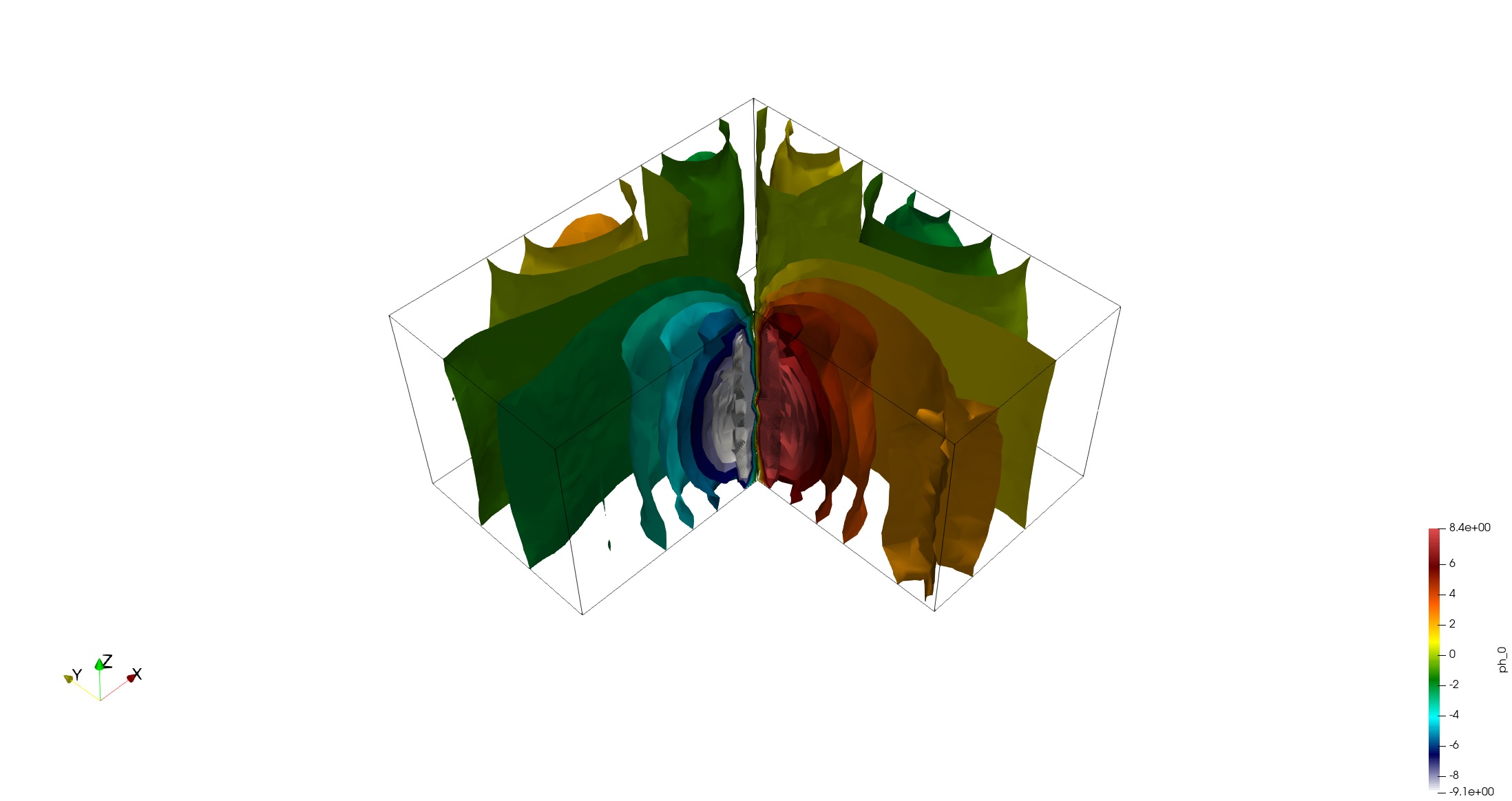}
	\end{minipage}
	\begin{minipage}{0.3275\linewidth}\centering
		{\footnotesize $p_{h,1},\mathbb{K}^{-1}\vert_{\Omega_D}=10^{1}\mathbb{I}$}\\
		\includegraphics[scale=0.108,trim=19cm 9cm 19cm 4cm,clip]{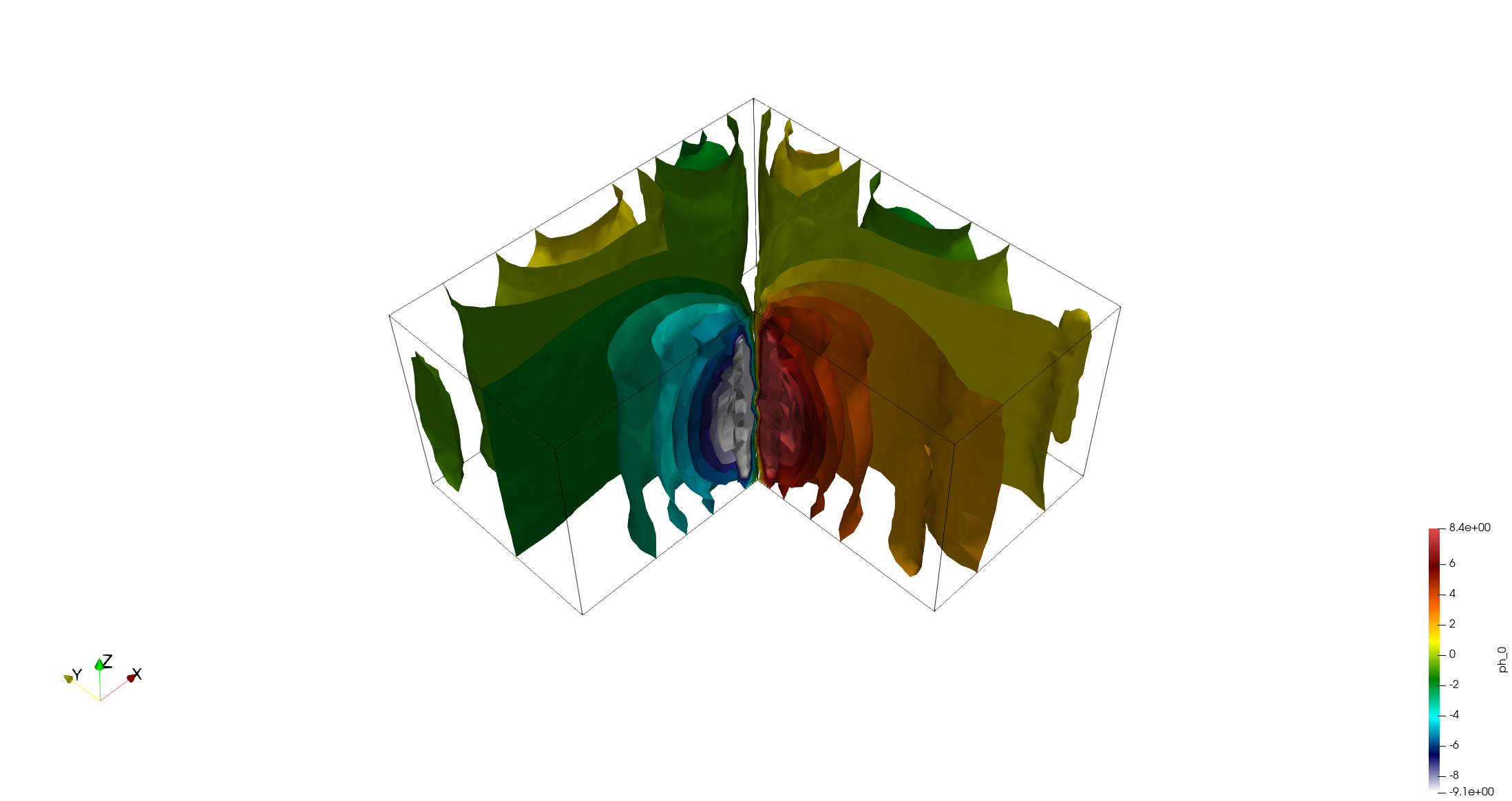}
	\end{minipage}
	\begin{minipage}{0.3275\linewidth}\centering
		{\footnotesize $p_{h,1},\mathbb{K}^{-1}\vert_{\Omega_D}=10^{2}\mathbb{I}$}\\
		\includegraphics[scale=0.108,trim=19cm 9cm 19cm 4cm,clip]{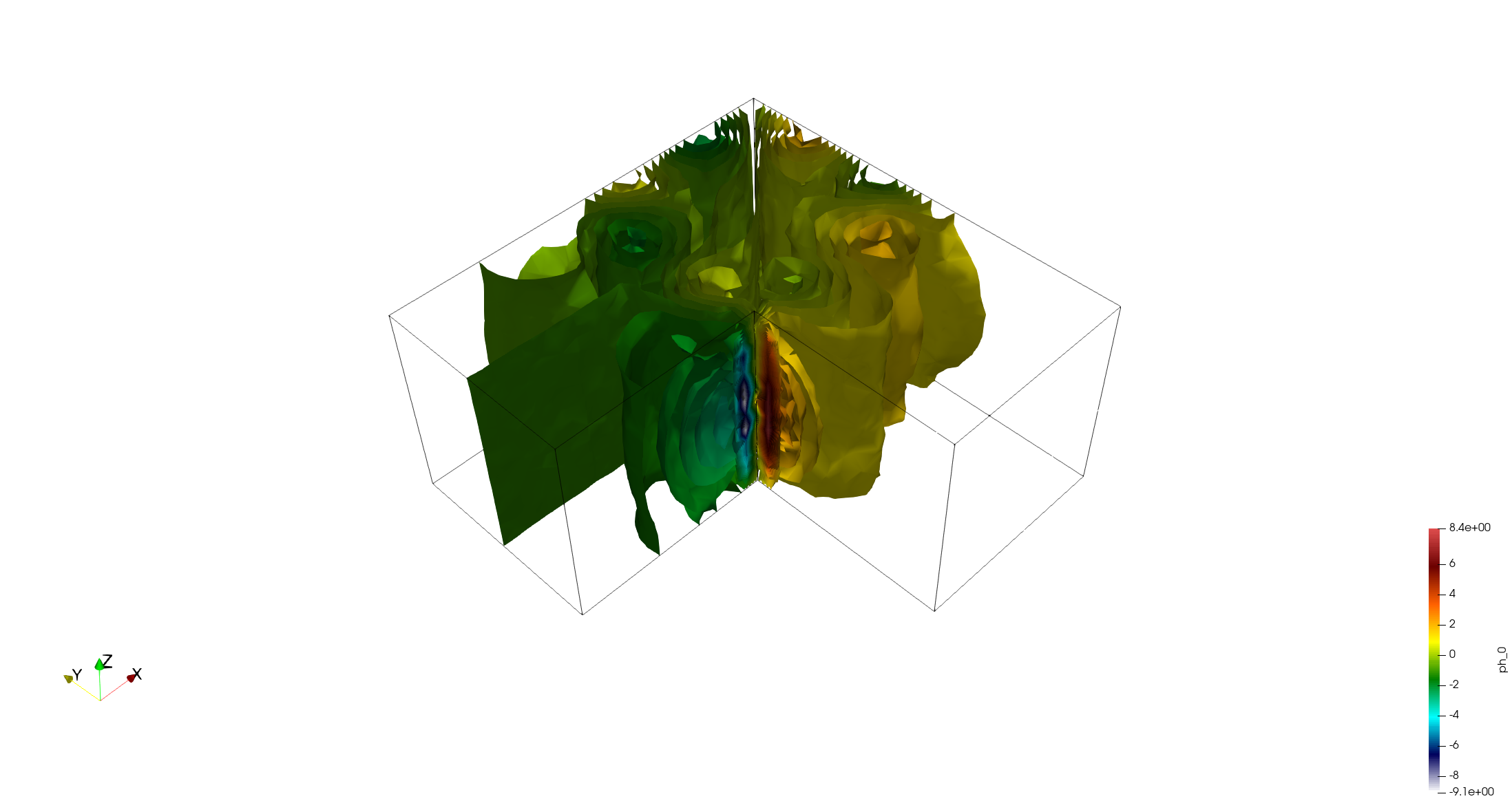}
	\end{minipage}
	\caption{Test \ref{subsec:3D-Lshape}. Velocity streamlines and pressure contour plots on the 3D Lshaped domain for the first eigenvalue on each permeability configuration.}
	\label{fig:3D-lshape-uh-ph}
\end{figure}

\begin{figure}[!hbt]\centering
	\begin{minipage}{0.3275\linewidth}\centering
		{\footnotesize $\mathbb{K}^{-1}\vert_{\Omega_D}=10^{-8}\mathbb{I}, \texttt{dof}=311044$}\\
		\includegraphics[scale=0.108,trim=19cm 9cm 19cm 4cm,clip]{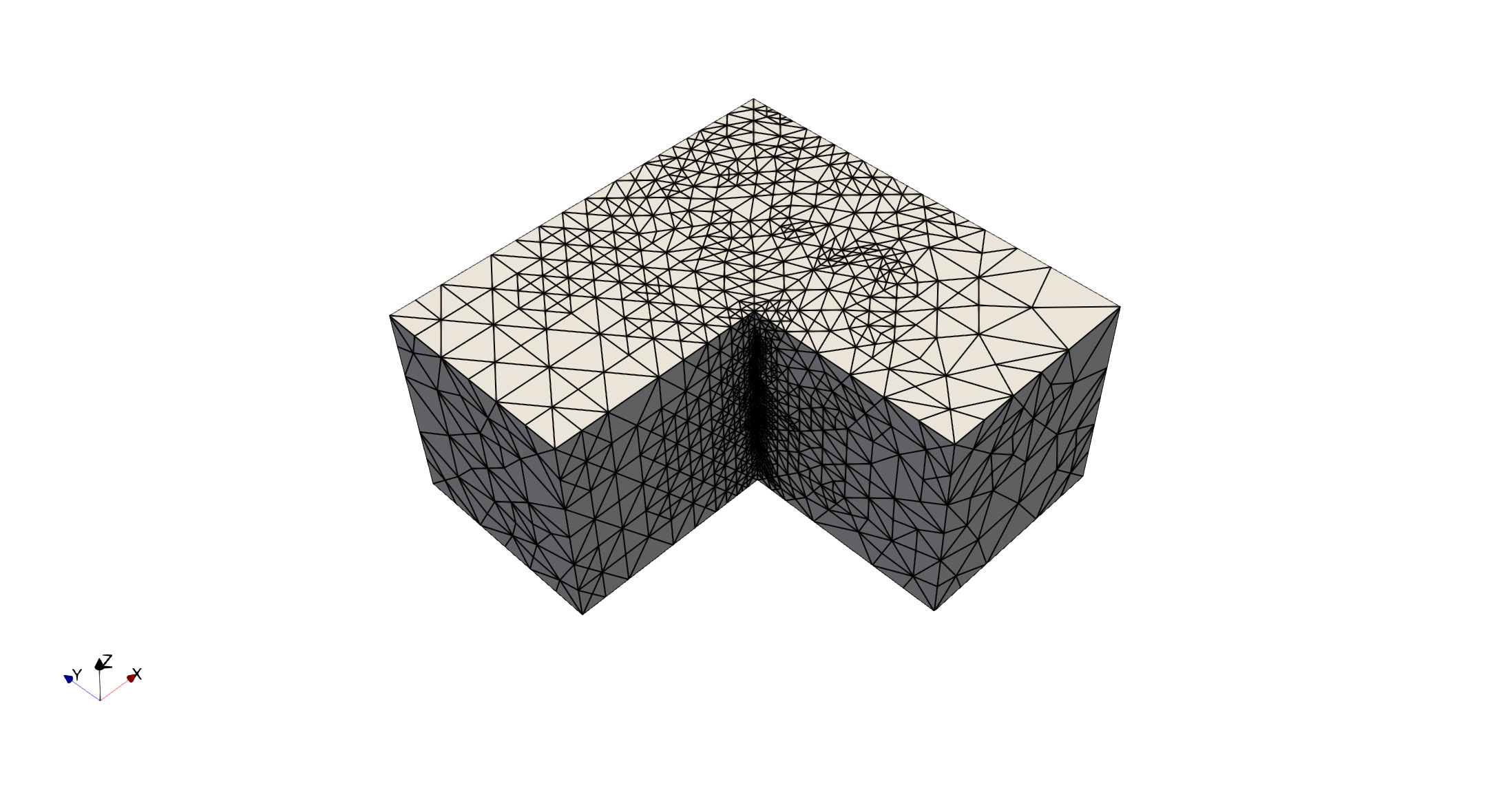}
	\end{minipage}
	\begin{minipage}{0.3275\linewidth}\centering
		{\footnotesize $\mathbb{K}^{-1}\vert_{\Omega_D}=10^{-8}\mathbb{I}, \texttt{dof}=730481$}\\
		\includegraphics[scale=0.108,trim=19cm 9cm 19cm 4cm,clip]{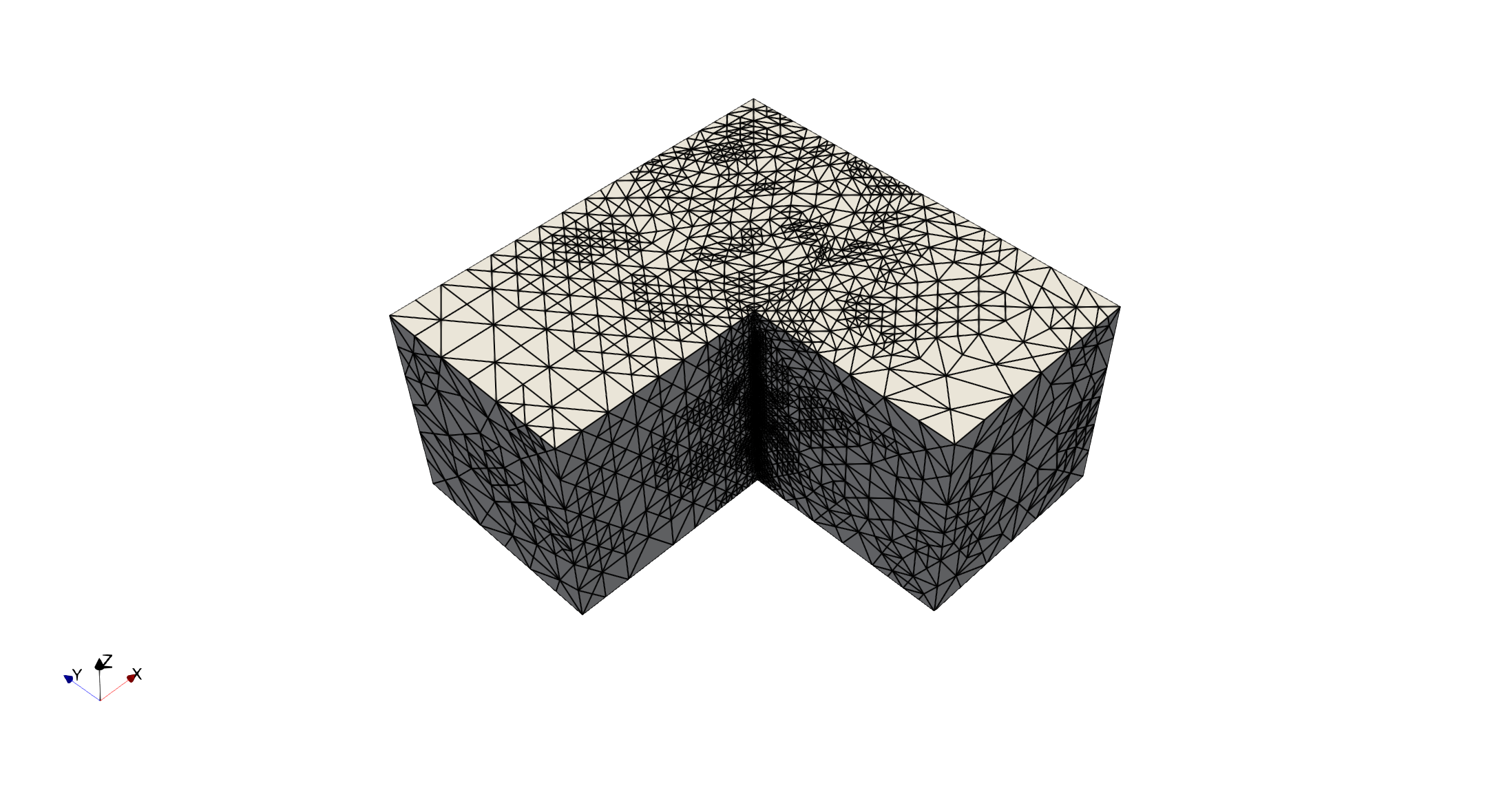}
	\end{minipage}
	\begin{minipage}{0.3275\linewidth}\centering
		{\footnotesize $\mathbb{K}^{-1}\vert_{\Omega_D}=10^{-8}\mathbb{I}, \texttt{dof}=1221041$}\\
		\includegraphics[scale=0.108,trim=19cm 9cm 19cm 4cm,clip]{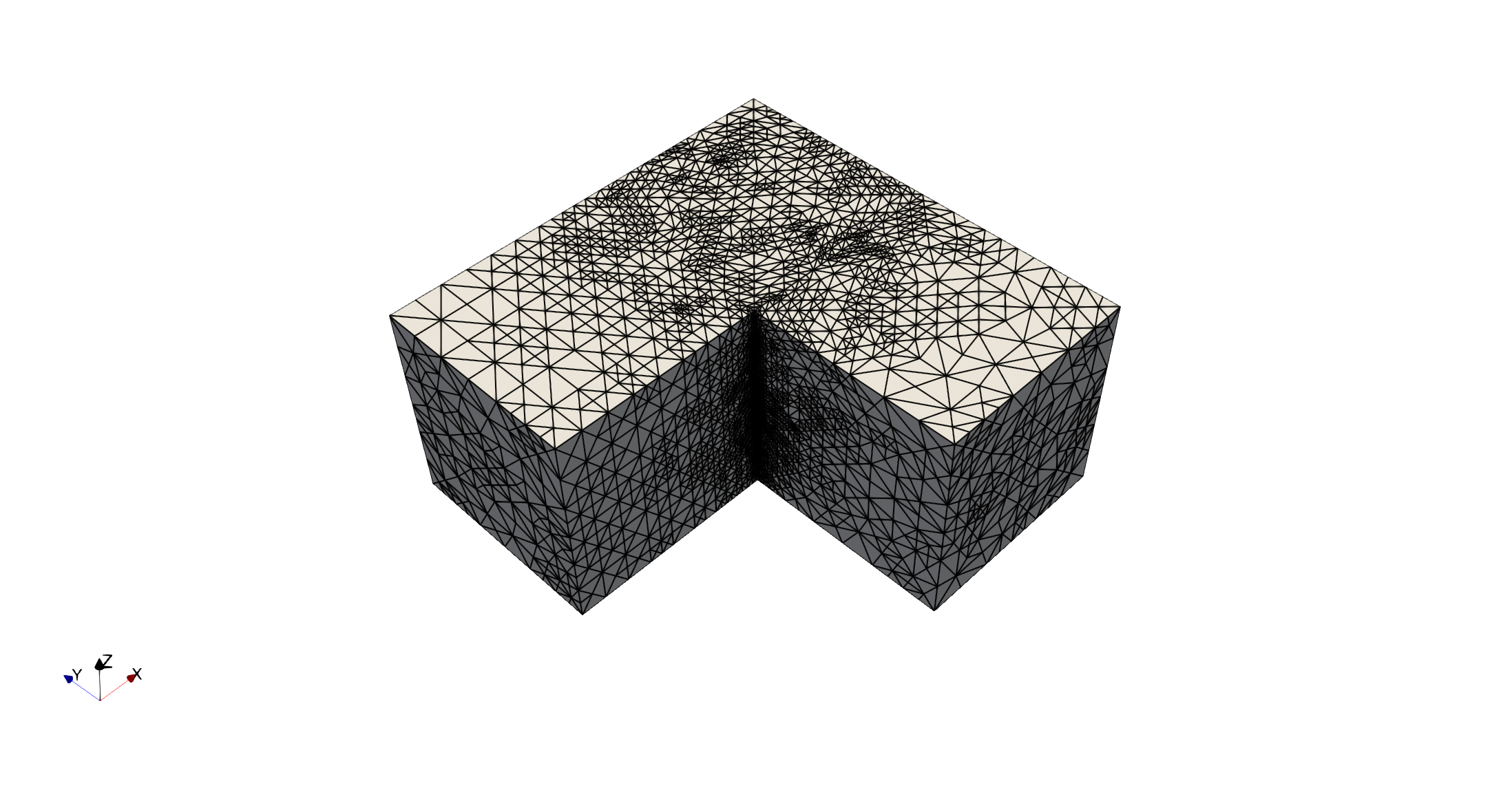}
	\end{minipage}\\
	\begin{minipage}{0.3275\linewidth}\centering
		{\footnotesize $\mathbb{K}^{-1}\vert_{\Omega_D}=10^{1}\mathbb{I}, \texttt{dof}=357946$}\\
		\includegraphics[scale=0.108,trim=19cm 9cm 19cm 4cm,clip]{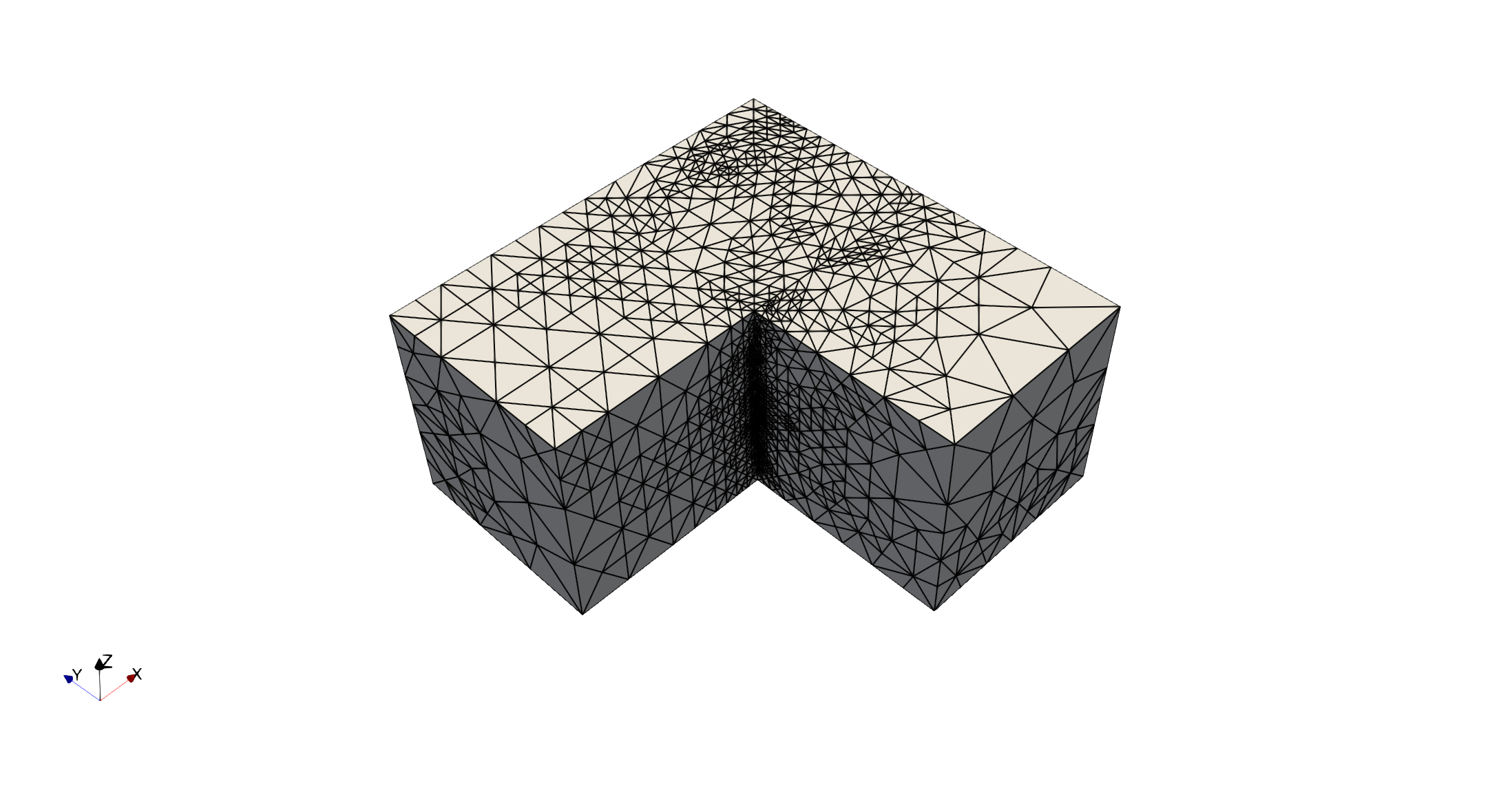}
	\end{minipage}
	\begin{minipage}{0.3275\linewidth}\centering
		{\footnotesize $\mathbb{K}^{-1}\vert_{\Omega_D}=10^{1}\mathbb{I}, \texttt{dof}=652658$}\\
		\includegraphics[scale=0.108,trim=19cm 9cm 19cm 4cm,clip]{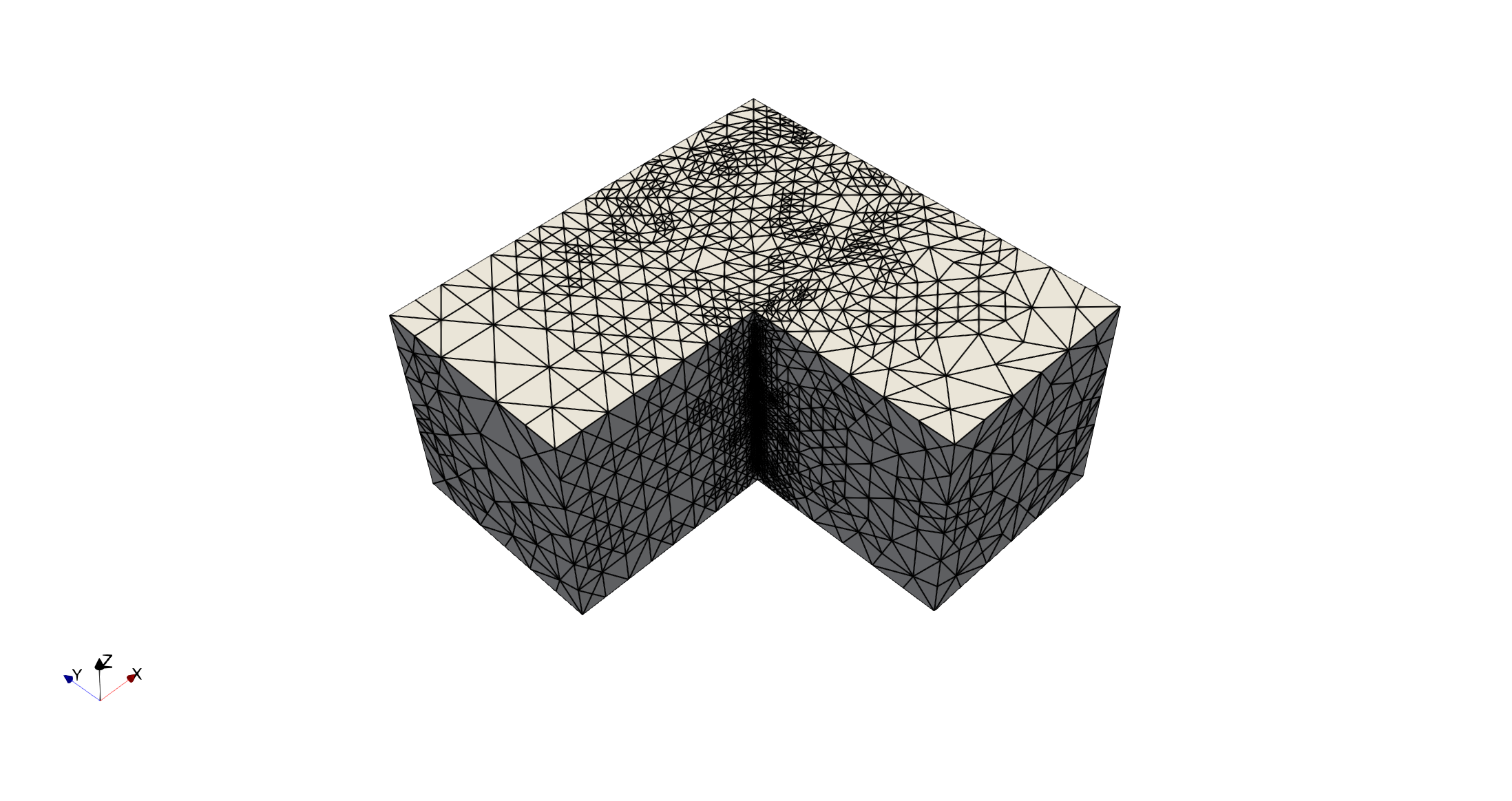}
	\end{minipage}
	\begin{minipage}{0.3275\linewidth}\centering
		{\footnotesize $\mathbb{K}^{-1}\vert_{\Omega_D}=10^{1}\mathbb{I}, \texttt{dof}=1229971$}\\
		\includegraphics[scale=0.108,trim=19cm 9cm 19cm 4cm,clip]{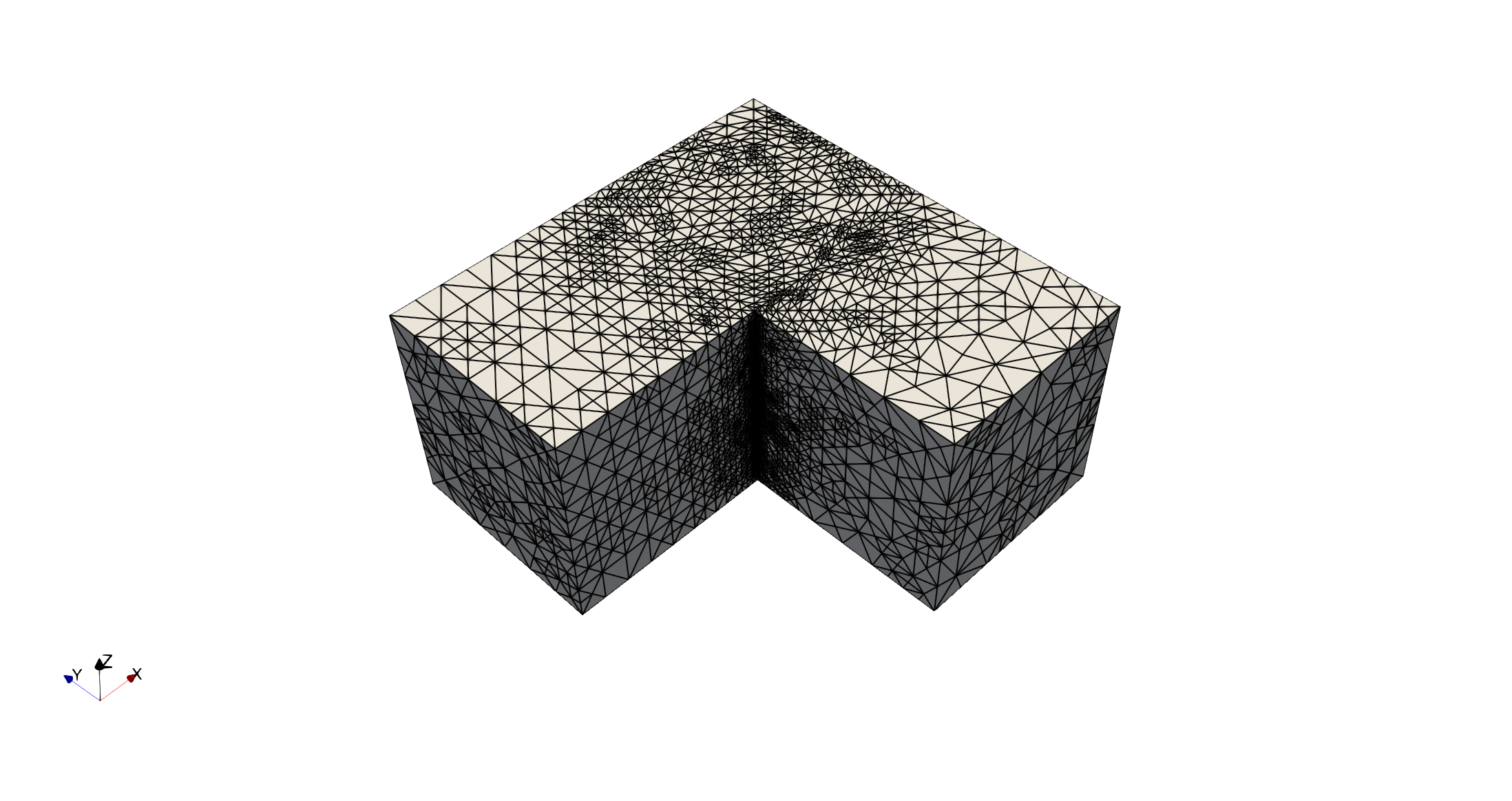}
	\end{minipage}\\
	\begin{minipage}{0.3275\linewidth}\centering
		{\footnotesize $\mathbb{K}^{-1}\vert_{\Omega_D}=10^{2}\mathbb{I}, \texttt{dof}=222581$}\\
		\includegraphics[scale=0.108,trim=19cm 9cm 19cm 4cm,clip]{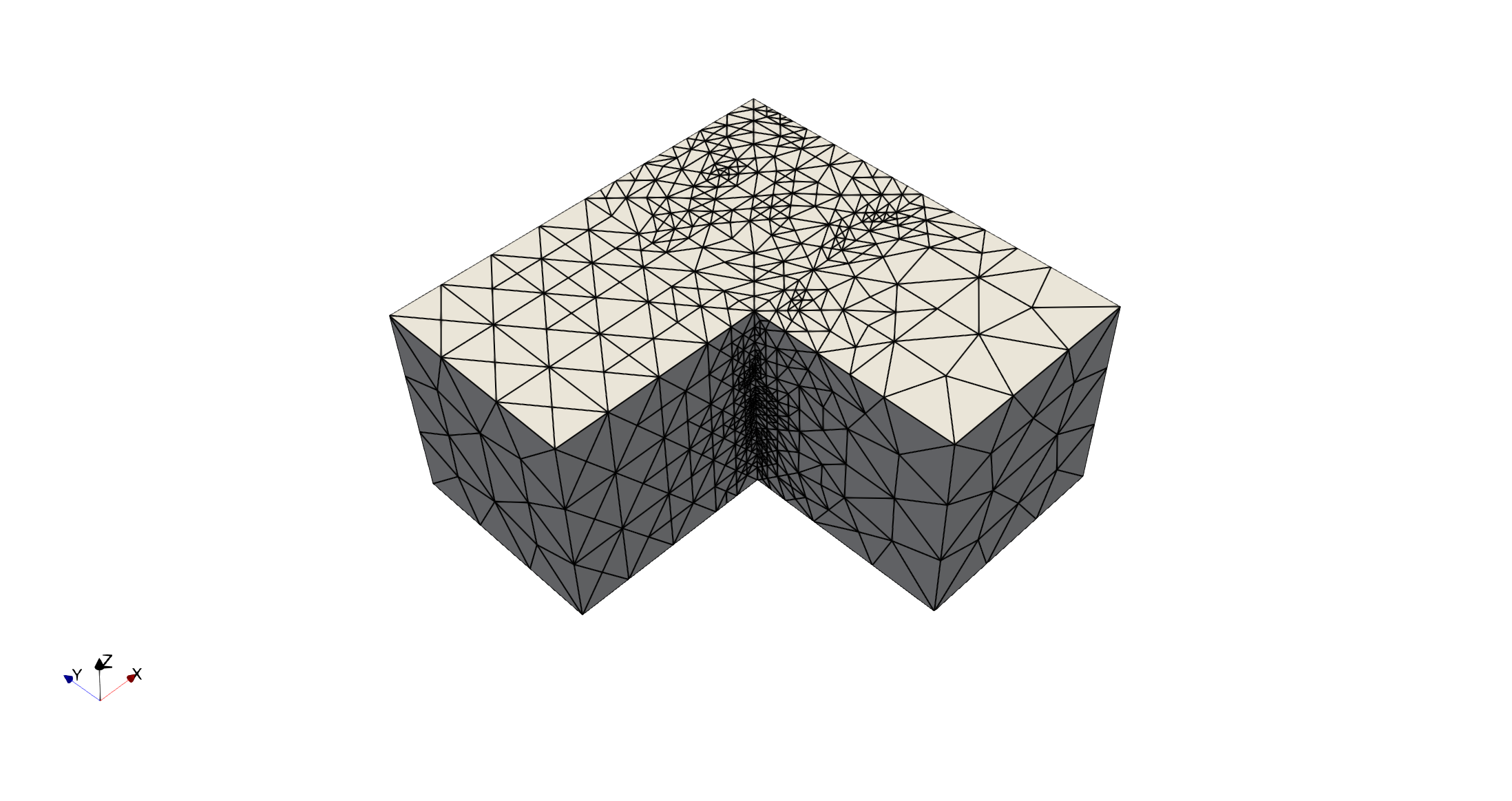}
	\end{minipage}
	\begin{minipage}{0.3275\linewidth}\centering
		{\footnotesize $\mathbb{K}^{-1}\vert_{\Omega_D}=10^{2}\mathbb{I}, \texttt{dof}=532889$}\\
		\includegraphics[scale=0.108,trim=19cm 9cm 19cm 4cm,clip]{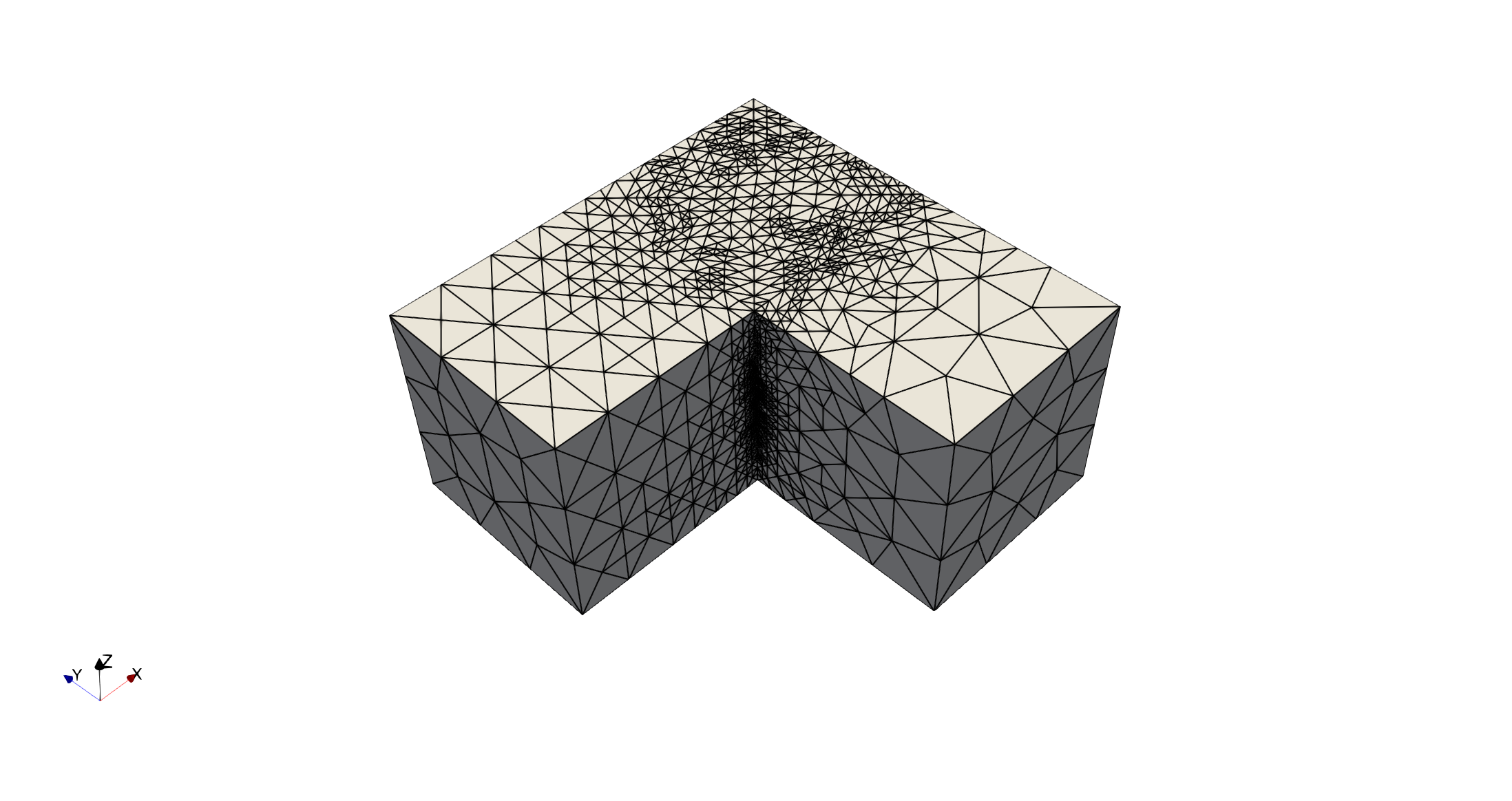}
	\end{minipage}
	\begin{minipage}{0.3275\linewidth}\centering
		{\footnotesize $\mathbb{K}^{-1}\vert_{\Omega_D}=10^{2}\mathbb{I}, \texttt{dof}=805199$}\\
		\includegraphics[scale=0.108,trim=19cm 9cm 19cm 4cm,clip]{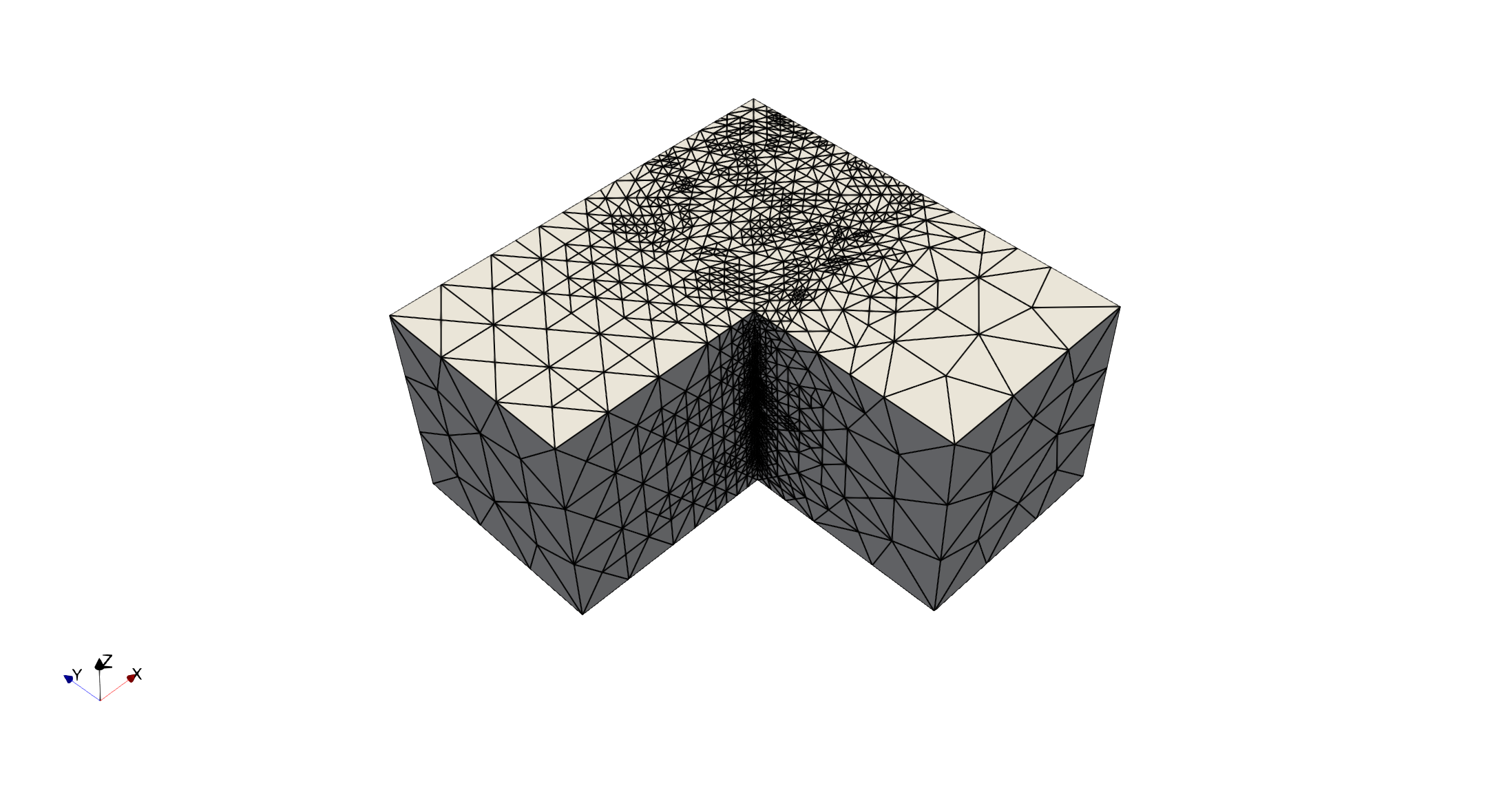}
	\end{minipage}
	\caption{Test \ref{subsec:3D-Lshape}.  Comparison of intermediate adaptive meshes with mini elements for the 3D Lshaped domain with different permeability parameters.}
	\label{fig:3D-lshape-meshes}
\end{figure}

\begin{figure}[!hbt]\centering
	\begin{minipage}{0.3275\linewidth}\centering
		{\footnotesize $\mathbb{K}^{-1}\vert_{\Omega_D}=10^{-8}\mathbb{I}, \texttt{dof}=126494$}\\
		\includegraphics[scale=0.102,trim=18cm 9cm 18cm 4cm,clip]{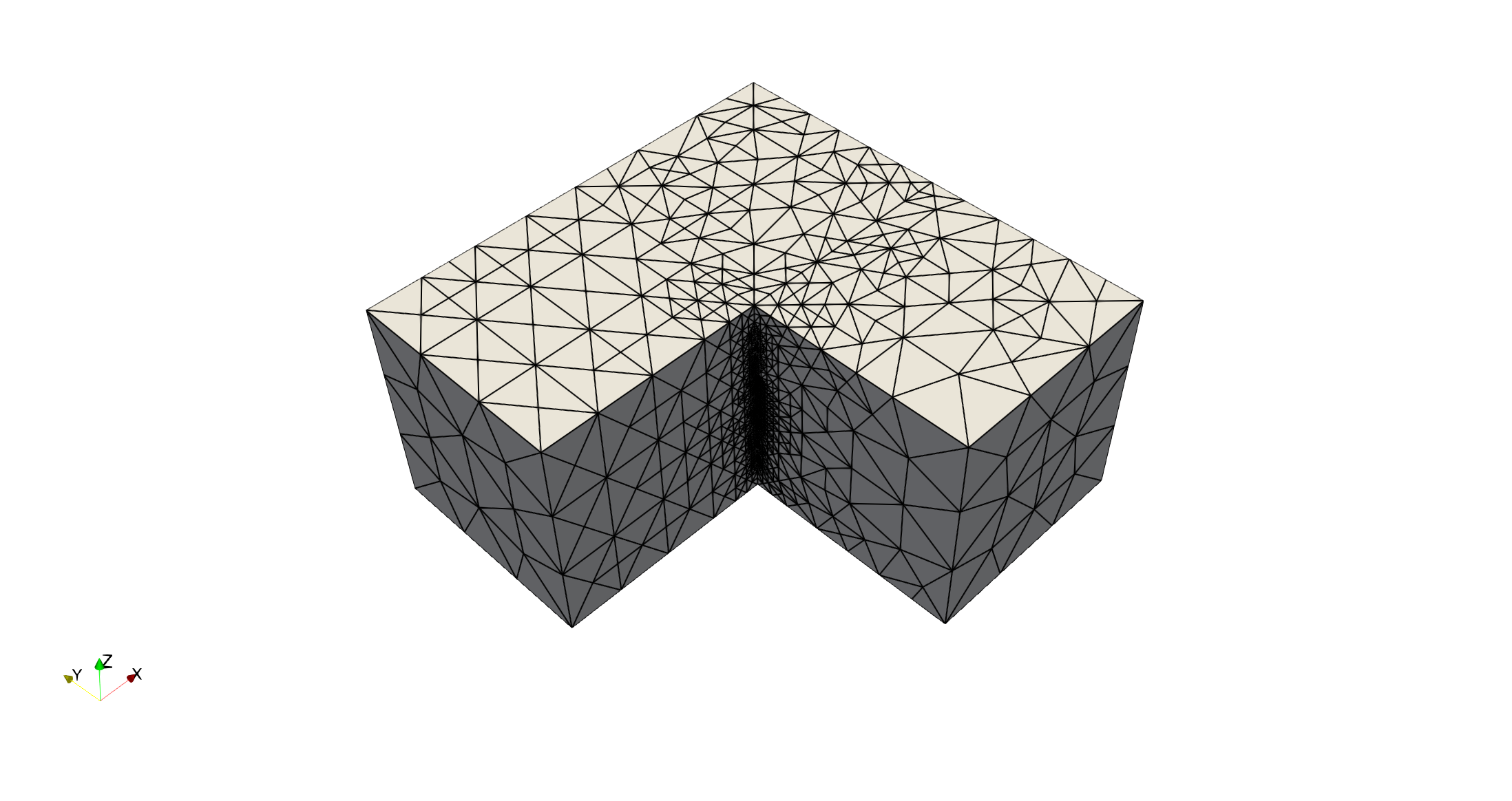}
	\end{minipage}
	\begin{minipage}{0.3275\linewidth}\centering
		{\footnotesize $\mathbb{K}^{-1}\vert_{\Omega_D}=10^{-8}\mathbb{I}, \texttt{dof}=322426$}\\
		\includegraphics[scale=0.102,trim=18cm 9cm 18cm 4cm,clip]{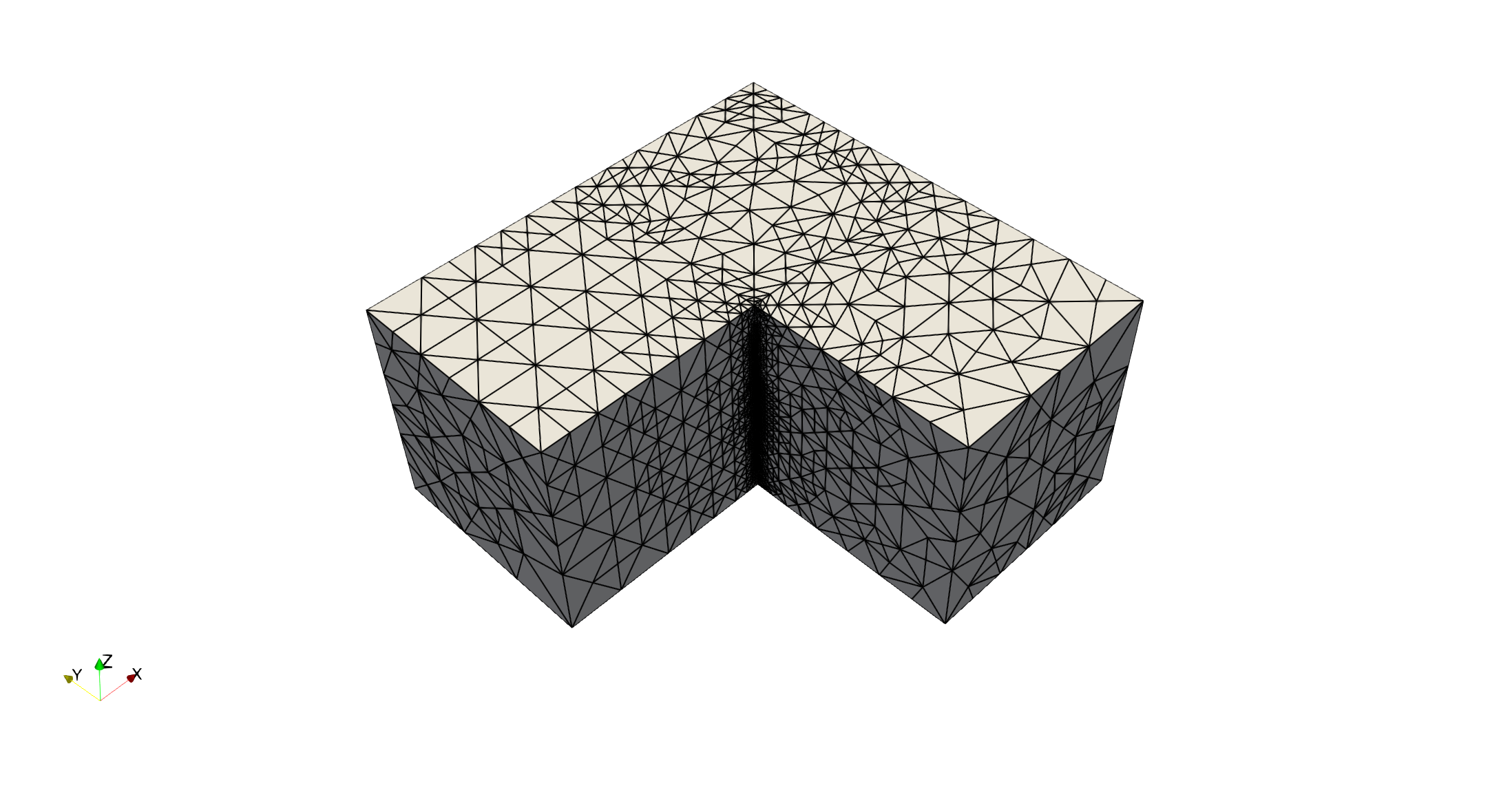}
	\end{minipage}
	\begin{minipage}{0.3275\linewidth}\centering
		{\footnotesize $\mathbb{K}^{-1}\vert_{\Omega_D}=10^{-8}\mathbb{I}, \texttt{dof}=533586$}\\
		\includegraphics[scale=0.102,trim=18cm 9cm 18cm 4cm,clip]{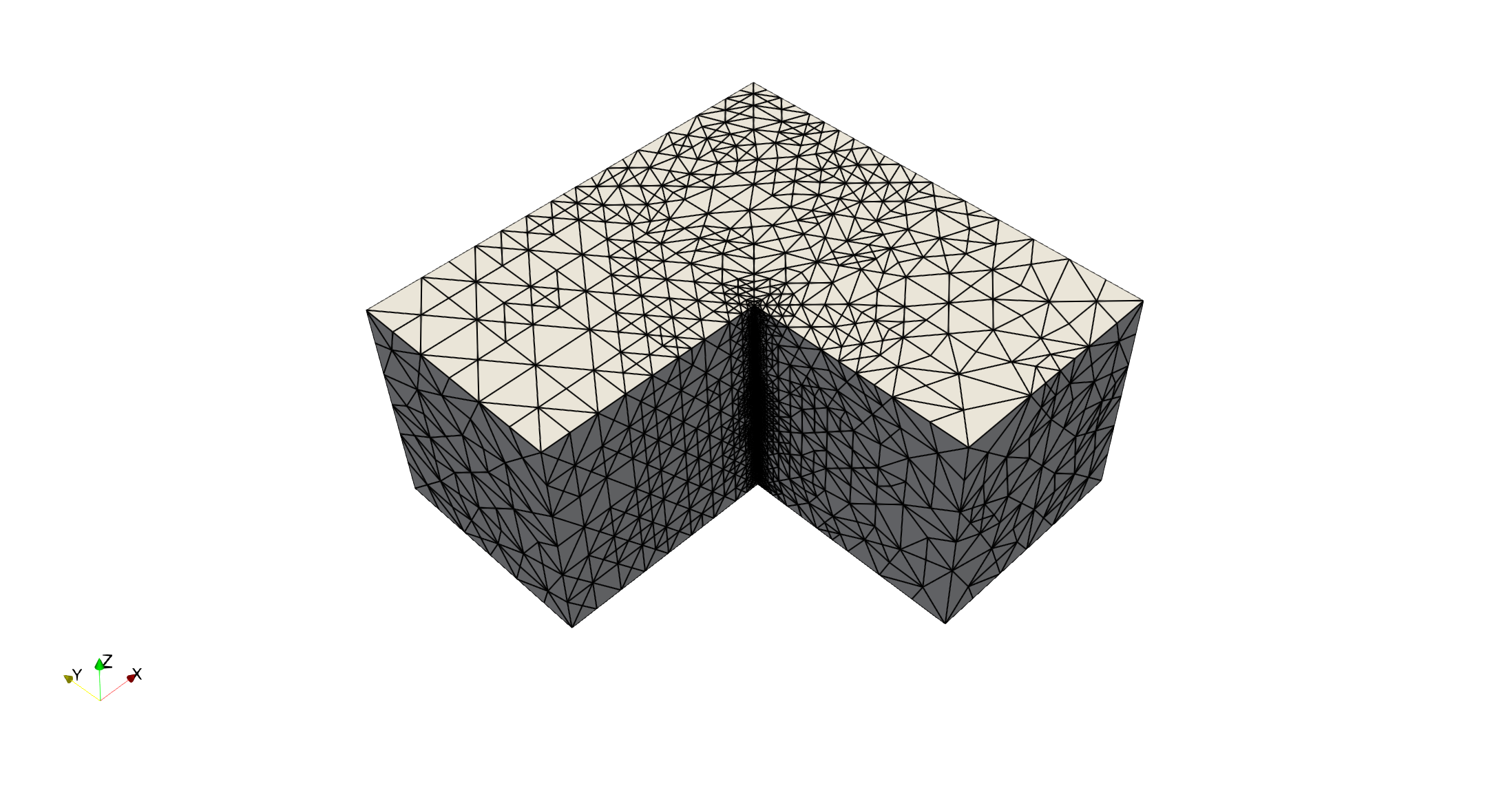}
	\end{minipage}\\
	\begin{minipage}{0.3275\linewidth}\centering
		{\footnotesize $\mathbb{K}^{-1}\vert_{\Omega_D}=10^{1}\mathbb{I}, \texttt{dof}=126479$}\\
		\includegraphics[scale=0.102,trim=18cm 9cm 18cm 4cm,clip]{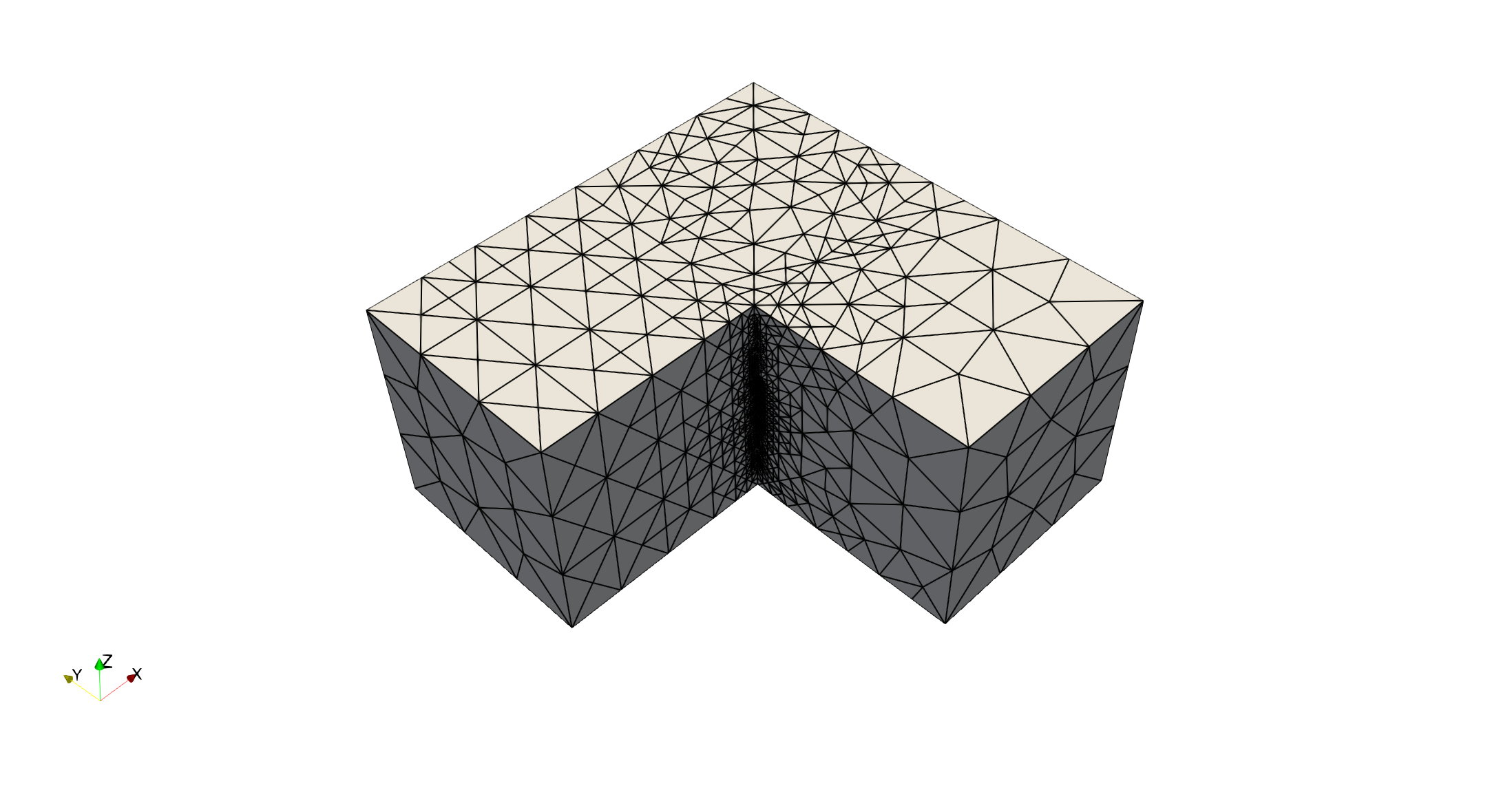}
	\end{minipage}
	\begin{minipage}{0.3275\linewidth}\centering
		{\footnotesize $\mathbb{K}^{-1}\vert_{\Omega_D}=10^{1}\mathbb{I}, \texttt{dof}=325259$}\\
		\includegraphics[scale=0.102,trim=18cm 9cm 18cm 4cm,clip]{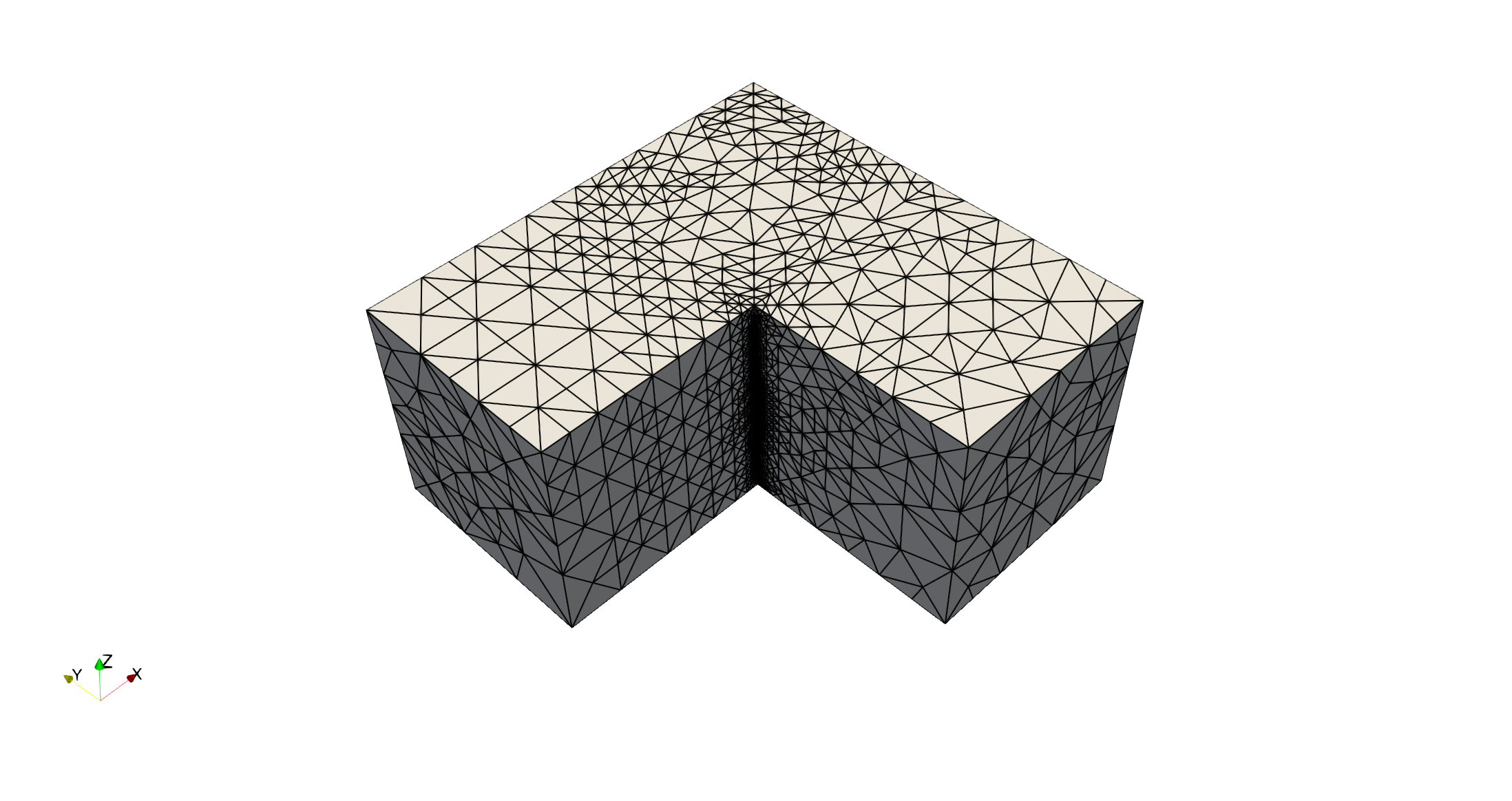}
	\end{minipage}
	\begin{minipage}{0.3275\linewidth}\centering
		{\footnotesize $\mathbb{K}^{-1}\vert_{\Omega_D}=10^{1}\mathbb{I}, \texttt{dof}=539384$}\\
		\includegraphics[scale=0.102,trim=18cm 9cm 18cm 4cm,clip]{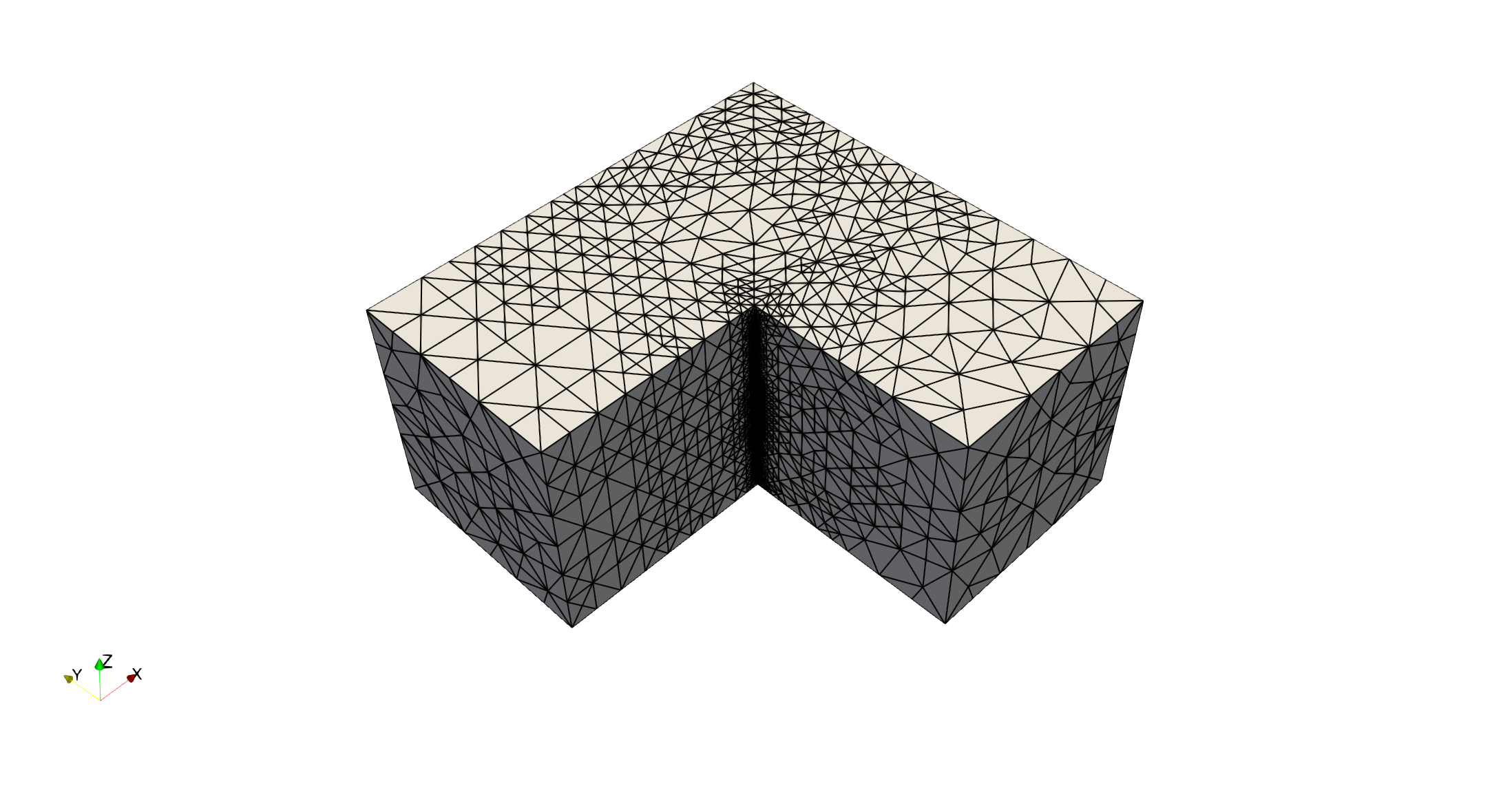}
	\end{minipage}\\
	\begin{minipage}{0.3275\linewidth}\centering
		{\footnotesize $\mathbb{K}^{-1}\vert_{\Omega_D}=10^{2}\mathbb{I}, \texttt{dof}=169884$}\\
		\includegraphics[scale=0.102,trim=18cm 9cm 18cm 4cm,clip]{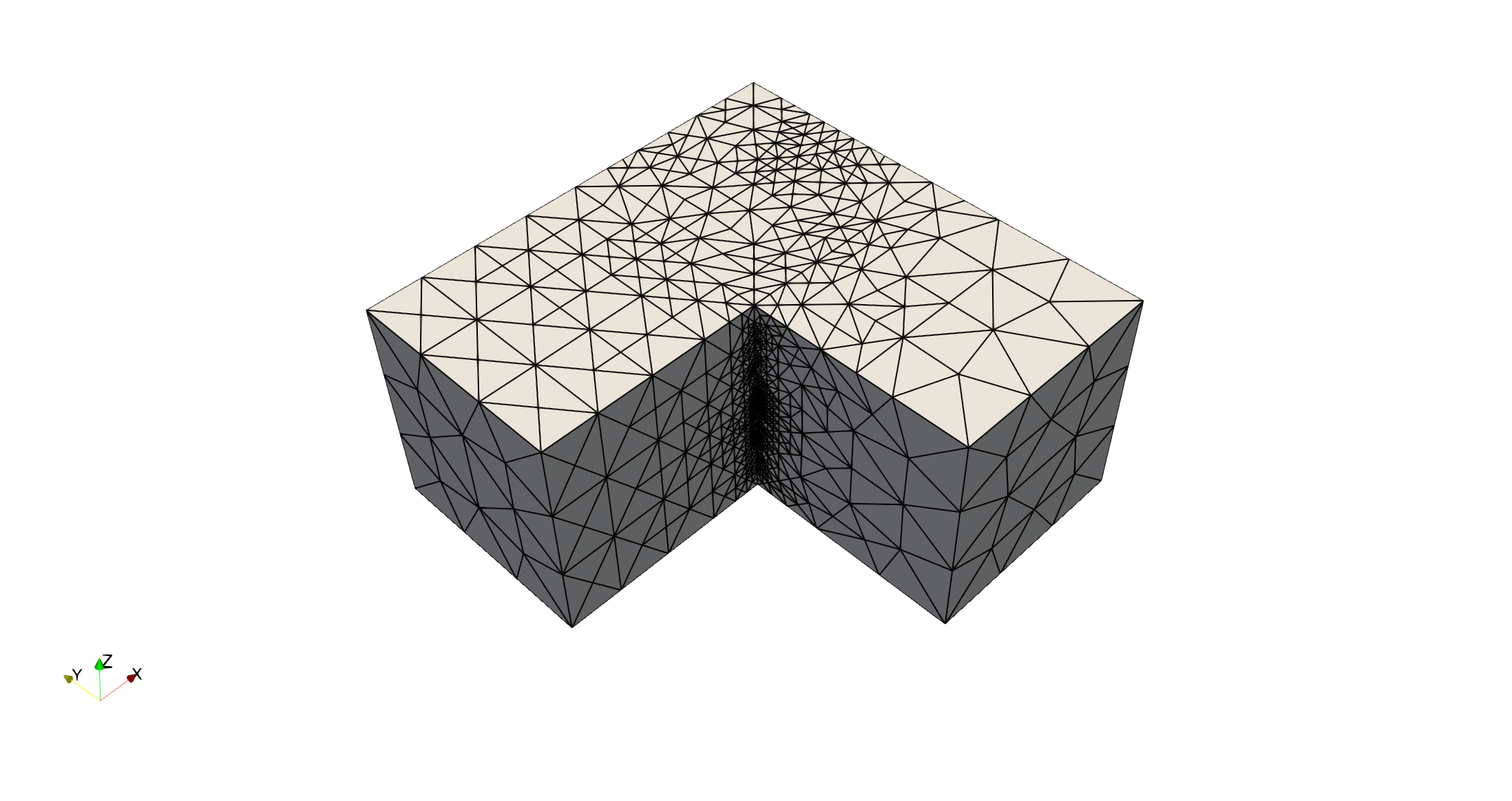}
	\end{minipage}
	\begin{minipage}{0.3275\linewidth}\centering
		{\footnotesize $\mathbb{K}^{-1}\vert_{\Omega_D}=10^{2}\mathbb{I}, \texttt{dof}=271239$}\\
		\includegraphics[scale=0.102,trim=18cm 9cm 18cm 4cm,clip]{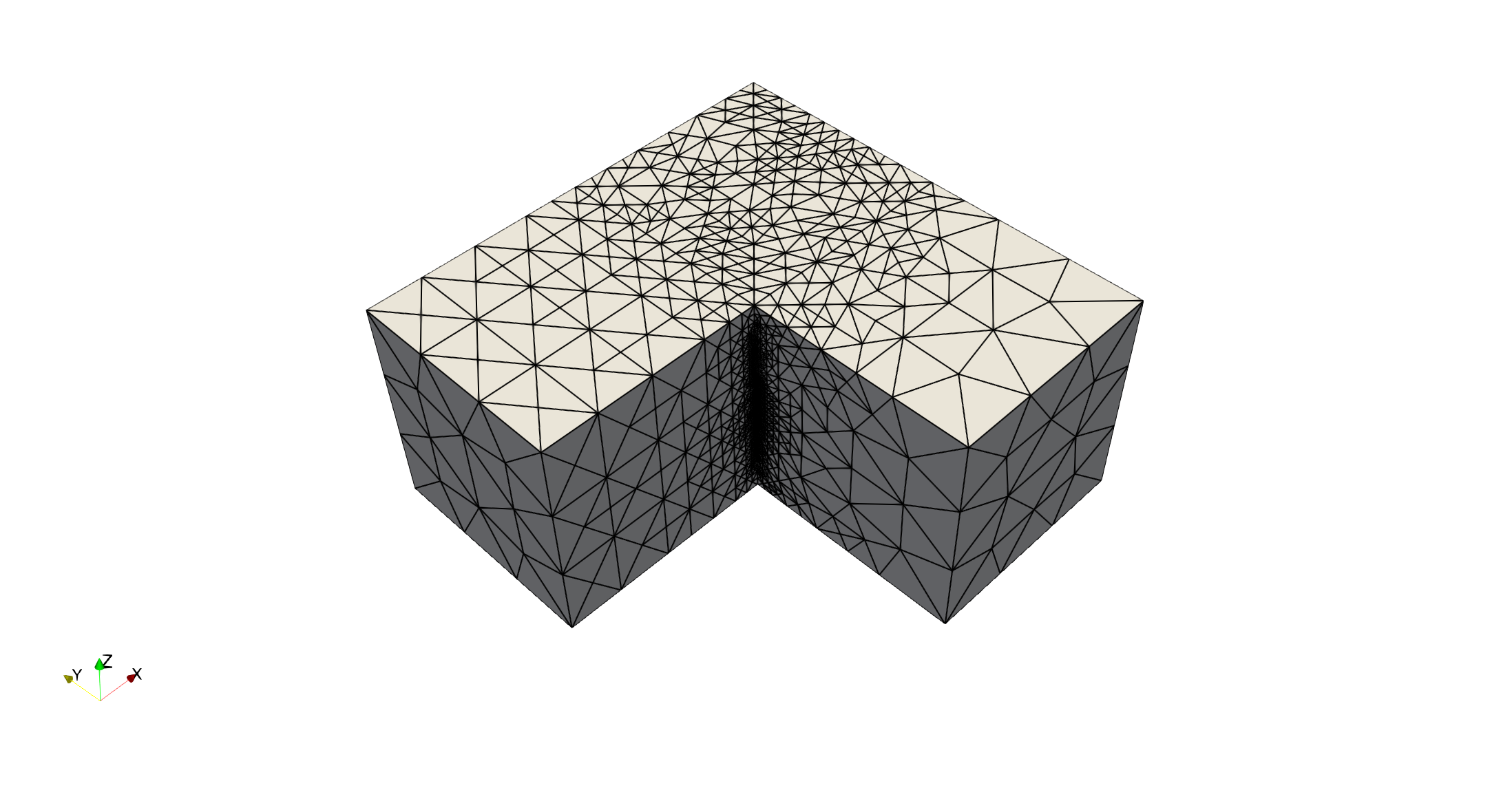}
	\end{minipage}
	\begin{minipage}{0.3275\linewidth}\centering
		{\footnotesize $\mathbb{K}^{-1}\vert_{\Omega_D}=10^{2}\mathbb{I}, \texttt{dof}=598861$}\\
		\includegraphics[scale=0.102,trim=18cm 9cm 18cm 4cm,clip]{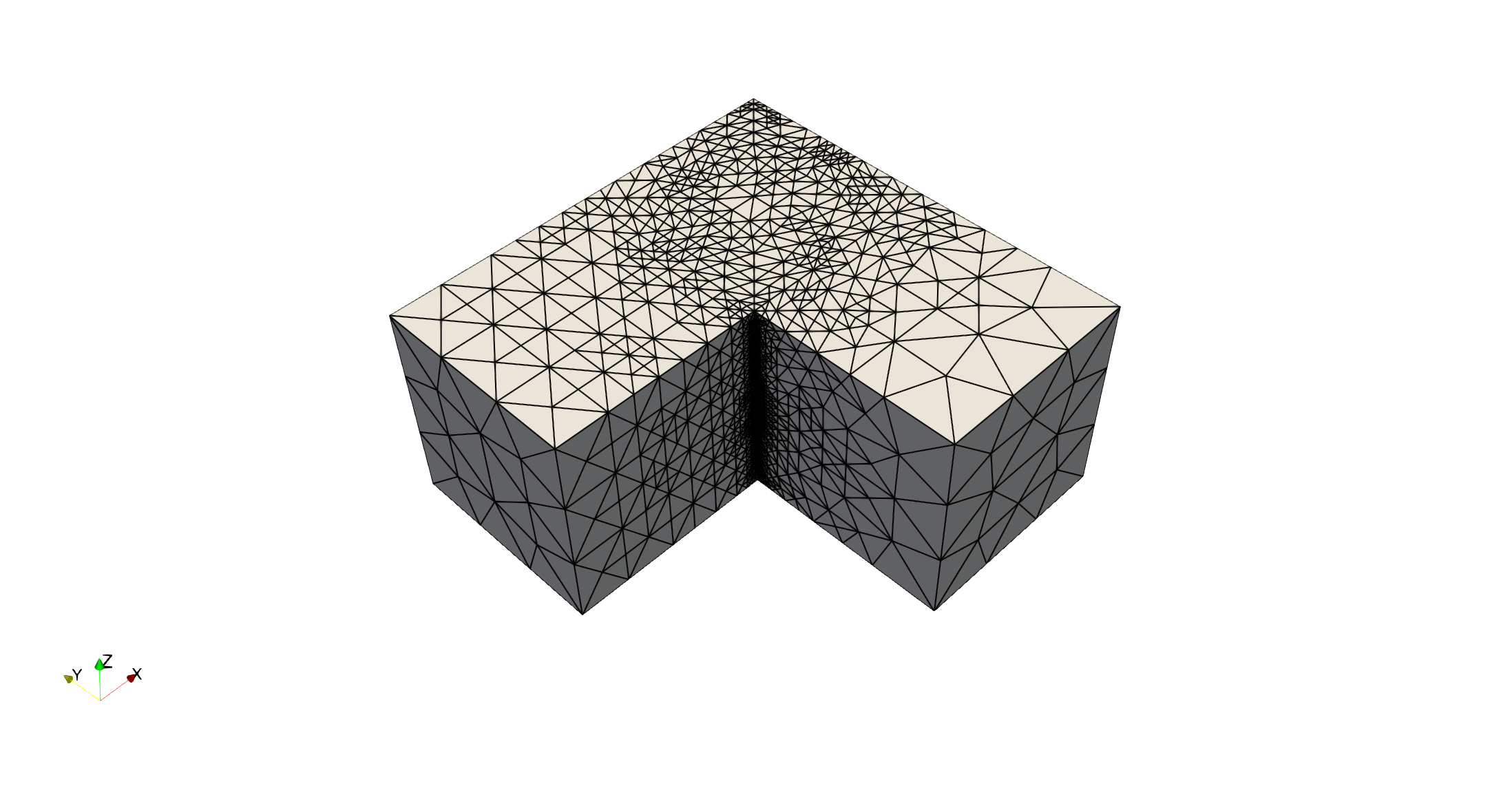}
	\end{minipage}
	\caption{Test \ref{subsec:3D-Lshape}.  Comparison of intermediate adaptive meshes with Taylor-Hood elements for the 3D Lshaped domain with different permeability parameters.}
	\label{fig:3D-lshape-meshes_TH}
\end{figure}

\begin{figure}[!hbt]\centering
	\begin{minipage}{0.59\linewidth}\centering
		\includegraphics[scale=0.37,trim=0cm 0cm 2cm 2cm,clip]{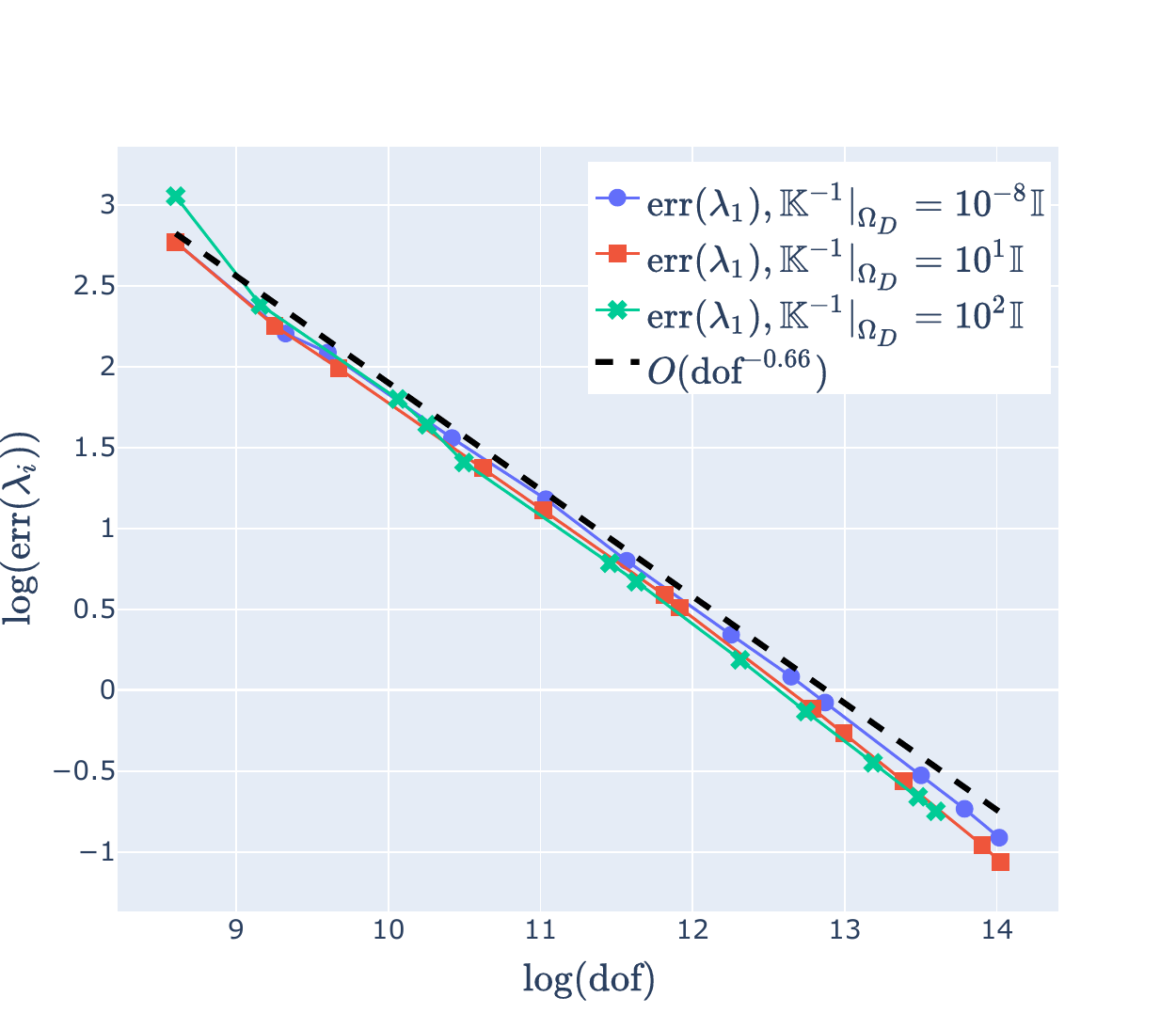}
	\end{minipage}
	\begin{minipage}{0.40\linewidth}\centering
		\includegraphics[scale=0.37,trim=0cm 0cm 2cm 2cm,clip]{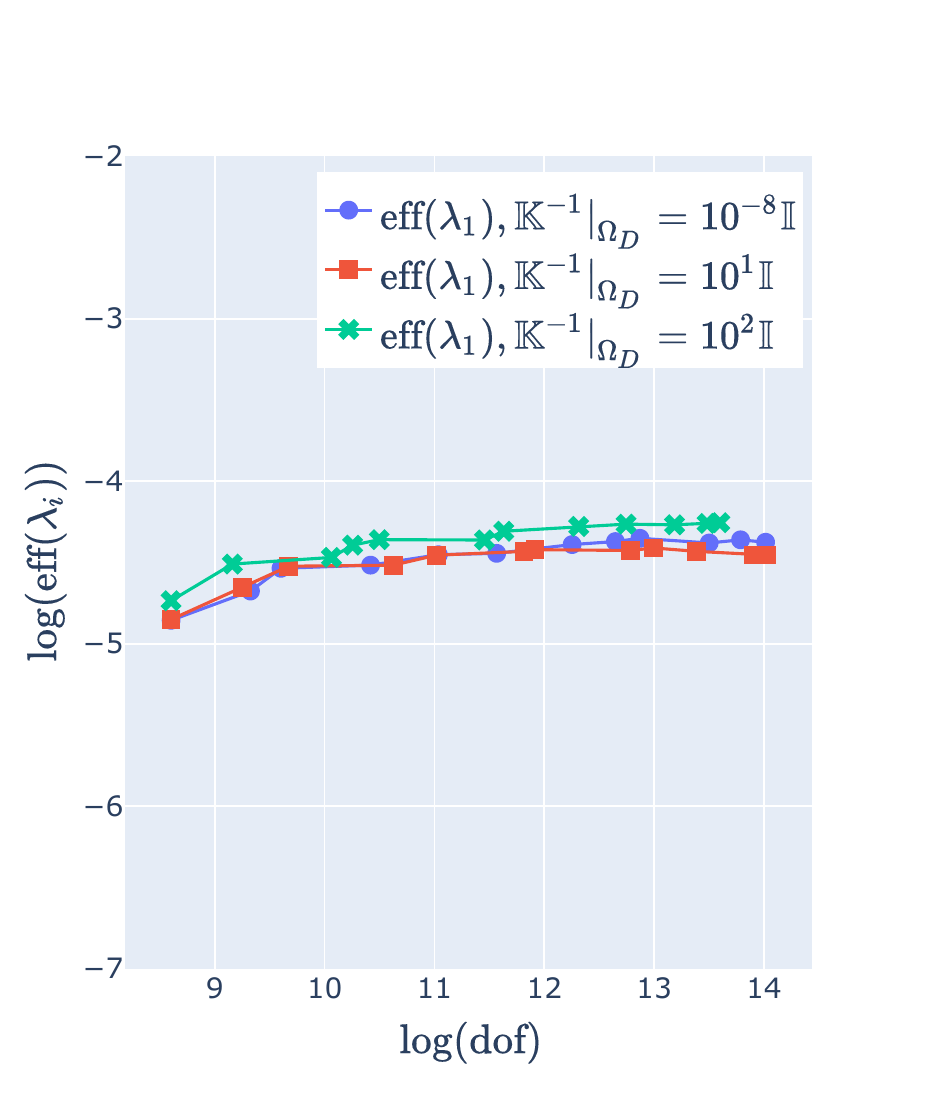}
	\end{minipage}\\
	\begin{minipage}{0.59\linewidth}\centering
		\includegraphics[scale=0.37,trim=0cm 0cm 2cm 2cm,clip]{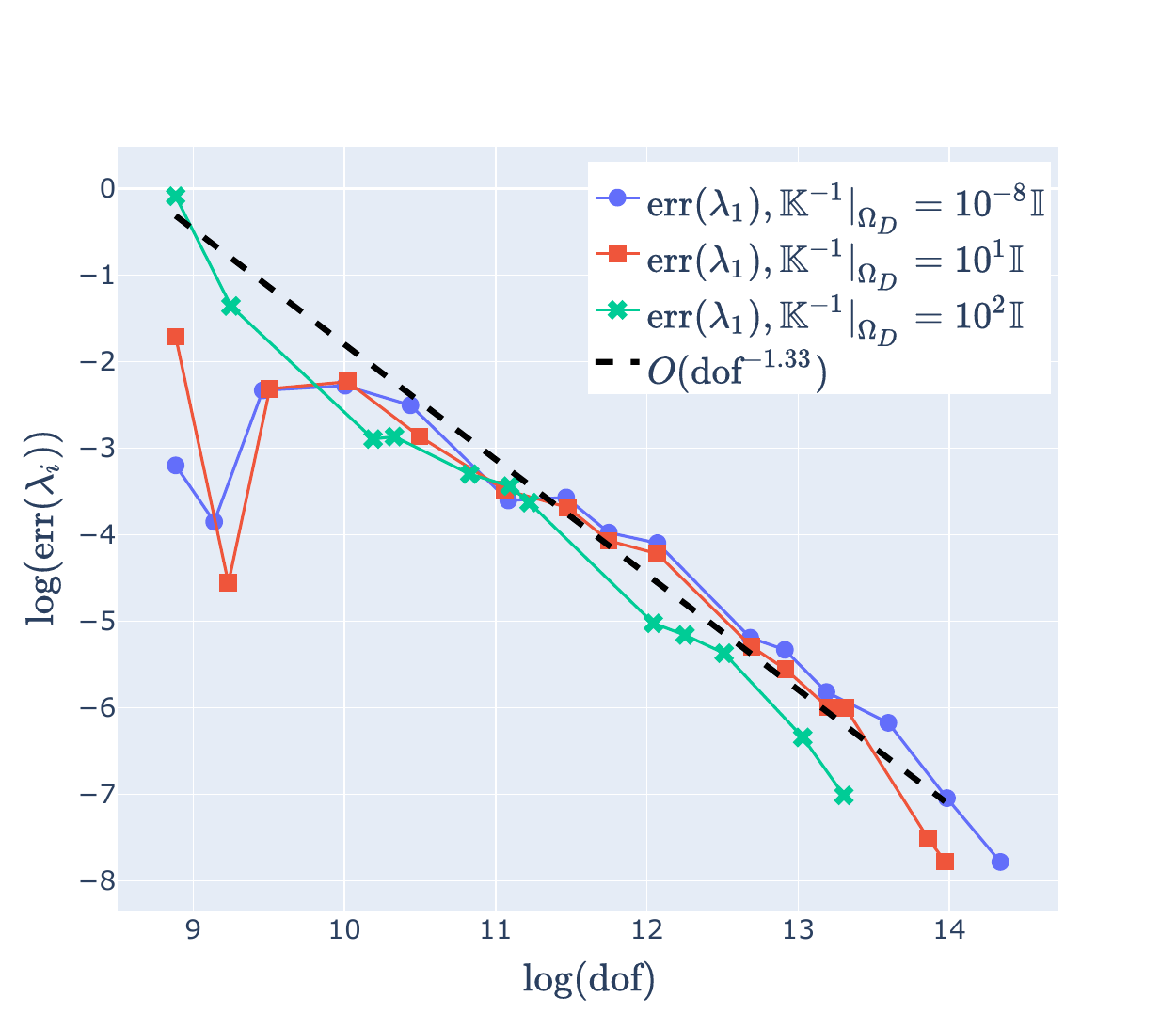}
	\end{minipage}
	\begin{minipage}{0.40\linewidth}\centering
		\includegraphics[scale=0.37,trim=0cm 0cm 2cm 2cm,clip]{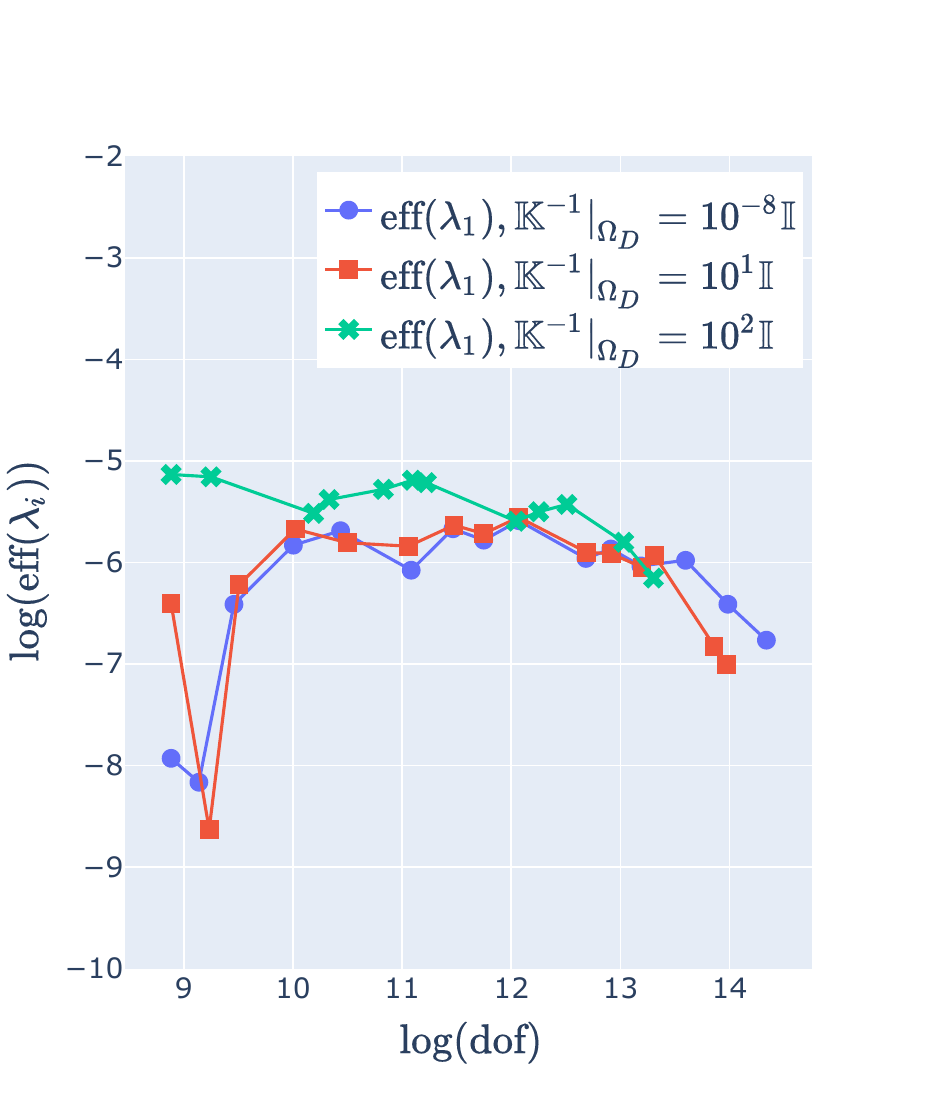}
	\end{minipage}
	\caption{Test \ref{subsec:3D-Lshape}. Error history for the uniform and adaptive refinements for the first eigenvalue on different permeability configurations (left) together with their corresponding effectivity indexes (right). Top row: Mini elements. Bottom row: Taylor-Hood elements.}
	\label{fig:error-eff-lshape3d}
\end{figure}
\section{Conclusions}
In this work, we have formulated an eigenvalue problem for fluid flow over a domain with potentially heterogeneous permeability. The model formulation emphasizes on the variability of the permeability parameter $\mathbb{K}$ across the domain. This approach has allowed us to show that an eigenvalue problem based on the Brinkman equations can yield valuable insights into the behavior of confined or free fluids within heterogeneous media. While this study does not address all possible configurations and limitations related to the choice of permeability in the spectral analysis, the numerical experiments consistently indicate that fluid flow tends to concentrate within Stokes regions when permeability-limited zones are present, suggesting a natural energy minimization tendency within the system. This observation invites further investigation, as this work represents the first step in that direction. Future studies will expand on this framework by considering aspects such as variable viscosity and the  analysis of the coupled Stokes-Darcy eigenvalue problem.

\subsection*{Availability of Data and Materials} The datasets generated and/or analyzed during the current study are available from the corresponding author on reasonable request.

\subsection*{Conflict of interest} The authors have no competing interests to declare that are relevant to the content of this article. Data sharing not applicable to this article as no datasets were generated or analysed during the current study.

\bibliographystyle{siamplain}
\bibliography{LRV_brinkmannEV}
\end{document}